\documentclass[11pt]{amsart}

\usepackage{amsmath, amssymb, amsthm, amsfonts, graphicx, pinlabel}

\input{Sections/Preamble}

\begin{document}

\title[A DGA for Legendrian knots from generating families]{A combinatorial DGA for Legendrian knots from generating families}

\author{Michael B. Henry}
\address{Siena College, Loudonville, NY 12211}
\email{mbhenry@siena.edu}

\author{Dan Rutherford}
\address{University of Arkansas, Fayetteville, AR  72701}
\email{drruther@uark.edu}

\begin{abstract}

For a Legendrian knot, $L \subset \R^3$, with a chosen Morse complex sequence (MCS) we construct a differential graded algebra (DGA) whose differential counts ``chord paths'' in the front projection of $L$. The definition of the DGA is motivated by considering Morse-theoretic data from generating families. In particular, when the MCS arises from a generating family, $F$, we give a geometric interpretation of our chord paths as certain broken gradient trajectories which we call ``gradient staircases''. Given two equivalent MCS's we prove the corresponding linearized complexes of the DGA are isomorphic. If the MCS has a standard form, then we show that our DGA agrees with the Chekanov-Eliashberg DGA after changing coordinates by an augmentation. 

\end{abstract}

\maketitle

\section{Introduction}
\mylabel{ch:intro}

In this article we associate a differential graded algebra, $(\mathcal{A}_\mathcal{C}, d)$, to a Legendrian knot, $L$, in standard contact $\R^3$.  Our approach is based on a conjectural extension of previously studied invariants which require as input a generating family, $F$, for $L$.  In the conjectured extension a DGA differential should be defined by counting gradient flow trees for difference functions arising from $F$. However, rather than working with an honest generating family, $F$, which is a one parameter family of functions, the definition of $(\mathcal{A}_\mathcal{C}, d)$ makes use of a Morse complex sequence, $\mathcal{C}$, which is an algebraic/combinatorial structure assigned to the front diagram of $L$. 

Morse complex sequences (abbreviated MCS) were introduced by Pushkar as a combinatorial substitute for generating families. They are designed to abstract the bifurcation behavior that generically occurs in the Morse complexes of a one parameter family of functions and metrics. In particular, after choosing an appropriate metric, $g$, a linear at infinity generating family $F$ for $L$ gives rise to an MCS which we denote as $\mathcal{C}(F, g)$. There is a notion of equivalence for Morse complex sequences which is based on a corresponding notion for generating families (see \cite{Jordan2006} and \cite{Hatcher1973}) and accounts for the ambiguity introduced by the choice of a metric. 

The DGA associated with an MCS is defined as the tensor algebra $\mathcal{A}_\mathcal{C}= \oplus_{n=0}^\infty A_\mathcal{C}^{\otimes n}$ of a $\zz_2$ vector space $A_C$ spanned by crossings and right cusps of the Legendrian knot $L$.  A grading on $\mathcal{A}_\mathcal{C}$ arises from a grading on $A_C$.  The differential $d$ is determined by the Leibniz rule and its restriction to $A_\mathcal{C}$ which is a sum of maps $d_n: A_\mathcal{C} \rightarrow A_\mathcal{C}^{\otimes n}$, $n\geq 1$.    Since there is no $d_0$ term, $d_1: A_\mathcal{C} \rightarrow A_\mathcal{C}$ satisfies $(d_1)^2 =0$ and we refer to $(A_\mathcal{C}, d_1)$ as the {\it linearized complex} of $\mathcal{A}_\mathcal{C}$.  For readers familiar with $A_\infty$-algebras, we note that the maps dual to the $d_n$ provide the dual of $A_\mathcal{C}$ with the structure  of a finite-dimensional $A_\infty$-algebra; see \cite{Civan2011}.

The $d_n$ are defined by counting sequences of vertical markers proceeding from right to left along the front diagram of $L$ which we call ``chord paths''. The MCS, $\mathcal{C}$, imposes restrictions on how successive markers may appear in a chord path. In the case of an MSC, $\mathcal{C} = \mathcal{C}(F, g)$, arising from a generating family, we show that our chord paths may be interpreted geometrically as certain broken gradient trajectories which we call ``gradient staircases''. Conjecturally, gradient staircases form the limiting objects of more usual gradient trajectories as the metric is rescaled towards a degenerate limit. 

Two of the main results of this paper are the following:

\begin{reptheorem}{thm:d2is0}
If $\mathcal{C}$ is an MCS for a nearly plat position Legendrian knot $L$, then 
$(\mathcal{A}_\mathcal{C}, d)$  
is a DGA. That is, $d$ has degree $-1$, satisfies the Liebniz rule, and has $d^2 = 0$. 
\end{reptheorem}

\begin{reptheorem}{thm:linear-iso}
Let $\sMCS^{+}$ and $\sMCS^{-}$ be MCSs for a nearly plat position Legendrian knot $L$.  If  $\sMCS^{+}$ and $\sMCS^{-}$ are equivalent, then the linearized complexes $(A_{\sMCS^{+}}, d^{+}_1)$ and $(A_{\sMCS^{-}}, d^{-}_1)$ are isomorphic.
\end{reptheorem}

Differential graded algebras were famously associated to Legendrian knots by Chekanov in \cite{Chekanov2002a}. He introduced an invariant $(\mathcal{A}(L), \partial)$ which has come to be known as the Chekanov-Eliashberg DGA since it provides a combinatorial approach to the Legendrian contact homology of \cite{Eliashberg2000} which is defined via counting holomorphic disks; see \cite{Etnyre2002} for the translation between these two invariants.  Recently, a close connection--not yet fully understood--has emerged between generating families and augmentations of the Chekanov-Eliashberg DGA; see \cite{Chekanov2005}, \cite{Fuchs2003}, \cite{Fuchs2004}, \cite{Fuchs2008}, \cite{Jordan2006}, \cite{L.Traynor2004},  and \cite{Sabloff2005}. One motivation for the current work was to try to further strengthen the ties between generating families and augmentations.

An augmentation is a DGA homomorphism $\epsilon: \mathcal{A}(L) \rightarrow \zz_2$.  Given an augmentation, $\epsilon$, a new differential, $\df^\saug$, arises on $\mathcal{A}(L)$ from an associated change of coordinates.  Our main result relating our construction to the Chekanov-Eliashberg algebra involves Morse complex sequences of a special form which exist within each MCS equivalence class.

\begin{reptheorem}{thm:A-form-equivalence}
Given an $A$-form MCS $\sMCS$ with corresponding augmentation $\saug$, there is a grading-preserving bijection between the generators of $\salg_{\sMCS}$ and $\salg(\sfront)$ such that $d = \df^{\saug}$, and hence $(\salg_{\sMCS}, d) = (\salg(\sfront), \df^{\saug})$.
\end{reptheorem}

Theorems~\ref{thm:linear-iso}~and~\ref{thm:A-form-equivalence}, along with the fact that every MCS is equivalent to an $A$-form MCS, give the following corollary relating the linearized homology of the MCS-DGA with the linearized contact homology of the CE-DGA.

\begin{repcorollary}{cor:linear-homology}
Let $H^{\sMCS}_{*}(\sfront)$ denote the homology groups of $(A_{\sMCS}, d_1)$ for an MCS $\sMCS \in \sFMCS$ and let $LCH_{*}^{\saug}(\sfront)$ denote the homology groups of $(\salg_1(\sfront), \df^{\saug}_1)$ for an augmentation $\saug \in \sAugL$. Then $$\{H^{\sMCS}_{*}(\sfront)\}_{\sMCS \in \sFMCS} = \{LCH_{*}^{\saug}(\sfront)\}_{\saug \in \sAugL}.$$
\end{repcorollary}

\subsection{Outline of the article} 
In Section 2, we recall the relevant background material about generating families and the Chekanov-Eliashberg DGA. Section 3 provides a sketch of the Morse theoretic method for defining a DGA from a generating family $F$, which is based on communications with Josh Sabloff. This serves as a differential topological motivation for the combinatorial DGA $(A_{\mathcal{C}}, d)$. Such a generating family DGA has yet to be rigorously defined, so an effort is made to keep this discussion brief.

In Section 4, the fundamental, yet lengthy, definition of a Morse complex sequence is provided, and the construction of an MCS from a generating family is given. Section 5 contains the main work of the paper. First, the DGA, $(A_{\mathcal{C}}, d)$, is defined via chord paths. Then, Theorems~\ref{thm:d2is0}~and~~\ref{thm:linear-iso} are proven by a hands-on analysis of the boundary points of relevant ``1-dimensional'' spaces of chord paths. 

In Section 6 we give a geometric explanation for chord paths as gradient staircases; see Proposition~\ref{prop:GSCorr}. In addition, a conjecture about realizing gradient staircases as the limit of gradient trajectories is stated. This provides the link between the combinatorial DGA $(A_{\mathcal{C}}, d)$ and the proposed generating function DGA of Section 3. Furthermore, the conjecture together with Theorems~\ref{thm:d2is0}~and~~\ref{thm:linear-iso} would provide an alternate approach to the main result of \cite{Fuchs2008}.

Finally, in Section 7 we recall the relevant results from \cite{Henry2011} on $A$-form MCSs, and conclude the article with a proof of Theorem~\ref{thm:A-form-equivalence}.

\subsection{Acknowledgments}

The September 2008 workshop conference on Legendrian knots at the American Institute of Mathematics (AIM) led to our current collaboration and sparked many interesting questions concerning generating families and Morse complex sequences. We would like to thank AIM and the workshop attendees Sergei Chmutov, Dmitry Fuchs, Victor Goryunov, Paul Melvin, Josh Sabloff, and Lisa Traynor for many fruitful discussions both during and following the workshop. In a communication to these attendees, Josh Sabloff outlined a program to define products on generating family homology using gradient flow trees. This program was a central motivation for our current work and so we wish to thank Josh again for his guidance. The second author would also like to thank Lenny Ng for many helpful conversations throughout this project. 
\section{Background}
\mylabel{ch:Background}

A \emph{Legendrian knot} $\sLeg$ is a smooth knot in $\rr^3$ whose tangent space sits in the 2-plane distribution determined by the contact structure $\xi_{\text{std}} = \ker(dz-ydx)$. The projection of $\sLeg$ to the $xz$-plane is called the \emph{front projection} and the projection to the $xy$-plane is called the \emph{Lagrangian projection}. We will denote both projections by $\sLeg$ and, when necessary, provide context to prevent ambiguities. 

A Legendrian knot may be isotoped slightly to ensure the singularities in the front projection are transverse crossings and semicubical cusps. We say a front projection is in \emph{plat position} if the left cusps have the same $x$-coordinate, the right cusps have the same $x$-coordinate, and no two crossings have the same $x$-coordinate. A front projection $\sfront$ is \emph{nearly plat} if it is the result of perturbing a front projection in plat position slightly so that the cusps have distinct $x$-coordinates. Every Legendrian knot class admits a plat and, hence, a nearly plat representative and, unless otherwise noted, all Legendrian knots in this article are assumed to be nearly plat.  

The resolution algorithm defined by Ng in \cite{Ng2003} relates front and Lagrangian projections and, hence, provides an important bridge between invariants defined from these projections. The algorithm isotopes a front projection $\sfront$ to a front projection $\sfront'$ so that the Lagrangian projection of $L'$ is topologically similar to $\sfront$. 
Figure~\ref{f:Ng-resolution} describes the isotopy on $\sfront$ near crossings and cusps and the result of this isotopy in the Lagrangian projection. 
The front $\sfront'$ is formed by stretching $\sfront$ in the $x$ direction and modifying the slopes of the strands of $\sfront$. The strands of $\sfront'$ have constant, decreasing slopes from top to bottom except near crossings and right cusps where the slopes of two consecutive strands are exchanged. Any front projection with strands arranged in this manner is said to be in \emph{Ng form}.  
Note that the crossings and right cusps of a front projection in Ng form are in bijection with the crossings of the corresponding Lagrangian projection.

\begin{figure}[t]
\centering
\includegraphics[scale=.5]{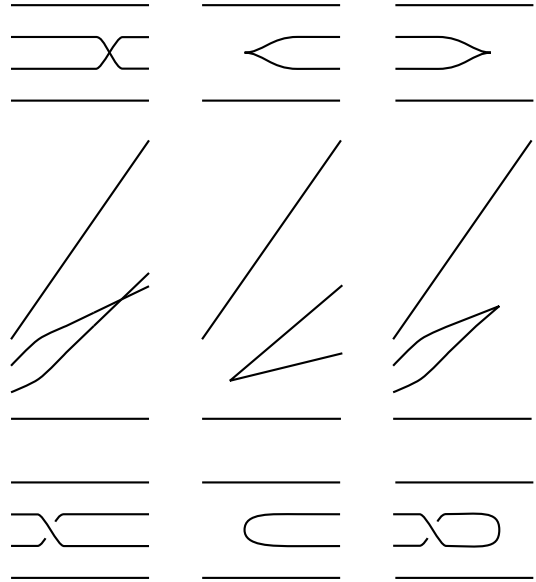}
\caption[Stretching a left cusp.]{The resolution algorithm near a crossing, left cusp, and right cusp. From top to bottom, the three rows show the original front projection, stretched front projection and resulting Lagrangian projection.}
\label{f:Ng-resolution}
\end{figure}

The Legendrian contact homology of Eliashberg and Hofer \cite{Eliashberg2000} provides a differential graded algebra that has given rise to many Legendrian invariants. In \cite{Chekanov2002a}, Chekanov formulates the differential graded algebra combinatorially and extracts a Legendrian invariant capable of distinguishing Legendrian knot classes not distinguished by the classical invariants. The resulting differential graded algebra is known as the \emph{Chekanov-Eliashberg differential graded algebra}, abbreviated CE-DGA and denoted $(\salg(\sLeg), \df)$. We will use Ng's formulation of $(\salg(\sLeg), \df)$ given in terms of the front projection; see \cite{Ng2003}. We refer the reader to any of \cite{Chekanov2002a,Chekanov2002,Etnyre2002,Sabloff2005} for a more complete introduction to the CE-DGA. 

The algebra $\salg(\sLeg)$ is the unital tensor algebra of the $\zz_2$ vector space $A(L)$ freely generated by labels $Q = \{q_1, \hdots, q_n \}$ assigned to the crossings and right cusps of the front projection $\sfront$. In this article, we will consider Legendrian knots admitting MCSs. Such Legendrian knots admit graded normal rulings, hence, we may assume the classical invariant known as the rotation number, $r(L)$, is $0$; see \cite{Henry2011} and \cite{Sabloff2005}.  Thus, $(\salg(\sfront), \df)$ has a $\zz$-grading which may be defined in terms of a Maslov potential on the front projection $\sfront$. 

\begin{figure}[t]
\labellist
\small\hair 2pt
\pinlabel {$i$} [tl] at 5 11
\pinlabel {$i+1$} [br] at 13 27
\pinlabel {$i$} [tl] at 102 11
\pinlabel {$i+1$} [bl] at 102 27
\endlabellist
\centering
\includegraphics[scale=1]{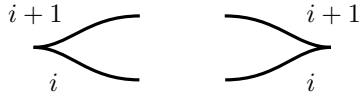}
\caption[]{The Maslov relation near cusps.}
\label{f:Maslov}
\end{figure}

\begin{definition}
	\mylabel{defn:Maslov}	A \emph{Maslov potential} on $\sfront$ is a map $\sMaslov$ from $\sLeg$ to $\Z$ which is constant except at cusps where it satisfies the relation shown in Figure~\ref{f:Maslov}.
\end{definition}

The grading on $(\salg(\sLeg), \df)$ is defined on generators and then extended additively to  $\salg(\sLeg)$. If $a \in Q$ corresponds to a right cusp, then $|a|=1$. If $a$ corresponds to a crossing, then $|a| = \sMaslov(T) - \sMaslov(B)$ where $T$ and $B$ are the strands crossing at $q$ and $T$ has smaller slope. 

An \emph{admissible disk} on a front projection $\sfront$ is an immersion of the disk $D^2$ into the $xz$-plane which satisfies the following: the image of $\df D^2$ is on $\sfront$; the map is smooth except possibly at crossings and cusps; the image of the map near singularity points looks locally like the diagrams in Figures~\ref{f:admissible-disks-1}; and there is one originating singularity (Figure~\ref{f:admissible-disks-1}~(a)) and one terminating singularity (Figure~\ref{f:admissible-disks-1}~(b)). The front projection is assumed to be nearly plat. Otherwise, additional local neighborhoods around singularity points must be considered; see Figure~5 in \cite{Ng2003}.

\begin{figure}[t]
\labellist
\small\hair 2pt
\pinlabel {(a)} [tl] at 108 18
\pinlabel {(b)} [tl] at 280 18
\pinlabel {(c)} [tl] at 386 18
\endlabellist
\centering
\includegraphics[scale=.7]{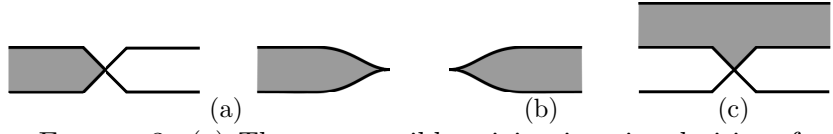}
\caption{(a) The two possible originating singularities of an admissible disk. (b) The terminating singularity of an admissible disk. (c) A convex corner. A second possible convex corner results from reflecting (c) across a horizontal axis.}
\label{f:admissible-disks-1}
\end{figure}

An admissible disk of a nearly plat front projection is embedded away from singularities. Hence the terminating (resp. originating) singularity is the left-most (resp. right-most) point in the image of the disk and at most one singularity can occur at a given crossing of $\sfront$. Given an admissible disk $D$ originating at $a$, let  $b_1, \hdots, b_k \in Q(\sfront)$ denote the convex corners of $D$ (Figure~\ref{f:admissible-disks-1}~(c)), ordered with respect to the counter-clockwise orientation on $\df D$ where $b_1$ is the first convex corner counter-clockwise from $a$. We associate the monomial $w(D) = b_1 \hdots b_k$ to $D$. If $D$ has no convex corners, then $w(D) = 1$. 

\begin{definition}
	\mylabel{defn:DGA-boundary}
	The differential $\df$ on the algebra $\aac(\sfront)$ is defined on a generator $a \in Q$ by the formula:
	\begin{equation}
		\mylabel{eq:DGA-diff}
		\df(a) = 
		\begin{cases} 
		\sum_{w} \# (\Delta(a, w)) w,  & \mbox{if }a\mbox{ is a crossing} \\
		1+\sum_{w} \# (\Delta(a, w)) w, & \mbox{if }a\mbox{ is a right cusp} 
		\end{cases}
	\end{equation}
	\noindent where the sum is over all monomials $w \in \salg(\sfront)$, $\Delta(a, w)$ is the diffeomorphism class of admissible disks originating at $a$ with $w(D) = w$, and $\# (\Delta(a, w))$ is the mod 2 count of the elements in $\Delta(a, w)$. We extend $\df$ to all of $\aac(\sfront)$ by linearity and the Leibniz rule. 
\end{definition}

\begin{theorem}[\cite{Chekanov2002a}, \cite{Ng2003}]
	\mylabel{thm:CE-DGA-diff}
	The differential $\df$ satisfies:
	\begin{enumerate}
		\item The sum in equation~(\ref{eq:DGA-diff}) is finite;
		\item $| \df a | = |a| - 1 $ modulo $2 r(L)$; and 
		\item $\df \circ \df = 0$.
	\end{enumerate}
	The homology of the CE-DGA $(\aac(\sLagr), \df)$ is a Legendrian isotopy invariant, as is its stable-tame isomorphism class.
\end{theorem}

In \cite{Chekanov2002a} and \cite{Chekanov2002}, Chekanov considers a type of algebra homomorphism on $(\aac(\sLagr), \df)$, called an augmentation, which allows us to extract a finite dimensional linear chain complex from the infinite dimensional DGA $(\aac(\sfront), \df)$. These maps, along with the resulting linear homology groups, have provided easily computable Legendrian knot invariants. An \emph{augmentation} is an algebra map $\saug : (\aac(\sfront), \df) \to \zz_2$ satisfying $\saug(1) = 1$, $\saug \circ \df = 0$, and if $\saug (q_i) = 1$ then $ |q_i| = 0 $. We let $\sAugL$ denote the set of augmentations of $(\aac(\sLagr), \df)$. Given $\saug \in \sAugL$ we define the algebra homomorphism $\phi^{\saug} : \aac(\sfront) \to \aac(\sfront)$ to be the extension of the map on generators given by $\phi^{\saug}(q) = q + \saug(q)$. This allows us to consider an alternate differential on $\aac(\sLagr)$ defined by $\df^{\saug}= \phi^{\saug} \df (\phi^{\saug})^{-1}$.

As a vector space, $\aac(\sfront)$ decomposes as $\aac(\sfront)= \bigoplus_{n=0}^{\infty} A(\sfront)^{\otimes n}$ where $A(\sfront)^{\otimes 0}=\zz_2$ and $A(\sfront)^{\otimes n}$ is spanned by the monomials of length $n$ in the generators from $Q$. The differential $\df$ decomposes as $\df = \sum_{n=0}^{\infty} \df_n$ where $\df_n q$ is the sum of monomials of length $n$ in $\df q$. The differential $\df^{\saug}$ has the comparative advantage that $\df^{\saug}_0 = 0$. Thus, $(\df_1^{\saug})^2=0$ and $(A(L), \df_1^{\saug})$ is a finite dimensional chain complex  with homology groups $LCH_{*}^{\saug}(\sLagr)$ called the \emph{linearized contact homology of $\saug$}. The set of homology groups $\{LCH_{*}^{\saug}(\sfront)\}_{\saug \in \sAugL}$ is a Legendrian knot invariant.

\subsection{Generating Families}
\mylabel{sec:Gen-fams}
Consider a function $F : \rr \times \rr^N \to \rr$.  We use coordinates $(x, e), x\in \rr$ and $e\in \rr^N$ for the domain which shall be viewed as a trivial vector bundle over $\rr$.  
Denote by $S_F$ the {\it fiber critical set}
\[
S_F = \left \{ (x,e) \, \left | \, \frac{\partial F}{\partial e}(x,e) = 0 \right \} \right..
\]
We will assume the matrix $\displaystyle \left [ \frac{\partial^2 F}{\partial x \partial e}(x,e) \, \frac{\partial^2 F}{\partial e^2}(x,e) \right ]$ has full rank at all $(x,e) \in S_F$, thus $S_F$ is a $1$-dimensional submanifold of $\rr \times \rr^N$.  Then the map
\[
i_F : S_F \to \rr^3, \quad i_F(x,e) = \left (x, \frac{\partial F}{\partial x}(x,e), F(x, e) \right )
\]
gives a Legendrian immersion.  If $L \subset \rr^3$ is a Legendrian knot or link with $i_F(S_F)= L$, we say that $F$ is a {\it generating family} for $L$. Figure~\ref{f:gen-fam-example} gives a generating family for a Legendrian unknot. We will require generating families to  be {\it linear at infinity}.  That is, $F$ should be equal to some nonzero linear function, $F(x,e) = l(e)$, outside of a compact subset $K \subset \rr \times \rr^N$. 

\begin{figure}[t]
\labellist
\small\hair 2pt
\pinlabel {$x$} [tl] at 958 217
\pinlabel {$z$} [tl] at 115 491
\pinlabel {$x_4$} [tl] at 430 238
\pinlabel {$F(x_4, \cdot)$} [tl] at 490 190
\endlabellist
\centering
\includegraphics[scale=.34]{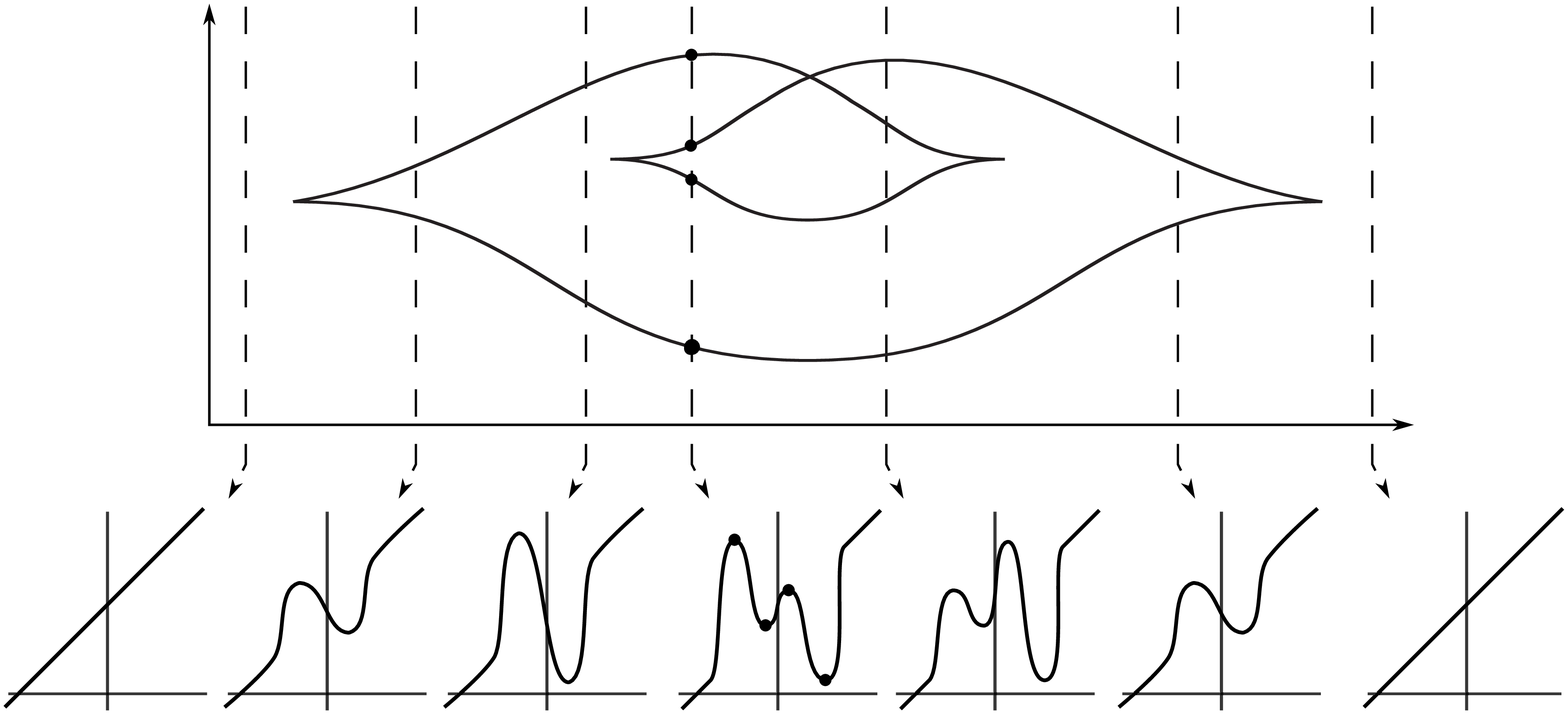}
\caption[A generating family for a Legendrian unknot.]{A linear at infinity generating family $F : \rr \times \rr \to \rr$ for a Legendrian unknot $L$. The graph of $F(x, \cdot): \rr \to \rr$ is given for several values of $x$. The four intersection points of $L \cap (\{x_4\} \times \rr)$ correspond to the four critical points of $F(x_4, \cdot)$.}
\label{f:gen-fam-example}
\end{figure} 

\begin{remark}  Not all Legendrian links $L \subset \rr^3$ admit generating families.  In fact, $L$ admits a linear at infinity generating family if and only if $L$ has a graded normal ruling.  These conditions are in turn equivalent to the existence of a graded augmentation of the Chekanov-Eliashberg DGA. See \cite{Chekanov2005}, \cite{Fuchs2003}, \cite{Fuchs2004}, \cite{Fuchs2008}, \cite{Ng2006}, and \cite{Sabloff2005} for details on these results.
\end{remark}

\section{Motivation from generating families}
\mylabel{ch:Gen-fams}

Generating families have been used to study Legendrian knots in standard contact $\R^3$ as well as in the $1$-jet space $J^1(S^1)$; see \cite{Jordan2006}, \cite{Traynor1997a}, or \cite{Traynor2001}.  In this section, we begin by recalling a construction of Pushkar (this is closely related to the approach in  \cite{Traynor2001} and \cite{Jordan2006}) which uses a generating family to assign homology groups to a Legendrian knot.  In the remainder of the section we sketch a Morse theoretic approach to extending a complex that computes these homology groups to a ``generating family DGA''.  The rigorous construction of such a DGA would involve a detailed study of spaces of difference flow trees (described below) in an appropriate analytic setting, and this is far beyond the scope of the present paper.  Our sole purpose for including this discussion is as motivation for the combinatorially defined DGA which appears later in Section~\ref{ch:DGA-to-MCS}.  As such, the exposition here is kept to a minimum.


\subsection{Generating family homology}
Given a linear at infinity generating family $F$ for a Legendrian knot $L$, one considers a corresponding {\it difference function} $w$ defined on the fiber product of the domain of $F$ with itself:
\begin{equation*}
w: \R \times \R^N \times \R^N\to \R, \quad w(x, e_1, e_2) = F(x, e_1) - F(x, e_2).
\end{equation*}
The {\it generating family homology} of $F$ is defined as the grading shifted (singular) homology groups
\begin{equation*} 
\mathcal{G}H_*(F) = H_{*+N+1}(w_{\geq \delta}, w_{= \delta}; \Z_2)
\end{equation*}
where $\delta$ is taken small enough so that the interval $(0, \delta)$ does not contain critical values. 
Given a fixed Legendrian $L$ one can consider the set of possible isomorphism types of graded groups $\{ \mathcal{G}H_*(F) \}$ where $F$ is any generating family for $L$, and the result is a Legendrian knot invariant. 

When computing $ \mathcal{G}H_*(F) $ one may use a grading shifted version of the Morse complex for the difference function, $\left ( C^{[\delta, +\infty)}_\ast(w; \Z_2), d \right )$. Specifically, we define a complex $(A_F, d_1)$ as follows. Let $\mathit{Crit}^+(w)$ denote the set of critical points $q_i$ of $w$ with positive critical value, $w( q_i) \geq \delta$. Let $A_F$ be the $\Z$-graded vector space over $\Z_2$ with basis $\mathit{Crit}^+(w)$ . A generator $q_i$ is assigned the degree
\begin{equation} 
|q_i| = \mathit{ind}(q_i) - (N +1),
\end{equation}
where $\mathit{ind}(q_i)$ denotes the Morse index and $N$ is the dimension of the fiber in the domain of $F$.  As in the usual Morse complex, the differential $d_1$ depends on an appropriately chosen metric and is defined by counting negative gradient trajectories connecting critical points of adjacent Morse index.


When $F$ is a generating family for a Legendrian $L \subset \R^3$, the critical points of $w$ with positive critical value are in one to one correspondence with the crossings of the $xy$-projection of $L$. Indeed,
\[\displaystyle
\frac{\partial w}{\partial e_1}(x,e_1, e_2) = \frac{\partial w}{\partial e_2}(x, e_1, e_2) = 0 \quad \mbox{implies that} \quad (x,e_1), (x, e_2) \in S_F,
\]
and $\displaystyle
\frac{\partial w}{\partial x}(x,e_1, e_2) = 0$ implies the second coordinates of $i_F(x, e_1)$ and $i_F(x, e_2)$ agree.
Thus, the (linear) generators of $A_F$ are in bijection with the (algebraic) generators of the Chekanov-Eliashberg DGA $\mathcal{A}(L)$, and the degrees of the corresponding generators can be shown to be the same as well; see, for instance, Proposition 5.2 of \cite{Fuchs2008}. 

In fact, it is shown in \cite{Fuchs2008} that there exists an augmentation $\epsilon$ so that the homology of the corresponding linearized complex $(\aac_1(\sLagr), \df_1^{\saug})$ is isomorphic to $\mathcal{G}H_*(F)$. This suggests the more ambitious question:

\bigskip


{\it Given a Legendrian knot $L$ with linear at infinity generating family $F$, is it possible to construct a DGA whose linearized homology agrees with the generating family homology, $\mathcal{G}H_*(F)$, using only data arising from $F$?}

\bigskip

More specifically, one could hope to extend the differential from the Morse complex, $d_1 : A_F \rightarrow A_F$, to a differential on the tensor algebra $\mathcal{A}_F= \oplus_{n=0}^\infty A_F^{\otimes n}$.  The problem then becomes to produce higher order terms $d_n : A_F \rightarrow A_F^{\otimes n}$ so that  $d= d_1 + d_2 + d_3 \ldots$ satisfies $d^2 =0$ after extending by the Leibniz rule.  Note that it is reasonable to expect $d_0 =0$ because of the form of the differentials, $\partial^\epsilon$, on the Chekanov-Eliashberg algebra. 

Works of Betz-Cohen \cite{Betz1994} and Fukaya \cite{Fukaya1997a} give an approach to the cup product and higher order product structures on the cohomology of a manifold $M$ via Morse theory.  Such product structures are realized at the chain level by counting maps from trees into $M$ whose edges parametrize gradient trajectories of certain Morse functions.  A variation on these techniques adapted to the present context was suggested to the authors by Josh Sabloff and will be  sketched in the following subsection.

\subsection{A generating family DGA}

We outline a definition of the higher order terms $d_2, d_3, \ldots$ on the generators $q_i$.  Together with the Leibniz rule this determines a differential on all of $\mathcal{A}_F$.  The first order term  $d_1$ was defined by counting gradient trajectories of the difference function $w$.  For $d_n$ with $n \geq 2$, we count ``difference flow trees'' which are trees mapped into higher order fiber products of the domain of $F$.  The edges are required to be integral curves to the negative gradients of particular generalized difference functions; see Definition \ref{def:DFT} below. 

We begin by describing the domains for our difference flow trees. Following \cite{Fukaya1997a} and \cite{Fukaya1997} we define a {\it metric tree with $1$ input and $n$ outputs} as an oriented tree $\tau$ with $n+1$ $1$-valent vertices (only one of which is oriented as an input) and no $2$-valent vertices together with some additional information; see Figure \ref{fig:DiffTree} for an example.  At each vertex with valence $\geq 3$ there is a single edge oriented into the vertex and an ordering of the outgoing edges is provided.  Finally, each internal edge $E_i$ is assigned a length $l_i \in (0, +\infty)$.

The ordering of edges at vertices specifies, up to isotopy,  an embedding of $\tau$ into the unit disk $D^2 \subset \R^2$ with the $1$-valent vertices located on $\partial D^2$.  This provides an ordering of the $1$-valent vertices as $v_0, v_1, \ldots, v_n$ where $v_0$ is the unique inwardly oriented vertex and the numbering runs counter-clockwise.  The complement $D^2 \setminus \tau$ consists of $n+1$ components each containing some arc of $S^1$. We label these arcs and the corresponding regions from $1$ to $n+1$ beginning to the left of $v_0$ and working counterclockwise; see Figure \ref{fig:DiffTree}. 

Next, we define generalized difference functions.  
Let $P_m$ denote the $m$-th order fiber product 
\begin{equation*} 
P_m := \R \times \underbrace{(\R^N \times \ldots \times \R^N)}_{m \,\,\textit{times}}.
\end{equation*}
Each choice of $1 \leq i < j \leq m$ produces a difference function
\[\begin{array}{cl}
w^m_{i,j} : P_m \rightarrow \R, & \\
w^m_{i,j}(x, e_1, \ldots e_m) = &F(x, e_i) - F(x, e_j) + (e_1^2+ \cdots + e^2_{i-1}) \\ &-(e_{i+1}^2+ \cdots+ e_{j-1}^2) +(e_{j+1}^2 +\cdots + e^2_{m}).
\end{array}
\]

Note that there is a clear one-to-one correspondence between the critical points of any of the 
$w^m_{i,j}$ and the critical points of $w = w^2_{1,2}$ which we denote by
\begin{equation} \label{eq:CritP}
\beta^m_{i,j} : \mathit{Crit}^+(w) \stackrel{\cong}{\longrightarrow} \mathit{Crit}^+(w^m_{i,j}).
\end{equation}

We are now able to give the key definition of this section.
\begin{definition} \label{def:DFT}
A {\it difference flow tree}, $\Gamma = (\tau, \{\gamma_m\})$, is a metric tree $\tau$ with $1$ input and $n$ outputs where to each edge $E_m$ we assign a map $\gamma_m: I_m \to P_{n+1}$. Here $I_m$ denotes $[0, l_m]$, $(-\infty, 0]$, or $[0, \infty)$ depending on whether the edge is internal, external and oriented into the disk, or external and oriented out of the disk respectively. In addition, we require

\begin{enumerate}

\item
The map $\gamma_m$ should be a negative gradient trajectory for $w^{n+1}_{i,j}$ where $i$ and $j$ are the labels of the regions of $D^2 \setminus \tau$ adjacent to the edge $E_m$; and 

\item
The $\gamma_m$ should fit together to give a continuous map from the entire tree into $P_{n+1}$. In particular, we require that for external edges $\lim_{t\rightarrow \pm \infty} \gamma_m (t)$ exists. These limits are necessarily critical points of the corresponding difference function. 
\end{enumerate}
\end{definition}

\begin{figure}
\centerline{\includegraphics[scale=1]{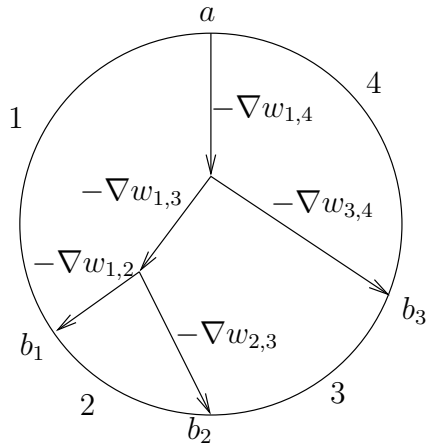}}
\caption{A difference flow tree in $\mathcal{M}^F(a; b_1, b_2, b_3)$.}
\label{fig:DiffTree}
\end{figure}

Now choose critical points $a; b_1, \ldots, b_n \in \mathit{Crit}^+(w)$. As in equation~(\ref{eq:CritP}) these critical points may be identified with the critical points of any of the $w^n_{i,j}$. We then define 
\[
\mathcal{M}^F(a; b_1, \ldots, b_n) 
\]
as the set of difference flow trees with the $1$-valent vertices $v_0, v_1, \ldots, v_n$ mapped to $a, b_1, \ldots, b_n$ respectively; see Figure \ref{fig:DiffTree}.  Finally, we define $d_n: A_F \rightarrow A_F^{\otimes n}$ according to
\begin{equation}
d_n a= \sum \# \mathcal{M}^F(a; b_1, \ldots, b_n) b_1\cdots b_n
\end{equation}
where the sum is over all monomials of length $n$ with $|b_1 \cdots b_n| = |a|-1$.  Here, the grading of $A_F$ is extended additively to $A_F^{\otimes n}$.

Similar spaces of gradient flow trees are studied in \cite{Fukaya1997}.  However, the setup in \cite{Fukaya1997} is a bit more restrictive about the way gradient vector fields are assigned to the edges of a tree so that our difference flow trees do not precisely fit as a special case.  Note also that somewhat different spaces of Morse flow trees have appeared in the work of Ekholm \cite{Ekholm2007} in connection with Legendrian contact homology.


An interesting analysis should be required to rigorously establish Legendrian knot invariants following the sketch of the DGA $\mathcal{A}_F$.  It seems possible that such a construction may work equally well in a higher dimensional setting.  We believe that our combinatorial results provide strong evidence that a rigorously defined generating family DGA exists and hope that our work will encourage a more serious investigation of this topic.

\section{A combinatorial approach to generating families}
\mylabel{ch:MCS}

Morse complex sequences were conceived as a way of encoding Morse-theoretic algebraic information coming from a generating family together with an appropriately chosen Riemannian metric. The definition of a Morse complex sequence (MCS) originates with Pushkar and first appears in print in \cite{Henry2011}. Our presentation differs slightly from that in \cite{Henry2011} as we assign MCSs to a given front projection $L$ rather than encode the front projection as part of the defining data of an MCS.  In this section, after recalling the definition of Morse complex sequences, we define how a generating family together with a choice of metric produces a Morse complex sequence.  The section is concluded with a recollection of the notion of equivalence for Morse complex sequences.

\subsection{A Morse Complex Sequence (MCS) on a front projection} 
\mylabel{sec:MCS-defn}

We begin by fixing a Legendrian knot with front projection $\sfront$ and Maslov  potential $\mu$. The front $\sfront$ need not be nearly plat, although, we do require that crossings and cusps of $\sfront$ have distinct $x$-coordinates.

\begin{definition}
A \emph{handleslide mark} on $\sfront$ is a vertical line segment in the $xz$-plane connecting two strands of $\sfront$ with the same Maslov potential and not intersecting the crossings and cusps of $\sfront$; see, for example, the two handleslide marks in Figure~\ref{f:MCS-trefoil-example}. An \emph{elementary marked tangle} is a portion  of a front diagram $\sfront$ lying between two vertical lines in the $xz$-plane containing either a single crossing, left cusp, right cusp, or handleslide mark. 
\end{definition}

In what follows we will use the convention of labeling the strands of $L$ in increasing order from \emph{top to bottom} along any vertical line $x= x_0$.  We will also use angled brackets $\langle \cdot ,\cdot \rangle$ to denote the usual pairing that arises when working with a vector space with a chosen basis. 

The reader may find it useful to consult Figure~\ref{f:MCS-trefoil-example} while working through the next, rather long, definition.

\begin{definition} 
\label{def:MCS}
A {\it Morse complex sequence} (or {\it MCS}), $\mathcal{C} = \left( \{(C_m, d_m)\}, \{x_m\}, H \right)$, for $L$ consists of a sequence  of values $x_1 < x_2 < \ldots < x_M$ together with a sequence of $\Z$-graded complexes of  vector spaces over $\Z_2$, $(C_1, d_1), \ldots, (C_M, d_M)$, and a collection $H$ of handleslide marks on $L$. 

The following requirements are included in the definition of an MCS.

\begin{enumerate}

\item Each vertical line $x=x_i$, $1 \leq i \leq M$, intersects the front projection $\sfront$ in a non-empty set of points and does not intersect crossings, cusps, or handleslide marks. The vertical lines decompose $\sfront$ into a collection of elementary marked tangles.    

\item For each $1 \leq m \leq M$, $C_m$ has a basis consisting of the points of intersection $L \cap \{ x= x_m\}$ labeled as $e_1, e_2, \ldots, e_{s_m}$ from {\it top to bottom}.  The degree of $e_i$ is given by the Maslov potential of the corresponding strand, $|e_i| = \mu (e_i)$.

\item The differentials $d_m$ have degree $-1$ and are triangular in the sense that
\[\displaystyle
d_m e_i = \sum_{j >i} c_{ij} e_j, \quad c_{ij} \in \Z_2.
\]

\item Suppose $x= x_m$, $m > 1$, is the left border of an elementary marked tangle $T$ containing a crossing or cusp between the strands labeled $k$ and $k+1$. If $T$ contains a left (resp. right) cusp then $\langle d_{m+1} e_k, e_{k+1} \rangle = 1$ (resp. $\langle d_m e_k, e_{k+1} \rangle = 1$). If $T$ contains a crossing, then $\langle d_m e_k, e_{k+1} \rangle = 0$. In the case of $m=1$, $\langle d_m e_1, e_{2} \rangle = 1$.

\item The complexes $(C_m, d_m)$ and $(C_{m+1}, d_{m+1})$ satisfy a relationship depending on the particular tangle lying in the interval $ x_m \leq x \leq x_{m+1}$ as follows:

\end{enumerate}

\noindent (a) {\bf Crossing:}  Supposing the crossing is between strands $k$ and $k+1$, the map $\phi : (C_m, d_m) \rightarrow (C_{m+1}, d_{m+1})$ given by:

			\begin{equation*}
			\phi ( e_i ) =
			\begin{cases}
  e_i &\mbox{ if $i \notin \{k, k+1\} $} \\
  e_{k+1} &\mbox{ if $i = k $} \\
  e_{k} &\mbox{ if $i = k+1 $}, \\
      \end{cases}
      \end{equation*}	
is an isomorphism of complexes.

\smallskip

\noindent (b) {\bf Left Cusp:} Suppose the strands meeting at the cusp are labeled $k$ and $k+1$.  We require the linear map
			\begin{equation*}
			\phi ( e_i) =
			\begin{cases}
  [e_i] &\mbox{ if $i < k $} \\
  [e_{i+2}] &\mbox{ if $i \geq k $}.
      \end{cases}
      \end{equation*}
 to be an isomorphism of complexes from $(C_{m}, d_{m})$ to the quotient of $(C_{m+1}, d_{m+1})$ by the acyclic subcomplex spanned by $\{ e_{k}, d_{m+1} e_{k} \}$. 

\smallskip

\noindent (c) {\bf Right Cusp:}  We make the same requirement as for a left cusp except we reverse the roles of $(C_m, d_m)$ and $(C_{m+1}, d_{m+1})$.

\smallskip

\noindent (d) {\bf Handleslide mark:}  Suppose the handleslide mark is between strands $k$ and $l$ with $k < l$. We require,
			\begin{equation*}
h_{k,l} : (C_m, d_m) \rightarrow (C_{m+1}, d_{m+1}), \quad			h_{k,l} ( e_i) =
			\begin{cases}
 e_i &\mbox{ if $i \neq k $} \\
  e_k + e_l &\mbox{ if $i = k $}.
      \end{cases}
      \end{equation*}
to be an isomorphism of complexes.
\end{definition}

We will denote by $\sFMCS$ the set of Morse complex sequences for $\sfront$. Examples of MCSs, utilizing Barannikov's graphical short-hand for ordered chain complexes from \cite{Barannikov1994}, are given in Figures~\ref{f:MCS-trefoil-example} and \ref{f:MCS-example}.

\begin{remark}
\label{rem:mcs-graphic}
Note that for all $m$ the complexes $(C_m, d_m)$ are acyclic.  This is clearly true for $m = 1$ and the requirements in Definition \ref{def:MCS} relating $(C_m,d_m)$ and $(C_{m+1}, d_{m+1})$ imply that the corresponding homology groups are isomorphic.
\end{remark}

\begin{figure}[t]
\labellist
\small\hair 2pt
\pinlabel {$e_1$} [tl] at 469 228
\pinlabel {$e_2$} [tl] at 530 210
\pinlabel {$e_3$} [tl] at 504 190
\pinlabel {$e_4$} [tl] at 543 166
\pinlabel {$z$} [tl] at 2 224
\pinlabel {$d_5 e_1 = e_2+e_3$} [tl] at 480 150
\pinlabel {$d_5 e_2 = d_5 e_3 = e_4$} [tl] at 480 130
\pinlabel {$d_5 e_4 =0$} [tl] at 480 110
\pinlabel {$x_1$} [bl] at 51 82
\pinlabel {$x_2$} [bl] at 100 82
\pinlabel {$x_3$} [bl] at 192 82
\pinlabel {$x_4$} [bl] at 265 82
\pinlabel {$x_5$} [bl] at 310 82
\pinlabel {$x_6$} [bl] at 366 82
\pinlabel {$x_7$} [bl] at 404 82
\pinlabel {$x_8$} [bl] at 435 82
\pinlabel {$(C_1, d_1)$} [bl] at -5 44
\pinlabel {$(C_2, d_2)$} [bl] at 70 44
\pinlabel {$(C_3, d_3)$} [bl] at 144 44
\pinlabel {$(C_4, d_4)$} [bl] at 217 44
\pinlabel {$(C_5, d_5)$} [bl] at 290 44	
\pinlabel {$(C_6, d_6)$} [bl] at 364 44	
\pinlabel {$(C_7, d_7)$} [bl] at 438 44	
\pinlabel {$(C_8, d_8)$} [bl] at 512 44	
\endlabellist
\centering
\includegraphics[scale=.60]{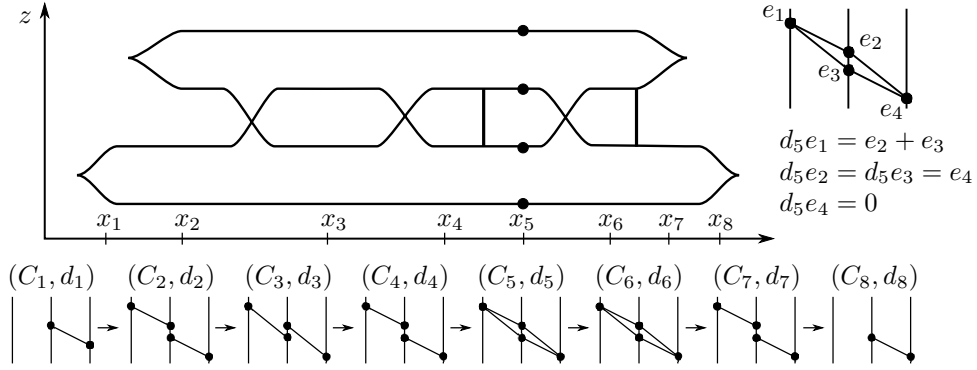}
\caption{An MCS on a Legendrian trefoil. The upper-right inset describes Barannikov's short-hand for the chain complex $(C_5, d_5)$.}
\label{f:MCS-trefoil-example}
\end{figure}

\subsubsection{Graphical presentation of an MCS via implicit handleslide marks}
\label{sec:MCS-graphic}
The isomorphisms required in Definition~\ref{def:MCS}~(5) do not allow $(C_{m+1}, d_{m+1})$ to be uniquely recovered from $(C_m, d_m)$ in the case of a left cusp. As an unfortunate consequence, the handleslide marks of an MCS do not always completely determine the sequence of complexes $(C_m, d_m)$. However, when discussing equivalence of MCSs it is convenient to have an entirely graphical method of encoding an MCS.  Such a graphical presentation requires recording some additional handleslide marks near the cusps as described in the current subsection.

Assume now that a left cusp lies in the region $x_m \leq x \leq x_{m+1}$, but also note that a similar discussion will apply to right cusps.  As in Definition \ref{def:MCS}~(5)~(b) we assume the new strands arising from the cusp correspond to the generators $e_k$ and $e_{k+1}$ in $C_{m+1}$.  Let $(B, d)$ denote the acyclic complex with basis $\{a, b\}$ and differential $d a = b$.  From Definition~\ref{def:MCS}~(5)~(b) it follows that there is an isomorphism of complexes
$F: (C_m, d_m) \oplus (B,d) \rightarrow (C_{m+1}, d_{m+1})$ given by 
\[
\begin{array}{rlr} F(e_i) = & e_i + \langle d_{m+1} e_i, e_{k+1} \rangle e_k, & i < k\\
F(a) = & e_k & \\
F(b) = & d_{m+1} e_k & \\
F(e_i) = & e_{i+2}, & i \geq k
\end{array}
\]
This chain isomorphism can be described as a composition of the handleslide maps $h_{k,l}$ defined in Definition \ref{def:MCS}~(5)~(d). Namely, 

\begin{enumerate}
		\item Let $e_{u_1}, \ldots, e_{u_s}$, $u_1 < \ldots < u_s < k$, denote the generators of $C_{m+1}$ satisfying  $\langle d_{m+1} e_{u_i}, e_{k+1} \rangle = 1$.  
		\item Let  $e_{v_1}, \ldots, e_{v_r}$, $k+1 < v_1 < \ldots < v_r$, denote the generators of $C_{m+1}$ satisfying $\langle d_{m+1} e_{k}, e_{v_i} \rangle = 1$.
\end{enumerate}

Then, 
\begin{equation} \label{eq:HSComp}
F = h_{k+1, v_1} \circ \hdots \circ h_{k+1, v_r} \circ h_{u_1, k} \circ \hdots \circ h_{u_s, k} \circ J,
\end{equation}
 where $J:
 C_m \oplus B \rightarrow C_{m+1}$ simply identifies the two underlying vector spaces via
\[
J(e_i) = \begin{cases} e_i &\mbox{ if $i < k$ } \\
  e_{i+2} & \mbox{ if $i \geq k $},      \end{cases} 
 \quad J(a) = e_k, \quad \mbox{and} \,\,J(b) = e_{k+1}.
 \]

Therefore, we can encode an MCS graphically by including additional markings near the cusps to indicate which handleslides occur in the composition in equation~(\ref{eq:HSComp}).  That is, immediately following a left cusp  we place what we will call an {\it implicit handleslide mark} between strands $\alpha$ and $\beta$ whenever $h_{\alpha,\beta}$ appears in the composition $F$. This amounts to including implicit handleslide marks from $i$ to $k$ if and only if $\langle d_{m+1} e_{i}, e_{k+1} \rangle = 1$ and from $k+1$ to $j$ if and only if $\langle d_{m+1} e_{k}, e_{j} \rangle = 1$. The resulting implicit handleslide marks are represented in our figures by dotted vertical lines to distinguish them from the ordinary handleslide marks, which for clarity we now refer to as \emph{explicit handleslide marks}.  We include implicit handleslide marks near right cusps in an analogous way.  The exact order in which we place implicit handleslide marks does not matter as the handleslide maps occurring in equation~(\ref{eq:HSComp}) commute with one another.

In summary,  $\sMCS \in \sFMCS$ is uniquely represented by the collection of implicit and explicit handleslide marks on $\sfront$. We call this collection the \emph{marked front projection for $\sMCS$}; see, for examples, Figures~\ref{f:MCS-trefoil-example} and \ref{f:MCS-example}. The sequence of complexes $\{(C_m, d_m) \}$ may be reconstructed using (5) (a)-(d) in Definition \ref{def:MCS} and equation~(\ref{eq:HSComp}). It is important to note that not every collection of such marks determines an MCS. 

\begin{figure}[t]
\centering
\includegraphics[scale=.4]{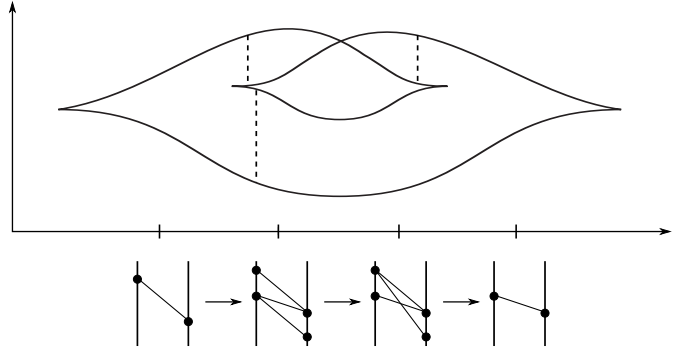}
\caption{The MCS corresponding to the generating family in Figure~\ref{f:gen-fam-example}. The MCS has three implicit handleslide marks.}
\label{f:MCS-example}
\end{figure}

We close this subsection with a dual remark. The sequence of complexes $\{(C_m, d_m) \}$ do not uniquely determine the location of the handleslide marks as conjugation by a handleslide map $h_{k,l}$ will often not affect a differential.

\subsection{From a Generating Family to an MCS} 
\mylabel{sec:Gen-to-MCS}

In this subsection we describe how a sufficiently generic generating family $F: \rr \times \rr^N \to \rr$ for a Legendrian link $L$ together with a (sufficiently generic) choice of metric, $g$, on $ \rr \times \rr^N $ produces an MCS, $\mathcal{C}(F, g)$, for $L$; see, for example, Figures~\ref{f:gen-fam-example}~and~\ref{f:MCS-example}. In fact, this construction motivates both the definition of an MCS as well as the nomenclature. 

Given a linear at infinity generating family $F$ we let $f_x: \R^N \to \R$, $x \in \R$ denote the corresponding family of functions $f_x = F(x, \cdot)$. We make the following assumptions:
\begin{enumerate}

\item $f_x$ is Morse except at finitely many values of $x$ which we denote as $s_1 < \ldots < s_M$.

\item For each $x = s_i$, $f_x$ has a single degenerate critical point which is a standard birth or death point.
\end{enumerate}

\begin{remark}
(i) Conditions (1) and (2) are generic in some appropriate sense so that a linear at infinity generating family $F$ can be made to satisfy (1) and (2) after a perturbation on a compact subset.

(ii) The birth (resp. death) critical points in (2) are exactly the left (resp. right) cusp points of the front projection of the corresponding Legendrian. 
\end{remark}

Recall that $S_F$ denotes the fiberwise critical set of $F$ and let $\Sigma \subset S_F$ denote those $(x, e) \in S_F$ such that $e$ is a degenerate critical point of $f_x$. The Morse index provides a continuous integer valued function on $S_F\setminus \Sigma$ which in turn gives a Maslov potential, $\mu$, for the corresponding Legendrian $L = i_F(S_F)$.

Choosing a family of metrics $g_x$ on $\R^N$ gives rise to gradient vector fields $\nabla f_x$. For those values of $x$ where $f_x$ is Morse and the pair $(f_x, g_x)$ satisfies the Morse-Smale condition we consider the Morse complexes $C(f_x, g_x)$. 
\begin{definition}
A {\it handleslide} is a non-constant gradient trajectory for one of the pairs $(f_x, g_x)$ which connects two non-degenerate critical points of the same Morse index.
\end{definition}

Following \cite{Laudenbach1992}, one can choose the family of metrics $g_x$ so that
\begin{enumerate}

\item Handleslides only occur at a finite number of values of $x$ which we denote as $t_1 < \ldots < t_M$. None of the vertical lines $x=t_i$ should intersect a crossing or a cusp, and at each $x= t_i$ only one handleslide should occur (up to reparametrization).

\item For the values $x= s_i$ where $f_{s_i}$ has a degenerate critical point, we can find $\delta > 0$ such that for any $x^- \in [s_i- \delta, s_i)$ and $x^+ \in (s_i, s_i + \delta]$, the Morse complexes $C(f_{x^-}, g_{x^-})$ and $C(f_{x^+}, g_{x^+})$ are defined and are related as in 
Definition \ref{def:MCS}~(5)~(b) or (c) depending on whether the degenerate critical point is a birth or death.

\item If the interval $[a, b]$ does not contain a degenerate critical point or a handleslide and the Morse complexes $C(f_a, g_a)$ and $C(f_b, g_b)$ are defined\footnote{Another unavoidable feature of one parameter families of gradient vector fields is that at some isolated points a non-transverse intersection may occur between descending and ascending manifolds of critical points of adjacent index. This is inconsequential as the standard way this will happen is that the non-transverse intersection will split into two new trajectories which cancel each other out in the Morse complex (even over $\Z$). See \cite{Laudenbach1992}.} then they are identical. Here we identify the critical points of $f_a$ and $f_b$ that belong to the same components of $S_F \setminus \Sigma$ and this allows us to view the underlying vector spaces as being the same. The differentials are then assumed to be equal.
\end{enumerate}

Given such a family of metrics we define an MCS, $\mathcal{C}(F,g)$, for $L$ in the following way. First add handleslide marks to $L$ where the actual handleslides of the family occur. Choose a sequence $x_0 < x_1 < x_2 < \ldots < x_M < x_{M+1}$ so that $F$ is linear when $x \notin [x_0, x_{M+1}]$ and each region $x_m \leq x \leq x_{m+1}$ contains a single crossing, cusp, or handleslide mark in its interior. Finally, take $(C_m, d_m)$ to be the Morse complex of the pair $(f_{x_m}, g_{x_m})$. It is shown in \cite{Laudenbach1992} that successive complexes $(C_m, d_m)$ and $(C_{m+1}, d_{m+1})$ are related as in Definition~\ref{def:MCS}~(5).

In summary, we have: 

\begin{proposition} \label{prop:GF2MCS} Let $F$ be  a generic linear at infinity generating family for a Legendrian link $L \subset \R^3$.  Then, one can choose a metric $g$ so that there is an MCS $\mathcal{C}(F,g)= \left( \{(C_m, d_m)\}, \{x_1, \hdots, x_M\}, H \right)$ satisfying:
\begin{itemize}
\item The collection of handleslide marks, $H$, correspond to the handleslides of the family $(f_x, g_x)$, and 
\item for $1 \leq m \leq M$, the complex $(C_m, d_m)$ agrees with the Morse complex of $f_{x_m}$ with respect to the metric $g_{x_m}$.
\end{itemize}
\end{proposition} 

\subsubsection{Terminology and notation for the differentials $d_m$ in an MCS}
\label{rem:GTNotation}

Let $\mathcal{C} = \left( \{(C_m, d_m)\}, \{x_m\}, H \right) \in \sFMCS$ . The generators of the complexes $C_m$ correspond to the strands of $L$ at $x=x_m$. To describe the differentials $d_m$ it is natural to use terminology motivated by the case of an MCS, $\mathcal{C}(F, g)$, arising from a generating family.  For instance, to indicate that $\langle d_m e_i, e_j \rangle = 1$ we may say that the $i$-th and $j$-th strands are ``connected by a gradient trajectory'' at $x=x_m$. This is indicated pictorially in our figures by a dotted arrow at $x=x_m$ pointing from the $i$-th strand to the $j$-th strand.  See, for instance, Figure \ref{f:MCS-explosion}.

The requirements from Definition \ref{def:MCS}~(5) may be easily described using this pictorial description of the differentials $d_m$. When moving through a crossing, the dotted arrows are just pushed along the front diagram.  When passing a handleslide (in either direction) new gradient trajectories appear (or disappear if they already exist since we work mod $2$) along any strands that can be connected by a ``broken trajectory'' consisting of one part gradient trajectory and one part handleslide mark.

\subsection{MCS Equivalence}
\mylabel{sec:MCS-equiv}

\begin{figure}[t]
\labellist
\small\hair 2pt
\pinlabel {1} [bl] at 128 696
\pinlabel {2} [bl] at 478 696
\pinlabel {3} [bl] at 128 582
\pinlabel {4} [bl] at 478 582
\pinlabel {5} [bl] at 128 472
\pinlabel {6} [bl] at 478 472
\pinlabel {7} [bl] at 128 361
\pinlabel {8} [bl] at 478 361
\pinlabel {9} [bl] at 128 248
\pinlabel {10} [bl] at 472 248
\pinlabel {11} [bl] at 123 136
\pinlabel {12} [bl] at 472 136
\pinlabel {13} [bl] at 123 24
\pinlabel {14} [bl] at 472 24
\endlabellist
\centering
\includegraphics[scale=.4]{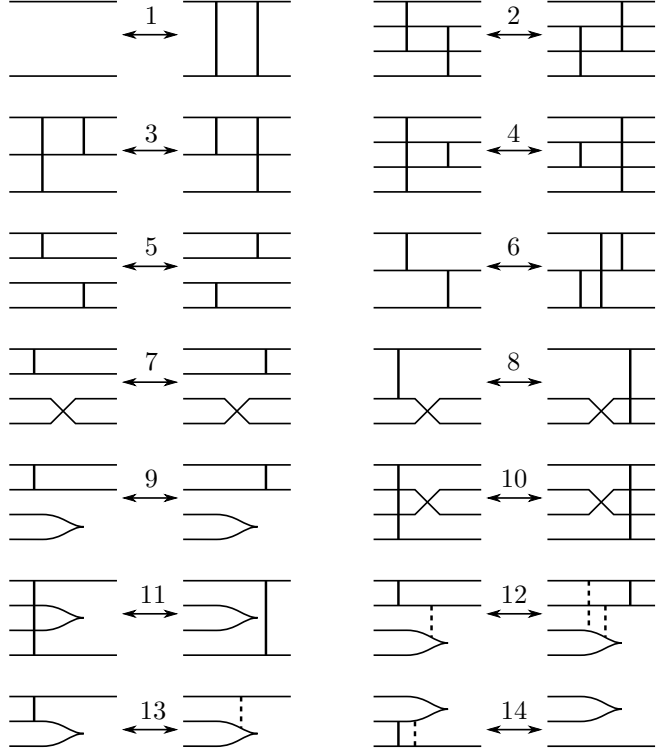}
\caption[MCS Moves.]{MCS moves 1-14.}
\label{f:MCS-equiv}
\end{figure}

 Each MCS $\sMCS \in \sFMCS$ is uniquely encoded by its marked front projection. There is a geometrically natural equivalence relation on $\sFMCS$ generated by a collection of local moves involving the implicit and explicit handleslide marks in $\sMCS$. The \emph{MCS moves} are described in Figures~\ref{f:MCS-equiv}~and~\ref{f:MCS-explosion}. Each move involves a portion of the front projection $\sfront$ containing at most one crossing or one cusp. Additional moves are created by reflecting moves 1-14 about a horizontal or vertical axis. We do not consider analogues of moves 1 - 6 for implicit handleslide marks, since the order of implicit handleslide marks at a birth or death is irrelevant. In moves 9, 11 and 12 there may be other implicit marks that the indicated explicit handleslide mark commutes past without incident. Additional implicit marks may also be located at the birth or death in moves 13 and 14.

\begin{figure}[t]
\labellist
\small\hair 2pt
\pinlabel {$u_1$} [br] at 10 105
\pinlabel {$u_2$} [br] at 10 80
\pinlabel {$i$} [br] at 10 55
\pinlabel {$j$} [br] at 10 30
\pinlabel {$v_1$} [br] at 10 5
\pinlabel {15} [bl] at 132 66
\endlabellist
\centering
\includegraphics[scale=.6]{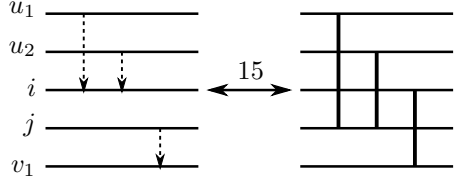}
\caption[MCS Moves.]{MCS move 15. The dotted arrows indicate $\langle d_p e_{u_1}, e_i \rangle = \langle d_p e_{u_2}, e_i \rangle = 1$ and $\langle d_p e_{j}, e_{v_1} \rangle = 1$.}
\label{f:MCS-explosion}
\end{figure}

MCS move 15 introduces or removes possibly many handleslide marks from an MCS. Suppose $\mathcal{C} = \left( \{(C_m, d_m)\}, \{x_m\}, H \right) \in \sFMCS$ and suppose for some $(C_p, d_p)$ there exists generators $e_i$ and $e_j$ with $i<j$ and $|e_i| = |e_j| -1$. Let $e_{u_1}, e_{u_2}, \hdots, e_{u_r}$ and $e_{v_1}, e_{v_2}, \hdots, e_{v_s}$ denote the generators of $C_p$ satisfying $\langle d_p e_{u_k}, e_i \rangle=1$ and $\langle d_p e_j, e_{v_k} \rangle=1$ respectively. It is straightforward to check that the composition of handleslide maps $h = h_{i, v_1} \circ \hdots \circ h_{i, v_s} \circ h_{u_1,j} \circ \hdots \circ h_{u_r, j}$ commutes with $d_p$. MCS move 15 introduces the handleslide marks corresponding to the maps in the composition; see Figure~\ref{f:MCS-explosion}. In particular, choose $y_0 = x_p < y_{1} < \hdots < y_{r+s} < x_{p+1}$ so that the tangle $T= \sfront \cap ([y_0, y_{r+s}] \times \rr)$ does not include singular points of $\sfront$. The new MCS $\sMCS'$ is created by introducing $r+s$ handleslide marks in $T$. For $0 < k \leq r$, place a handleslide mark in the tangle $\sfront \cap([y_{k-1}, y_k] \times \rr)$ between strands $u_k$ and $j$. For $r < k \leq r+s$, place a handleslide mark in the tangle $\sfront \cap([y_{k-1}, y_k] \times \rr)$ between strands $i$ and $v_{k-r}$. Since $d_p = h \circ d_p \circ h^{-1}$, the new marks, along with the existing marks and chain complexes of $\sMCS$, uniquely define an MCS $\sMCS'$. MCS move 15 also allows us to remove such sequences of handleslide marks when they exist. When $r+s=0$ no new handleslide marks are introduced and $\sMCS$ remains unchanged. We will not consider such a case to be an application of MCS move 15. 

Two MCSs $\sMCS$ and $\sMCS'$ in $\sFMCS$ are \emph{equivalent} if there exists a sequence $\sMCS = \sMCS_1, \sMCS_2, \hdots, \sMCS_n = \sMCS' \in \sFMCS$ such that for each $1 \leq i < n$, $\sMCS_i$ and $\sMCS_{i+1}$ are related by an MCS move. The set of equivalence classes $\sbMCS$ is denoted $\sFMCSeq$. Proposition 3.8 in \cite{Henry2011} shows that if we begin with an MCS, then the result of an MCS move is also an MCS and, hence, this equivalence is well-defined.

\begin{remark}
The equivalence relation on Morse complex sequences is motivated by considering possible changes to the MCS $\mathcal{C}(F,g)$ of a generating family $F$ arising from varying the choice of metric $g$; see \cite{Hatcher1973}. 
\end{remark}

\section{Associating a DGA to an MCS}
\mylabel{ch:DGA-to-MCS}

In the current section, we associate a DGA  $(\mathcal{A}_\mathcal{C}, d)$ to a Legendrian knot with chosen Morse complex sequence, $\mathcal{C}$, and give proofs of Theorems \ref{thm:d2is0} and \ref{thm:linear-iso}.  Our approach is a combinatorial analog of that described in Section 3.  In place of the spaces of difference flow trees, $\mathcal{M}^F(a; b_1, \ldots, b_n)$, we introduce sets of {\it chord paths},  $\mathcal{M}^{\mathcal{C}}(a; b_1, \ldots, b_n)$.  The elements of a given $\mathcal{M}^{\mathcal{C}}(a; b_1, \ldots, b_n)$ are finite sequences of vertical markers connecting strands of the front diagram of $L$ (``chords'') which advance along the front diagram from right to left subject to constraints dictated by $\mathcal{C}$.  Later, in Section 6 we give an alternate interpretation of chord paths as certain broken gradient trajectories of the difference function when $\mathcal{C} = \mathcal{C}(F,g)$ for some generating family $F$.  While the proofs in the present section are entirely combinatorial, we provide occasional remarks on how the combinatorics translate to this more geometric setting.

\subsection{Spaces of chord paths}
\label{sec:chord-paths}

In this section we will make use of the standing assumption that the front diagram of $L$ is nearly plat.  This assumption is not essential, but it simplifies the definition of the sets $\mathcal{M}^{\mathcal{C}}(a; b_1, \ldots, b_n)$ as extra possibilities need to be allowed for the behavior of chord paths near cusps with implicit handleslides.

We begin by introducing some preliminary terminology and notations. A {\it chord}, $\lambda$, on a front diagram $L$ will refer to a vertical segment connecting two strands of the front diagram.  A chord is uniquely specified by its $x$-coordinate and the two strands that it connects. We will write $\lambda = (x_0, [i,j])$ to indicate that $\lambda$ is placed at $x= x_0$ and has the $i$-th strand as its upper end point and $j$-th strand as its lower end point.  We use here our convention of enumerating the strands of $L$ at $x=x_0$ as $1,2, \ldots$ from \emph{top to bottom}, so $i<j$.    

Suppose $\mathcal{C}= \left( \{(C_m, d_m)\}, \{x_m\}, H \right)$ is an MCS for the Legendrian link $L$. Let $a$ (resp. $b$) be a crossing or right cusp (resp. crossing) of $L$.  

\begin{figure}[t]
\labellist
\small\hair 2pt
\pinlabel {$x_l$} [tr] at 57 5
\pinlabel {$k$} [tr] at 18 61
\pinlabel {$k+1$} [tr] at 18 36
\pinlabel {$a$} [tl] at 98 45
\pinlabel {$x_l$} [tl] at 158 5
\pinlabel {$k$} [tl] at 250 61
\pinlabel {$k+1$} [tl] at 250 36
\pinlabel {$a$} [tl] at 189 28
\endlabellist
\centering
\includegraphics[scale=.7]{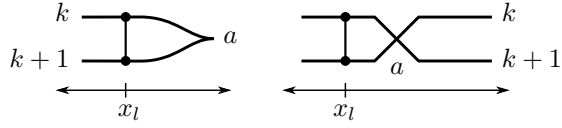}
\caption{The two possible originating chords in a chord path.}
\label{f:chord-path-right-cusp-crossing-axes}
\end{figure}

\begin{figure}[t]
\labellist
\small\hair 2pt
\pinlabel {$i$} [tr] at 9 85
\pinlabel {$k$} [tr] at 9 60
\pinlabel {$k+1$} [tr] at 9 35
\pinlabel {$x_p$} [tl] at 20 5
\pinlabel {$x_{p+1}$} [tl] at 90 5
\pinlabel {$k$} [tl] at 263 85
\pinlabel {$k+1$} [tl] at 263 60
\pinlabel {$j$} [tl] at 263 35
\pinlabel {$x_p$} [tl] at 163 5
\pinlabel {$x_{p+1}$} [tl] at 235 5
\endlabellist
\centering
\includegraphics[scale=.7]{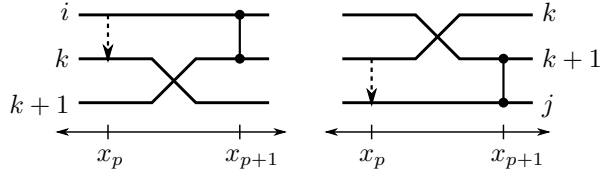}
\caption[]{Possible terminal chords in a chord path.}
\label{f:chord-path-end-crossings-axes}
\end{figure}

\begin{definition} \label{def:ChordPath}
 A {\it chord path} from $a$ to $b$ is a finite sequence of chords $\Lambda = (\lambda_1, \lambda_2, \ldots, \lambda_n)$, so that 
\begin{enumerate}

\item The $x$-values of the $\lambda_r$ are a subset of the $x$-values $ x_1 < x_2 < \ldots < x_s $ used in the MCS to decompose $L$ into elementary tangles.  Furthermore,  subject to this restriction, the $x$-value of $\lambda_{r+1}$ should be located immediately to the left of $\lambda_r$ for any $1 \leq r < n$.  

\item Assuming $a$ occurs between $x_l$ and $x_{l+1}$ and involves strands $k$ and $k+1$, we require $\lambda_1 = ( x_l, [k, k+1])$; see Figure~\ref{f:chord-path-right-cusp-crossing-axes}. We say $\lambda_1$ is the \emph{originating chord} of $\Lambda$.

\item Assuming $b$ lies between $x_{p}$ and $x_{p+1}$ and involves strands $k$ and $k+1$ we require either 

\begin{enumerate}

\item $\lambda_n = (x_{p+1}, [i, k])$  with $\langle d_p e_i, e_k \rangle = 1$, or

\item $\lambda_n = (x_{p+1}, [k+1, j])$  with $\langle d_p e_{k+1}, e_j \rangle = 1$.

\end{enumerate}
See Figure~\ref{f:chord-path-end-crossings-axes}. We say $\lambda_n$ is the \emph{terminating chord} of $\Lambda$.

\item Suppose $1 \leq r < n$, $\lambda_r = (x_{m+1}, [i,j])$, and the chord $\lambda_{r+1} = (x_{m}, [i',j'])$.  Then, $i'$ and $j'$ are required to satisfy some restrictions depending on the type of tangle appearing in the region $x_m \leq x \leq x_{m+1}$ as follows:
\end{enumerate}

\smallskip

\noindent (a) {\bf Crossing:} Assuming the crossing involves strands $k$ and $k+1$, we require $i' = \sigma(i), j' = \sigma(j)$ where $\sigma$ denotes the transposition $(k \,\, k+1)$.  Note that a chord of the form $\lambda_r = (x_{m+1}, [k,k+1])$ is not allowed to occur to the right of a crossing in a chord path.

\smallskip

\noindent (b) {\bf Left cusp:}  This cannot occur since $L$ is nearly plat.

\smallskip

\noindent (c) {\bf Right cusp:} Assuming the strands meeting at the cusp are labeled $k$ and $k+1$, we require $i' = \tau(i)$, $j' = \tau(j)$ where $\tau(l) = 
\begin{cases} 
l & \mbox{if $l < k$} \\
l+2 & \mbox{if $l \geq k$}.
\end{cases}$

\noindent (d) {\bf Handleslide:}  Suppose the handleslide occurs between strand $k$ and strand $l$ with $k < l$.  Here, $i' = i$ and $j' =j$ is always allowed.  In addition, if $i= k$ and $l < j$ we also allow $i'=l$ and $j' =j$.  Analogously, if $j=l$ and $i < k$ then we allow $i' =i$ and $j' = k$.  This rule can be viewed as allowing the endpoint of a chord to possibly jump along the handleslide.  Such a jump will shrink the length of the chord.

\end{definition}

\medskip

\begin{figure}[t]
\labellist
\small\hair 2pt
\pinlabel {$\lambda_{i+1}$} [tl] at 21 125
\pinlabel {$\lambda_{i}$} [tl] at 92 125
\endlabellist
\centering
\includegraphics[scale=.65]{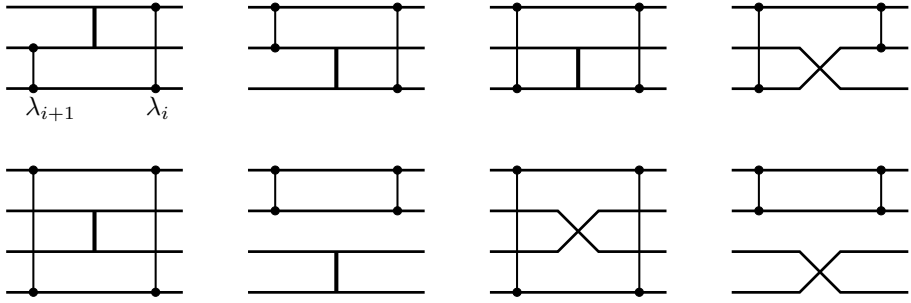}
\caption[]{Successive chords in a chord path near a crossing or handleslide. The chords in a chord path progress from right to left. This is indicated in the top left figure.}
\label{f:chord-path-handleslides-crossings}
\end{figure}

Figure~\ref{f:chord-path-handleslides-crossings} provides a pictorial summary of the appearance of successive chords in a chord path. We let $\mathcal{M}^{\mathcal{C}}(a; b)$ denote the set of all chord paths from $a$ to $b$.  

In order to obtain the higher order terms in the differential we need to allow additional behavior of chord paths near crossings.

\begin{definition} \label{def:ConvexCorners} A {\it chord path with convex corners} is a sequence of chords $\Lambda = (\lambda_1, \ldots, \lambda_n)$ satisfying all the requirements of a chord path except that if $\lambda_r$ and $\lambda_{r+1}$ have a crossing, $c$, between them involving strands $k$ and $k+1$ then we allow the following as an alternative to Definition~\ref{def:ChordPath}~(4a):

If $\lambda_r = (x_{m+1}, [i,k])$ we allow $\lambda_{r+1} = (x_{m}, [i,k])$.  Also, if $\lambda_r = (x_{m+1}, [k+1,j])$ we allow $\lambda_{r+1} = (x_{m}, [k+1,j])$; see Figure~\ref{f:chord-path-convex-corners}. If either of these alternatives occur, then we refer to $c$ as a \emph{convex corner} of $\Lambda$.

To a chord path with convex corners $\Lambda$ we assign a {\it word}, $w(\Lambda)$, in the following manner.  Following the chord path from right to left, read off all the convex corners which occur at the top of a chord as $f_1, \ldots, f_s$.  Now working backwards, from left to right, read off the convex corners which occur at the bottom of a chord as $h_1, \ldots, h_t$.  Finally, let $b$ denote the crossing where the chord path terminates, and put $w(\Lambda) = (f_1 \cdots f_s) b (h_1 \cdots h_t)$; see Figure~\ref{f:chord-path}.
\end{definition}

\begin{figure}[t]
\centering
\includegraphics[scale=.7]{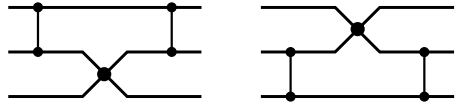}
\caption{Convex corners of a chord path at a crossing. Henceforth, the crossing of a convex corner will be marked with a dot.} 
\label{f:chord-path-convex-corners}
\end{figure}

\begin{figure}[t]
\labellist
\small\hair 2pt
\pinlabel {$a$} [tl] at 307 37
\pinlabel {$b$} [tl] at 51 40
\pinlabel {$f_1$} [tl] at 137 104
\pinlabel {$h_1$} [tl] at 247 15
\endlabellist
\centering
\includegraphics[scale=.7]{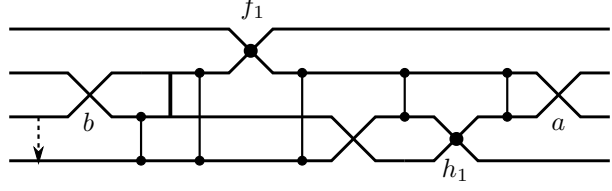}
\caption{A chord path $\Lambda$ with two convex corners and $w(\Lambda) = f_1 b h_1$.}
\label{f:chord-path}
\end{figure}

Now, let $\mathcal{M}^{\mathcal{C}}(a; b_1, \ldots, b_n)$ denote the set of all chord paths with convex corners beginning at $a$ and having $w(\Lambda) = b_1\cdots b_n$.  Recall that for front projections a Maslov potential is used to assign a degree to crossings and right cusps as in the construction of the Chekanov-Eliashberg algebra.  (See Section 2.)

\begin{proposition} \label{prop:CPDegree} 
If  $\mathcal{M}^{\mathcal{C}}(a; b_1, \ldots, b_n)$ is non-empty, then $|b_1\cdots b_n|= |a|-1$.
\end{proposition}

\begin{proof}
Given a chord $\lambda = (x_m, [i,j])$, let $\mu(\lambda)$ denote the difference of the value of the Maslov potential on the upper and lower endpoints of $\lambda$.  That is, $\mu(\lambda) = \mu(e_i) - \mu(e_j)$ at $x=x_m$. 

Now, suppose $\Lambda =  (\lambda_1, \lambda_2, \ldots, \lambda_n) \in \mathcal{M}^{\mathcal{C}}(a; b_1, \ldots, b_n)$.  For $1 \leq r < n$, a case by case analysis shows $\mu(\lambda_r) = \mu(\lambda_{r+1})$ unless there is a convex corner between $\lambda_r$ and $\lambda_{r+1}$.  In the latter case, supposing the convex corner occurs at the crossing $b_i$ we have $\mu(\lambda_{r+1}) = \mu(\lambda_r) - |b_i|$.  Now, suppose that $b_j$ is the crossing where the chord path terminates.  The remaining $b_i$ all appear as convex corners at some point in the chord path, so we see inductively that
\begin{equation}\label{eq:CPind}
\mu(\lambda_n) =  \mu(\lambda_1) - \sum_{\substack{1 \leq i \leq n \\ i \neq j}} |b_i|.
\end{equation}  

By definition $\mu(\lambda_1) = |a|$.   Furthermore, Definition~\ref{def:ChordPath}~(3) and the fact that the differentials in the complexes $(C_m, d_m)$ have degree $-1$ show that $\mu(\lambda_n)= |b_j| +1$.  Combining these two observations with equation~(\ref{eq:CPind}) completes the proof.

\end{proof}

\subsection{A DGA associated to an MCS}
\label{sec:MCS-DGA-defn}

We are now prepared to associate a DGA to an MCS $\mathcal{C}$.  Although the underlying algebra is identical to that of the Chekanov-Eliashberg DGA we use a distinct notation for clarity.  Let $A_{\mathcal{C}}$ denote the $\Z_2$ vector space generated by the set, $Q(L)$, of crossings and right cusps of $L$, and let $\mathcal{A}_\mathcal{C} = \bigoplus_{n=0}^\infty A_\mathcal{C}^{\otimes n}$ denote the corresponding tensor algebra.  The same degrees are assigned to generators as in the Chekanov-Eliashberg DGA following Definition~\ref{defn:Maslov}.
Next, we define a differential $d: \mathcal{A}_\mathcal{C} \rightarrow  \mathcal{A}_\mathcal{C}$ to have the form $\displaystyle d a = \sum_{n=1}^\infty d_n a $ on generators where
\[
d_n: A_{\mathcal{C}} \rightarrow A_{\mathcal{C}}^{\otimes n}, \quad d_n a = \sum \# \mathcal{M}^{\mathcal{C}}(a; b_1, \ldots, b_n) b_1 \cdots b_n.
\]
The sum ranges over all monomials $b_1\cdots b_n$ with $|b_1\cdots b_n| = |a| - 1$ and $\# \mathcal{M}^{\mathcal{C}}(a; b_1, \ldots, b_n)$ is the mod 2 count of the elements in $\mathcal{M}^{\mathcal{C}}(a; b_1, \ldots, b_n)$. The sums used to define $d a$  and $d_n a$ are finite because for a given $\mathcal{C}$ there are only finitely many chord paths.  

Our main results regarding this construction are the following.

\begin{theorem} \label{thm:d2is0}
If $\mathcal{C}$ is an MCS for a nearly plat position Legendrian knot $L$, then 
$(\mathcal{A}_\mathcal{C}, d)$  
is a DGA. That is, $d$ has degree $-1$, satisfies the Liebniz rule, and has $d^2 = 0$. 
\end{theorem}

\begin{theorem} 
\label{thm:linear-iso}
Let $\sMCS^{+}$ and $\sMCS^{-}$ be MCSs for a nearly plat position Legendrian knot $L$.  If  $\sMCS^{+}$ and $\sMCS^{-}$ are equivalent, then the linearized complexes $(A_{\sMCS^{+}}, d^{+}_1)$ and $(A_{\sMCS^{-}}, d^{-}_1)$ are isomorphic.
\end{theorem}

The assumption that $L$ is nearly plat is not essential.  We make it to simplify the definition of chord paths and reduce the number of cases considered in the proofs below.

\subsection{Proof that $(d_1)^2 = 0$} 
\mylabel{sec:d2=0}

As a warm up, we show that in the linearized complex $(A_\mathcal{C}, d_1)$, $d_1 \circ d_1=0$. This case illustrates the nature of our argument and is shorter as chord paths with convex corners need not be considered.  

Let $a, b \in Q(L)$ with $|b| = |a| -2$.
The coefficient of $b$ in $(d_1)^2 a$ is given by
\begin{equation} \label{eq:dsq1}
\langle d^2 a , b \rangle = \# \bigcup_{|c| = |a| -1} \mathcal{M}^\mathcal{C}(a;c) \times  \mathcal{M}^\mathcal{C}(c;b).
\end{equation}
The right hand side will be shown to be $0$ mod $2$ via a combinatorial analog of the corresponding proof for moduli spaces of gradient flow lines.  The elements of $\# \bigcup_{|c| = |a| -1} \mathcal{M}^\mathcal{C}(a;c) \times  \mathcal{M}^\mathcal{C}(c;b)$ consist of once broken chord paths from $a$ to $b$, and in the following we simply refer to them as {\it broken chord paths}.

To begin we introduce a ``1-dimensional'' set of chord paths, $\mathcal{M}^\mathcal{C}(a;b)$.  

\begin{definition}
\label{defn:exceptional-step}
When $|b| = |a| -2$ we define a chord path $\Lambda \in \mathcal{M}^\mathcal{C}(a;b)$ to be a sequence of chords $(\lambda_1, \ldots, \lambda_n)$ as in Definition \ref{def:ChordPath} except that for a single value of $1 \leq r < n$, $\lambda_r$ and $\lambda_{r+1}$ violate Definition~\ref{def:ChordPath} (4) and instead satisfy 
\[
\lambda_r = (x_p, [i,j]) \quad \mbox{and} \quad \lambda_{r+1} = (x_p, [i', j'])
\]
with either
\begin{enumerate}
\item $i' = i$ and $\langle d_p e_{j'}, e_j \rangle = 1$, or
\item $j' = j$ and $\langle d_p e_{i}, e_{i'} \rangle = 1$.
\end{enumerate}
We will refer to the pair of chords $\lambda_r$ and $\lambda_{r+1}$ as the {\it exceptional step} of $\Lambda$; see Figure~\ref{f:exceptional-step}.
\end{definition}

\begin{figure}[t]
\labellist
\small\hair 2pt
\pinlabel {$i$} [tr] at 6 70
\pinlabel {$i'$} [tr] at 6 45
\pinlabel {$j=j'$} [tr] at 6 20
\pinlabel {$\lambda_r$} [tr] at 110 9
\pinlabel {$\lambda_{r+1}$} [tr] at 50 9

\pinlabel {$i=i'$} [tl] at 315 70
\pinlabel {$j'$} [tl] at 315 45
\pinlabel {$j$} [tl] at 315 20
\pinlabel {$\lambda_r$} [tr] at 307 9
\pinlabel {$\lambda_{r+1}$} [tr] at 250 9
\endlabellist
\centering
\includegraphics[scale=.7]{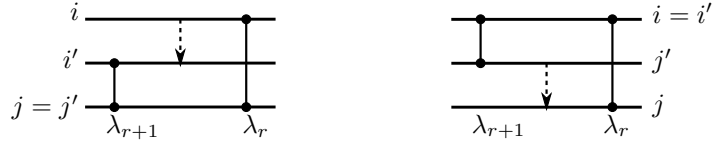}
\caption{The possible exceptional steps in a chord path $\Lambda \in \mathcal{M}^{\sMCS}(a;b)$ where $|b|=|a|-2$.}
\label{f:exceptional-step}
\end{figure}

\begin{remark}
\begin{enumerate}

\item[(i)] In the terminology of Section \ref{rem:GTNotation}, an exceptional step corresponds to an endpoint of the chord $\lambda_r$ jumping along a gradient trajectory in such a way that the length of the chord is decreased.

\item[(ii)]  Proposition \ref{prop:CPDegree} easily generalizes to show that if there is a chord path from $a$ to $b$ with a single exceptional step, then $|b| = |a| - 2$.

\item[(iii)] In the geometric setting of Section~\ref{ch:grad-staircase}, chord paths with a single exceptional step correspond to gradient staircases with one of the intermediate steps, $\varphi_m$, decreasing the index of fiber critical points by $1$.
\end{enumerate}
\end{remark}

\begin{lemma}[Main Structural Lemma] \label{lem:MSL}
Suppose $|b| = |a| - 2$.  Then, there exists a graph $G = (V, E)$ having vertex set
\[ \displaystyle
V = \mathcal{M}^\mathcal{C}(a; b) \cup \left( \bigcup_{|c| = |a| -1} \mathcal{M}^\mathcal{C}(a;c) \times  \mathcal{M}^\mathcal{C}(c;b) \right)
\]
and satisfying the properties
\begin{enumerate}
\item  Each vertex in $\mathcal{M}^\mathcal{C}(a; b)$ is $2$-valent.

\item  Each vertex in $\bigcup_{|c| = |a| -1} \mathcal{M}^\mathcal{C}(a;c) \times  \mathcal{M}^\mathcal{C}(c;b)$ is $1$-valent.

\end{enumerate}
\end{lemma}

Let us pause to describe how Lemma \ref{lem:MSL} completes the proof that $(d_1)^2 = 0$.  Since $V$ is finite, it follows that $G$ is a compact $1$-dimensional manifold with boundary consisting precisely of the vertices in $\bigcup_{|c| = |a| -1} \mathcal{M}^\mathcal{C}(a;c) \times  \mathcal{M}^\mathcal{C}(c;b)$.  Such a manifold is a disjoint union of circles and closed intervals, so it follows that equation~(\ref{eq:dsq1}) is  $0$ modulo $2$.


\subsection{Proof of Lemma \ref{lem:MSL}}

We need to introduce an edge set $E$ for $G$ with the desired properties.  Edges will be determined by analyzing the way the exceptional step of a chord path may be moved around in the front projection of $\sfront$.  For vertices in $\mathcal{M}^\mathcal{C}(a; b)$, the exceptional step is assigned both a  ``left type'' and a ``right type'' depending on the tangle appearing  to the left or, in the latter case, right of the exceptional step and the way in which the chord path passes through this tangle.  An exhaustive list of types for exceptional steps is provided below.  Each type is given an abbreviation reflecting the particular tangle adjacent to the exceptional step which is prefixed with a bold-faced ${\bf L}$ or ${\bf R}$ to indicate whether the tangle appears to the left or right of the exceptional step.  To give a flavor for the notation, an exceptional step situated with a crossing to its left and a handleslide to its right could, for instance, have left type `${\bf L}\mathrm{C1}$' and right type `${\bf R}\mathrm{H3}$'.

The assignment of edges to vertices in $\mathcal{M}^\mathcal{C}(a; b)$ is based on the type of the exceptional step, and is done in such a way that each type provides a single edge.  Since an exceptional step has both a left type and a right type, (1) of Lemma \ref{lem:MSL} will follow.  In addition, (2) will follow as vertices in $\bigcup_{|c| = |a| -1} \mathcal{M}^\mathcal{C}(a;c) \times  \mathcal{M}^\mathcal{C}(c;b)$ will be assigned a single chord path in $\mathcal{M}^\mathcal{C}(a; b)$ by ``gluing'' the two chord paths together near $c$.  

\subsubsection{Description of the edge set $E$}

First, we assign edges to the $1$-valent vertices.  Given 
\[
({\bf \alpha}, {\bf \beta})= \left((\alpha_1, \ldots, \alpha_{n_1}) , (\beta_1, \ldots, \beta_{n_2}) \right) \in  \mathcal{M}^\mathcal{C}(a;c) \times  \mathcal{M}^\mathcal{C}(c;b)
\]
we introduce an edge whose other end lies on the chord path in $\mathcal{M}^\mathcal{C}(a; b)$ obtained from gluing ${\bf \alpha}$ and ${\bf \beta}$ as
\[
(\alpha_1, \ldots, \alpha_{n_1}, \lambda, \beta_1, \ldots, \beta_{n_2}).
\]
The additional chord $\lambda$ is obtained from pushing $\alpha_{n_1}$ through the crossing $c$, with an exceptional step occurring between $\lambda$ and $\beta_1$; see Figure~\ref{f:edge-RC2-broken-chord-paths}.

\begin{figure}[t]
\labellist
\small\hair 2pt
\pinlabel {$\beta_1$} [tl] at 50 9
\pinlabel {$\alpha_n$} [tl] at 235 9
\pinlabel {$\beta_1$} [tl] at 344 9
\pinlabel {$\lambda$} [tl] at 372 9
\pinlabel {$\alpha_n$} [tl] at 425 9
\endlabellist
\centering
\includegraphics[scale=.6]{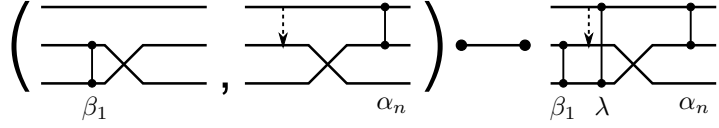}
\caption[]{A broken chord path in $\mathcal{M}^{\sMCS}(a;c) \times \mathcal{M}^{\sMCS}(c;b)$ is connected by an edge to a chord path in $\mathcal{M}^{\sMCS}(a;b)$ with an exceptional step of type RC2. The edge is represented by the horizontal segment in the middle. Later figures represent the edge similarly.}
\label{f:edge-RC2-broken-chord-paths}
\end{figure}

\subsubsection{Classification of exceptional steps into left and right types}

Often the type of a given exceptional step is based on the nature of the differentials in the complexes $(C_m, d_m)$ near the exceptional step.   In particular, for a pair of strands corresponding to generators $e_i, e_j$ in $C_m$ we will frequently need to know whether $\langle d_m e_i, e_j \rangle = 0$ or $1$, i.e. whether strands $i$ and $j$ are connected by a gradient trajectory at $x=x_m$.  (The reader should recall the terminology from Section \ref{rem:GTNotation} which we will now use frequently.)

\medskip

Near {\bf crossings} other than the terminal crossing $b$:

\medskip

\noindent Type $\mathrm{C}1$:  We say that an exceptional step near a crossing is of Type ${\bf L}\mathrm{C}1$ or ${\bf R}\mathrm{C}1$   if there is another chord path in $\mathcal{M}^\mathcal{C}(a; b)$ obtained by moving the exceptional step to the other side of the crossing.  In this case we include an edge in $E$ between these two chord paths; see Figure~\ref{f:edge-LC1-RC1}.  

 As in Section \ref{rem:GTNotation} above, there is a bijection between gradient trajectories before and after the crossing, so the only thing that can prevent us from moving the exceptional step to the other side of the crossing is if the relevant gradient trajectory no longer consists of an interval contained within the chord.  This can happen only if the non-jumping end point of the chord and the end of the gradient trajectory lie on the two crossing strands.  This situation  is impossible if the crossing is to the left of the exceptional step as it would produce a chord running directly into the crossing.  However, it can occur when the crossing is located to the right of the exceptional step, and we say the exceptional step has right type ${\bf R}\mathrm{ C2}$; see Figure~\ref{f:edge-RC2-broken-chord-paths}.  Note that ${\bf R}\mathrm{ C2}$ exceptional steps arise as precisely the vertices connected to broken chord paths in $\mathcal{M}^\mathcal{C}(a;c) \times  \mathcal{M}^\mathcal{C}(c;b)$ by the edges described above.  

In summary, edges in $E$ are assigned to vertices with exceptional steps adjacent to a crossing as follows:
\[
{\bf L}\mathrm{ C1} \leftrightarrow {\bf R}\mathrm{ C1} \quad \mbox{and} \quad {\bf R}\mathrm{ C2} \leftrightarrow \mbox{Broken chord paths}.
\]
See Figures~\ref{f:edge-RC2-broken-chord-paths}~and~\ref{f:edge-LC1-RC1}.

\begin{figure}[t]
\labellist
\small\hair 2pt
\pinlabel {{\bf L}C1} [tl] at 48 6
\pinlabel {{\bf R}C1} [tl] at 208 6
\pinlabel {{\bf L}C1} [tl] at 352 6
\pinlabel {{\bf R}C1} [tl] at 510 6
\endlabellist
\centering
\includegraphics[scale=.6]{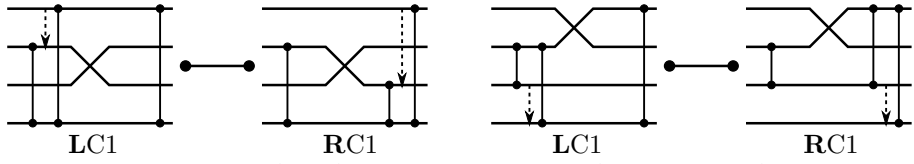}
\caption[]{Edges between vertices with exceptional steps near a crossing.}
\label{f:edge-LC1-RC1}
\end{figure}
 
Near the {\bf Terminal Crossing} $b$: Assuming the exceptional step occurs right before the chord path reaches the crossing $b$, there are two cases to consider.

\begin{figure}[t]
\labellist
\small\hair 2pt
\pinlabel {{\bf L}T1} [tl] at 46 12
\pinlabel {{\bf L}T1} [tl] at 206 12
\pinlabel {(a)} [tl] at 133 12
\pinlabel {{\bf L}T2} [tl] at 348 12
\pinlabel {{\bf L}T2} [tl] at 507 12
\pinlabel {(b)} [tl] at 425 12
\endlabellist
\centering
\includegraphics[scale=.6]{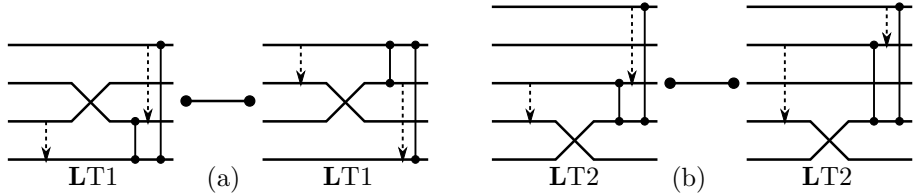}
\caption[]{Edges between vertices with exceptional steps near the terminal crossing.}
\label{f:edge-LT1-LT2}
\end{figure}

\medskip

\noindent Type ${\bf L}\mathrm{T1}$:  The end point of the chord that jumps ends up on one of the crossing strands; see Figure~\ref{f:edge-LT1-LT2}~(a).  According to the requirement (3) from Definition \ref{def:ChordPath} there is a gradient trajectory to the left of the crossing sharing the same non-crossing strand end as the chord.  There then arises another chord path of this type obtained by altering which end of the chord jumps at the exceptional step.

\medskip

\noindent Type ${\bf L}\mathrm{T2}$:  The chord end point which jumps along the exceptional step does not land on either of the crossing strands; see Figure~\ref{f:edge-LT1-LT2}~(b).  Suppose $b$ is located between $x=x_p$ and $x=x_{p+1}$ and involves strands $k$ and $k+1$.  Then, the gradient trajectory constituting the exceptional step combined with the gradient trajectory guaranteed by Definition \ref{def:ChordPath}~(3) provide a broken trajectory at $x= x_{p+1}$ from some $e_i$ to $e_{k+1}$ (or from $e_k$ to some $e_j$).  We assign edges by choosing some division of such broken trajectories into pairs.

In more detail, to provide edges in this case we need to make use of the fact that in the MCS $\mathcal{C}$ the complex $(C_{p+1}, d_{p+1})$ has $d_{p+1}^2=0$.  This implies that for any $i$ (resp. any $j$) there is an even number of once broken gradient trajectories from $i$ to $k+1$  (resp. from $k$ to $j$), so for each $i$ (resp. $j$) we can divide the set of such broken trajectories into pairs.  We assume that such a choice has been fixed once and for all, and this provides edges between pairs of  vertices of type ${\bf L}\mathrm{T2}$.  As a passing remark, in the geometric setting where $\mathcal{C} = \mathcal{C}(F, g)$, the broken trajectories naturally occur in pairs as the boundary points of the compactification of $\mathcal{M}(f_{x_{p+1}}; e_i, e_{k+1})$. 

In summary, we have just described edges in $E$ which connect vertices of type ${\bf L}T1$ and ${\bf L}T2$ with distinct vertices of the same type.  Each vertex occurs as the endpoint of exactly one such edge. 

\medskip

Near a {\bf Handleslide}:  We first classify the exceptional steps into six separate types and then indicate which types are connected by edges in $E$.  The impatient reader can jump to the pictorial presentation given in Figure~\ref{f:edge-H1-to-H6}.

As discussed in Section \ref{rem:GTNotation} above, from Definition \ref{def:MCS}~(5d) it follows that strands which can be connected by a broken trajectory consisting of the handleslide mark and an ordinary gradient trajectory will be connected by a gradient trajectory  on exactly one side of the handleslide.  Furthermore, these are the only pairs of strands where the existence of a gradient trajectory is altered by the handleslide.  More formally, if the handleslide lies between $x=x_m$ and $x= x_{m+1}$ and connects strands $k$ and $l$ then $\langle d_m e_i, e_j \rangle \neq \langle d_{m+1} e_i, e_j \rangle$ if and only if $e_i= e_k$ and $1 = \langle d_m e_l, e_j\rangle = \langle d_{m+1} e_l, e_j \rangle$ or $e_j = e_l$ and $1 = \langle d_m e_i, e_k\rangle = \langle d_{m+1} e_i, e_k \rangle$.

For classifying exceptional steps, an obvious criterion to consider is whether or not the chord path jumps along the handleslide.  Either answer is refined into three separate cases. In the first three types we suppose the chord path does not jump along the handleslide.
 
\medskip

\noindent Type $\mathrm{H}1$: The chord path does not jump along the handleslide and the gradient trajectory involved in the exceptional step exists on both sides of the handleslide.

\medskip

If the gradient trajectory involved in the exceptional step ceases to exist on the other side of the handleslide mark it is due to the existence of a broken trajectory consisting of one part gradient trajectory and one part handleslide.  We refine this case into two cases depending on the order that these portions of the broken trajectory occur.  View the chord path as moving from right to left.

\medskip

\noindent Type $\mathrm{H}2$:  At the exceptional step, the jumping end of the chord passes the gradient trajectory portion and then the handleslide mark.

\medskip

\noindent Type $\mathrm{H}3$:  The jump along the exceptional step passes the handleslide first and then  the gradient trajectory portion.

\medskip

For the final three types we make the assumption that the chord path does jump along the handleslide.  Note that this implies the gradient trajectory involved in the exceptional step will exist on both sides of the handleslide mark.  Indeed, in order for the chord path to jump along both the exceptional step and the handleslide the two markings must lie on non-overlapping vertical intervals.

\medskip

\noindent Type $\mathrm{H}4$:  The jump along the handleslide and the jump at the exceptional step occur at two different ends of the chord.

\medskip

If the jump along the exceptional step and the handleslide occur at a common end of the chords $\lambda_r$, then they together constitute a broken trajectory.  Therefore, a gradient trajectory running along their combined vertical interval appears on exactly one side of the handleslide.

\medskip

\noindent Type $\mathrm{H}5$:  This trajectory exists to the \emph{right} of the handleslide.

\medskip

\noindent Type $\mathrm{H}6$:  This trajectory exists to the \emph{left} of the handleslide.

\medskip

As usual each handleslide type will be prefixed with ${\bf L}$ or ${\bf R}$ to indicate whether the handleslide lies to the left or to the right of the exceptional step.

Edges are assigned as follows between chord paths which are identical away from the handleslide:

\begin{align}
{\bf L}\mathrm{H}1 & \leftrightarrow {\bf R}\mathrm{H}1, \,\, & {\bf L}\mathrm{H}2 & \leftrightarrow {\bf L}\mathrm{H}5, \\ \nonumber
{\bf L}\mathrm{H}3 & \leftrightarrow {\bf R}\mathrm{H}5, \,\, & {\bf L}\mathrm{H}4 & \leftrightarrow {\bf R}\mathrm{H}4, \\ \nonumber
{\bf L}\mathrm{H}6 & \leftrightarrow {\bf R}\mathrm{H}2, \,\, & {\bf R}\mathrm{H}3 & \leftrightarrow {\bf R}\mathrm{H}6.\nonumber
\end{align}

See Figure~\ref{f:edge-H1-to-H6}. 

This completes the description of the edge set.  Each element of $\mathcal{M}^\mathcal{C}(a,b)$ has a unique left and right type, and a single edge assignment has been given based on each possible type.  Thus, condition (1) in the statement of Lemma \ref{lem:MSL} holds.  Condition (2) follows as a single edge was assigned to each ``broken chord path'' in  $\bigcup_{|c| = |a| -1} \mathcal{M}^\mathcal{C}(a;c) \times  \mathcal{M}^\mathcal{C}(c;b)$.  
\begin{flushright}
$\Box$
\end{flushright}


\begin{figure}[t]
\labellist
\small\hair 2pt
\pinlabel {{\bf L}H1} [tl] at 40 238
\pinlabel {{\bf R}H1} [tl] at 200 238
\pinlabel {{\bf L}H2} [tl] at 360 238
\pinlabel {{\bf L}H5} [tl] at 520 238

\pinlabel {{\bf L}H3} [tl] at 40 124
\pinlabel {{\bf R}H5} [tl] at 200 124
\pinlabel {{\bf L}H4} [tl] at 360 124
\pinlabel {{\bf R}H4} [tl] at 520 124

\pinlabel {{\bf L}H6} [tl] at 40 9
\pinlabel {{\bf R}H2} [tl] at 200 9
\pinlabel {{\bf R}H3} [tl] at 360 9
\pinlabel {{\bf R}H6} [tl] at 520 9
\endlabellist
\centering
\includegraphics[scale=.55]{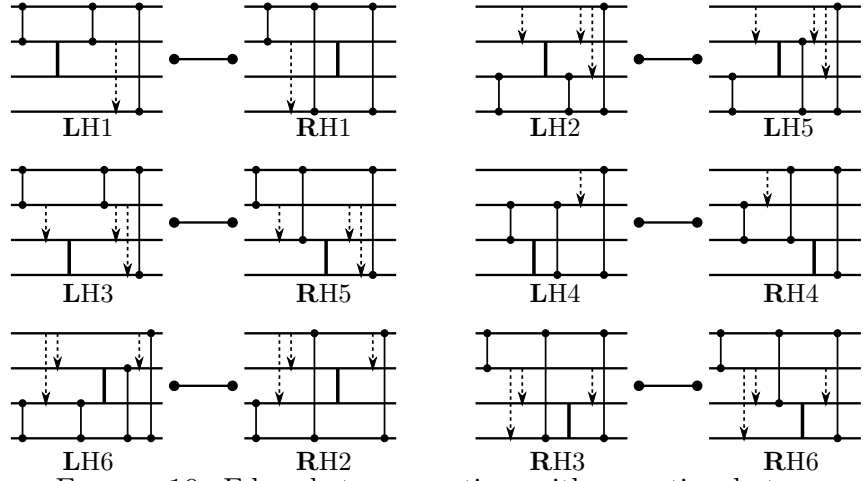}
\caption[]{Edges between vertices with exceptional steps near a handleslide.}
\label{f:edge-H1-to-H6}
\end{figure}

\subsection{Proof of Theorem \ref{thm:d2is0}}

To see that the full differential has $d^2 = 0$, we need to show that for any $a, b_1, \ldots, b_n \in Q(L)$ with $|b_1\cdots b_n| = |a| -2$, $\langle d^2 a, b_1\cdots b_n \rangle = 0$.  From the Liebniz rule, we have
\[
\langle d^2 a , b_1 \cdots b_n \rangle = \# \bigcup_{|c| = |b_i\cdots b_{j}| +1}   \mathcal{M}^\mathcal{C}(a; b_1, \ldots, b_{i-1}, c, b_{j+1}, b_n) \times \mathcal{M}^\mathcal{C}(c; b_i, b_{i+1}, \ldots, b_{j}).
\]
We use the same strategy as in the previous subsection to show the right hand side is $0$ mod $2$.  In this case, the ``$1$-dimensional moduli space'' $\mathcal{M}^\mathcal{C}(a; b_1, \ldots, b_{n})$ has a slightly more elaborate description.  Along with moving the exceptional step around there is a possibility of varying the location of a branch point in a difference flow tree.  The combinatorial description follows.

\begin{figure}[t]
\labellist
\small\hair 2pt
\pinlabel {$i$} [tr] at 7 67
\pinlabel {$l$} [tr] at 7 44
\pinlabel {$j$} [tr] at 7 19
\pinlabel {$\lambda_r$} [tr] at 110 6
\pinlabel {$\beta_1$} [tr] at 53 6
\pinlabel {$\alpha_1$} [tr] at 35 6
\endlabellist
\centering
\includegraphics[scale=.7]{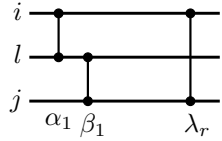}
\caption{Exceptional branching in a chord path $\Lambda \in \mathcal{M}^{\sMCS}(a;b_1, \hdots, b_n)$.}
\label{f:exceptional-branching}
\end{figure}

\begin{definition} \label{def:1dMS} When $|b_1\cdots b_n| = |a| -2$, the elements of $\mathcal{M}^\mathcal{C}(a; b_1, \ldots, b_{n})$ are chord paths possibly with convex corners with word $b_1\cdots b_n$ except that exactly one of the two following exceptional features occurs

\begin{enumerate}
\item A single exceptional step, or

\item At some $x= x_p$ the chord $\lambda_r = (x_p, [i,j])$ splits into two chords $\alpha_1$ and $\beta_1$ so that for some $i< l < j$,
$\alpha_1 = (x_p, [i,l])$ and $\beta_1 = (x_p, [l,j])$.  Subsequently, $\alpha_1$ and $\beta_1$ are extended to chord paths with convex corners $A= (\alpha_1, \ldots, \alpha_{M_1})$ and $B = (\beta_1, \ldots, \beta_{M_2})$ satisfying all of the requirements of Definition \ref{def:ChordPath} except for (2).  We refer to the part of the chord path where $\lambda_r$ splits into $\alpha_1$ and $\beta_1$ as the {\it exceptional branching}; see Figure~\ref{f:exceptional-branching}.

\end{enumerate}

In the second case, the associated word $w(\Lambda)$ is defined as the following product: Let $f_1, \ldots, f_s$ (resp. $h_1, \ldots, h_t$) denote the convex corners occurring along the top (resp. bottom) ends of the $\lambda_i$ read from right to left (resp. left to right).  Then,
\[
w(\Lambda) = (f_1\cdots f_s) w(A)\cdot w(B) (h_1\cdots h_t).
\]
\end{definition} 

The proof that $d^2 =0$ is completed by an analog of Lemma \ref{lem:MSL} with the vertex set now consisting of the union of $\mathcal{M}^\mathcal{C}(a; b_1, \ldots, b_{n})$ and $$\bigcup_{|c| = |b_i\cdots b_{j}| +1}   \mathcal{M}^\mathcal{C}(a; b_1, \ldots, b_{i-1}, c, b_{j+1}, b_n) \times \mathcal{M}^\mathcal{C}(c; b_i, b_{i+1}, \ldots, b_{j})$$ with vertices from the former subset $2$-valent and vertices from the latter subset $1$-valent.

We keep all edges arising from exceptional steps positioned near crossings and handleslides as in the proof of Lemma \ref{lem:MSL}.  In addition, we now describe edges arising from

\begin{enumerate}

\item Broken chord paths in $\mathcal{M}^\mathcal{C}(a; b_1, \ldots, b_{i-1}, c, b_{j+1}, b_n) \times \mathcal{M}^\mathcal{C}(c; b_i, b_{i+1}, \ldots, b_{j})$ with $c$ a convex corner.

\item Exceptional branchings.

\item Exceptional steps positioned near convex corners.

\end{enumerate}

For the first case, notice that such a pair of chord paths $(\Lambda_1, \Lambda_2)$ may be ``glued'' together so that an exceptional branching occurs immediately to the left of $c$.  In the notation of Definition \ref{def:1dMS}, $A$ and $B$ (resp. $B$ and $A$) will consist of $\Lambda_2$ and the tail end of $\Lambda_1$ if the convex corner of $\Lambda_1$ at $c$ occurs at the upper end point (resp. lower end point) of the chord.  Note, that in either case the word of the glued chord path is indeed $b_1\cdots b_n$.  See Figure~\ref{f:edge-RBC2-broken-chord-paths-and-LBC2}.

\subsubsection{Edges arising from exceptional branchings.}

\bigskip

\noindent Near {\bf crossings}:

\medskip

\begin{figure}[t]
\labellist
\small\hair 2pt
\pinlabel {{\bf R}BC2} [tl] at 365 9
\pinlabel {{\bf L}BC2} [tl] at 510 9
\pinlabel {} [tl] at 80 9
\endlabellist
\centering
\includegraphics[scale=.55]{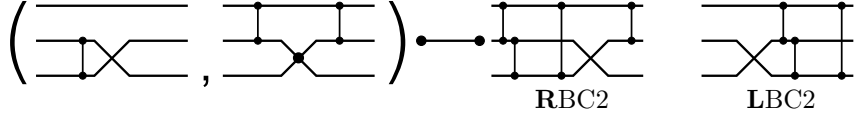}
\caption[]{A chord path with an exceptional branching of type RBC2 arises from gluing at a convex corner. A chord path with exceptional branching of type LBC2 cannot occur.}
\label{f:edge-RBC2-broken-chord-paths-and-LBC2}
\end{figure}

\noindent Type $\mathrm{BC}1$:  Either the branch point lies on a non-crossing strand, or prior to the branching both of the edges of the chord lie on non-crossing strands. 

\medskip 

In this case, it is possible to the simply move the branching point to the other side of the crossing.  Thus, we can assign edges ${\bf L}\mathrm{BC}1 \leftrightarrow {\bf R}\mathrm{BC}1$; see Figure~\ref{f:edge-LBC1-RBC1}.

\medskip

\begin{figure}[t]
\labellist
\small\hair 2pt
\pinlabel {{\bf L}BC1} [tl] at 45 6
\pinlabel {{\bf R}BC1} [tl] at 200 6
\pinlabel {{\bf L}BC1} [tl] at 350 6
\pinlabel {{\bf R}BC1} [tl] at 510 6
\endlabellist
\centering
\includegraphics[scale=.55]{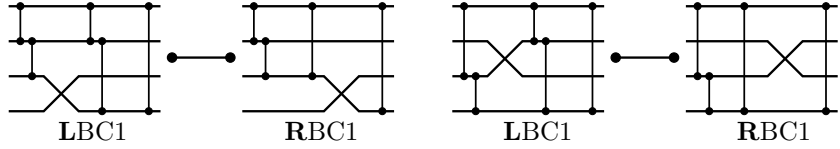}
\caption[]{Edges between vertices with exceptional branchings near a crossing.}
\label{f:edge-LBC1-RBC1}
\end{figure}

\noindent Type $\mathrm{BC}2$:  Before the exceptional branching one end of the chord lies on one of the crossing strands, and the break point of the branching also lies on one of the crossing strands.   

\medskip

It is easy to see that one of the cases BC1 and BC2 must occur.  Also, note that ${\bf L}\mathrm{BC}2$   cannot occur as it would produce a chord running directly into a crossing.  The exceptional branchings of type ${\bf R}\mathrm{BC}2$ are precisely the chord paths that arise from gluing at a convex corner $c$; see Figure~\ref{f:edge-RBC2-broken-chord-paths-and-LBC2}.
\[
{\bf R}\mathrm{BC}2 \leftrightarrow \mbox{Broken chord path with the break occurring at a convex corner.}
\]

\bigskip

\noindent Near a {\bf convex corner}:

\medskip

\noindent Type $\mathrm{BX}1$:   The convex corner does not involve one of the endpoints which is newly formed by the branching.  

\medskip

If the convex corner does involve one of the new endpoints then it must lie to the left of the exceptional branching.  We subdivide this case into two types.

\medskip

\noindent Type ${\bf L}\mathrm{BX}2$ (resp. Type ${\bf L}\mathrm{BX}3$) :  The chord breaks at the bottom (resp. top) of the two crossing strands.  

\medskip

Edges can be assigned as ${\bf L}\mathrm{BX}1 \leftrightarrow {\bf R}\mathrm{BX}1$ and ${\bf L}\mathrm{BX}2 \leftrightarrow {\bf L}\mathrm{BX}3$; see Figure~\ref{f:edge-LBX1-RBX1-LBX2-LBX3}.  Note that exceptional branchings of type ${\bf R}\mathrm{BX}2$ and ${\bf R}\mathrm{BX}2$ cannot occur.  

\begin{figure}[t]
\labellist
\small\hair 2pt
\pinlabel {{\bf L}BX1} [tl] at 35 6
\pinlabel {{\bf R}BX1} [tl] at 191 6
\pinlabel {{\bf L}BX2} [tl] at 357 6
\pinlabel {{\bf L}BX3} [tl] at 516 6
\endlabellist
\centering
\includegraphics[scale=.55]{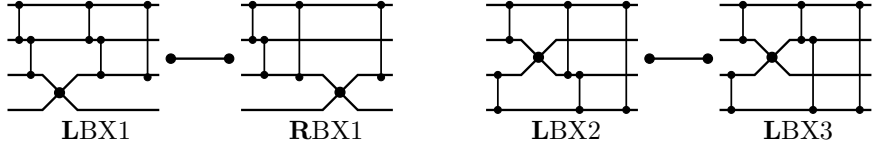}
\caption[]{Edges between vertices with exceptional branchings near a convex corner.}
\label{f:edge-LBX1-RBX1-LBX2-LBX3}
\end{figure}


\noindent Near a {\bf handleslide}:  

\medskip

\noindent Type $\mathrm{BH}1$:  None of the chords jump along the handleslide.

\medskip

\noindent Type $\mathrm{BH}2$:  One of the chords jumps, but it is along an endpoint other than the two new ones created by the exceptional branching.  (That is, other than the bottom end of $\alpha_1$ and the top end of $\beta_1$.)

\medskip

\noindent Type ${\bf L}\mathrm{BH}3$: Bottom end of $\alpha_1$ jumps.

\medskip

\noindent Type ${\bf L}\mathrm{BH}4$: Top end of $\beta_1$ jumps.

\medskip

Clearly, there are not analogous cases ${\bf R}\mathrm{BH}3$ or ${\bf R}\mathrm{BH}4$ to consider.  The cases listed above are exhaustive, and we assign edges as

\begin{align}
{\bf L}\mathrm{BH}1 & \leftrightarrow {\bf R}\mathrm{BH}1, \,\, & {\bf L}\mathrm{BH}2 & \leftrightarrow {\bf R}\mathrm{BH}2, \\ \nonumber
{\bf L}\mathrm{BH}3 & \leftrightarrow {\bf L}\mathrm{BH}4. & & \nonumber
\end{align}

See Figure~\ref{f:edge-BH1-to-BH4}.

\begin{figure}[t]
\labellist
\small\hair 2pt
\pinlabel {{\bf L}BH1} [tl] at 35 128
\pinlabel {{\bf R}BH1} [tl] at 190 128
\pinlabel {{\bf L}BH2} [tl] at 348 128
\pinlabel {{\bf R}BH2} [tl] at 507 128

\pinlabel {{\bf L}H3} [tl] at 37 9
\pinlabel {{\bf L}H4} [tl] at 197 9
\endlabellist
\centering
\includegraphics[scale=.55]{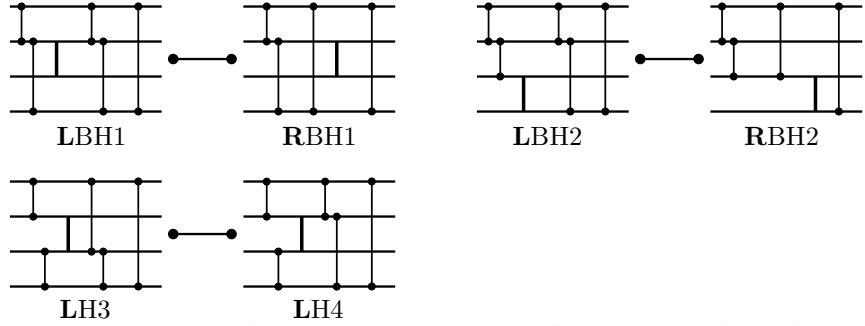}
\caption[]{Edges between vertices with exceptional branchings near a handleslide.}
\label{f:edge-BH1-to-BH4}
\end{figure}

\subsubsection{Exceptional steps near convex corners and branching near terminal crossings.}

A few types of edges may interchange the exceptional steps and branching.  Such an edge will necessarily exchange a convex corner for a terminal crossing as well.

\bigskip

\noindent
Exceptional steps near a {\bf convex corner}:

\medskip

\noindent Type $\mathrm{X1}$:  The exceptional step occurs between two non-crossing strands.

\medskip

\noindent Type ${\bf L}\mathrm{X2}$:  The chord which jumps at the exceptional step lands on one of the crossing strands.

\medskip

\begin{remark} (i) In this case the convex corner has to appear to the left of the exceptional step.  

(ii) Prior to the jump the chord must have its two ends on non-crossing strands and stretch over a vertical interval containing the $z$-coordinate of the crossing.
\end{remark}

\noindent Type ${\bf R}\mathrm{X3}$:  The chord which jumps from a crossing strand to a non-crossing strand during the exceptional step.  (This implies the convex corner lies to the right of the exceptional step.)

\medskip

Exceptional branching near a {\bf terminal crossing}:

\medskip

\noindent Type ${\bf L}\mathrm{BT1}$:  The branching breaks the chord at one of the crossing strands.

\medskip

\noindent Type ${\bf L}\mathrm{BT2}$:  The branching breaks the chord at a non-crossing strand.

\medskip

Edges can be assigned as

\begin{align}
{\bf L}\mathrm{X1} & \leftrightarrow {\bf R}\mathrm{X1}, \,\, & {\bf L}\mathrm{X2} & \leftrightarrow {\bf L}\mathrm{BT1}, \\ \nonumber
{\bf R}\mathrm{X2} & \leftrightarrow {\bf L}\mathrm{BT2}. & & \nonumber
\end{align}

\begin{figure}[t]
\labellist
\small\hair 2pt
\pinlabel {{\bf L}X1} [tl] at 40 130
\pinlabel {{\bf R}X1} [tl] at 195 130
\pinlabel {{\bf L}X2} [tl] at 355 130
\pinlabel {{\bf L}BT2} [tl] at 515 130

\pinlabel {{\bf R}X3} [tl] at 40 10
\pinlabel {{\bf L}BT2} [tl] at 193 10
\endlabellist
\centering
\includegraphics[scale=.55]{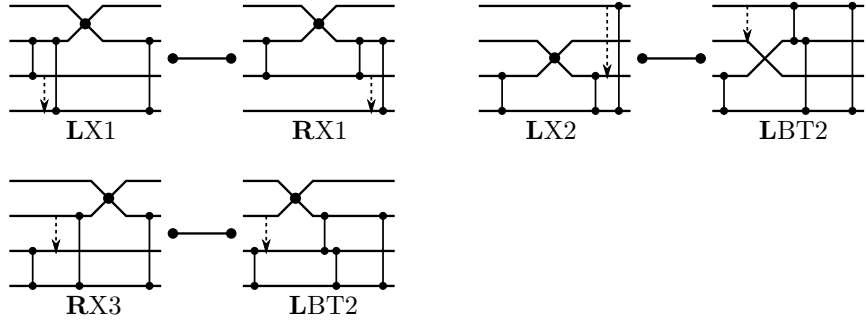}
\caption[]{Edges between vertices with exceptional branchings near the terminal crossing.}
\label{f:edge-X1-LX2-LBT1-RX2-LBT2}
\end{figure}

See Figure~\ref{f:edge-X1-LX2-LBT1-RX2-LBT2}. This completes the edge set $E$, and the proof that $d^2=0$.
\begin{flushright}
$\Box$
\end{flushright}

\subsection{Equivalent MCSs induce isomorphic linearized DGAs} 
\mylabel{sec:Linear-inv}
We prove Theorem~\ref{thm:linear-iso} by considering each MCS move depicted in Figures~\ref{f:MCS-equiv}~and~\ref{f:MCS-explosion}. The proof of Theorem~\ref{thm:linear-iso} for an MCS move obtained by reflecting a move depicted in Figures~\ref{f:MCS-equiv}~and~\ref{f:MCS-explosion} is similar. In the case of MCS moves 1-7 and 9-14, the chain isomorphism is the identity map. We require a ``combinatorial gluing argument'' similar to the proof of $d^2=0$ in order to construct the chain isomorphism for MCS moves 8 and 15. 

\begin{proof}[Proof of Theorem~\ref{thm:linear-iso}]
We need only consider the case when $\sMCS^{+}$ is equivalent to $\sMCS^{-}$ by a single MCS move. Given a fixed MCS move, we construct a chain isomorphism $\Phi : (A_{\sMCS^{+}}, d^{+}_1) \to (A_{\sMCS^{-}}, d^{-}_1)$ by specifying its action on the generating set $Q(\sfront)$. Recall from Section~\ref{sec:MCS-DGA-defn}, given $a \in Q(\sfront)$, 
\begin{align*}
&d^{+}_1 a = \sum_{b \in Q(L)} \# \mathcal{M}^{\sMCS^{+}}(a; b) b \quad \mbox{ and } \quad d^{-}_1 a = \sum_{b \in Q(L)} \# \mathcal{M}^{\sMCS^{-}}(a; b) b. \nonumber 
\end{align*}
where $\mathcal{M}^{\sMCS^{+}}(a; b)$ and $\mathcal{M}^{\sMCS^{-}}(a; b)$ are the sets of chord paths (without convex corners) from $a$ to $b$ in $\sMCS^{+}$ and $\sMCS^{-}$ respectively. 

We begin with a few notational conventions and simplifying observations. Regardless of the MCS move under consideration, we will let $T \subset \sfront$ be the tangle in a small neighborhood of the move depicted in Figures~\ref{f:MCS-equiv}~and~\ref{f:MCS-explosion}. For each MCS move, we will assume $\sMCS^{+}$ is to the left of the double arrow in Figures~\ref{f:MCS-equiv}~and~\ref{f:MCS-explosion} and $\sMCS^{-}$ is to the right. If $\sMCS^{+}$ and $\sMCS^{-}$ contain one handleslide, then it is labeled $\alpha$. If $\sMCS^{+}$ (resp. $\sMCS^{-}$) contains exactly two handleslides, then they are labeled $\alpha$ and $\beta$ from left to right (resp. right to left). In MCS move 6, the handleslides in $\sMCS^{-}$ are labeled $\beta$, $\gamma$, $\alpha$ from left to right.

Chord paths beginning at $a$ progress to the left. Therefore, if $a$ is left of $T$ or contained in $T$, then $\mathcal{M}^{\sMCS^{+}}(a; b) = \mathcal{M}^{\sMCS^{-}}(a; b)$ for all $b \in Q(\sfront)$ and so $d^{+}_1 a = d^{-}_1 a$. Hence, independent of the MCS move under consideration, we will define $\Phi(a) = a$ for each such $a \in Q(\sfront)$. If $a$ and $b$ are both right of $T$ then, again, $\mathcal{M}^{\sMCS^{+}}(a; b) = \mathcal{M}^{\sMCS^{-}}(a; b)$. Thus, we need only consider $a$ to the right of $T$ and $b$ contained in $T$ or to the left of $T$. Given such a pair $a$ and $b$, and $i \geq 0$, let $\mathcal{M}^{\sMCS^{+}}(a; b, i) \subset \mathcal{M}^{\sMCS^{+}}(a; b)$ be chord paths that jump along exactly $i$ handleslide marks of $\sMCS^{+}$ in $T$. The set $\mathcal{M}^{\sMCS^{-}}(a; b, i)$ is similarly defined. If $b$ is left of $T$ and $\Lambda \in \mathcal{M}^{\sMCS^{+}}(a; b)$ passes through $T$ without jumping along a handleslide in $T$, then there is a corresponding chord path $\Lambda' \in  \mathcal{M}^{\sMCS^{-}}(a; b)$ that agrees with $\Lambda$ away from $T$ and does not jump along a handleslide in $T$. In fact, this gives us a bijection between chord paths in $\mathcal{M}^{\sMCS^{+}}(a; b, 0)$ and $\mathcal{M}^{\sMCS^{-}}(a; b, 0)$. Thus, we need only consider $\mathcal{M}^{\sMCS^{+}}(a; b, i)$ and $\mathcal{M}^{\sMCS^{-}}(a; b, i)$ with $i > 0$, except in moves 7, 8, and 10 where we must also consider the case where $b$ is in $T$ and $i=0$. We will prove the following claim for MCS moves 1-7 and 9-14. MCS moves 8 and 15 will be handled separately. 

\begin{claim}
\label{claim:iso-1}
Suppose $\sMCS^{+}$ is equivalent to $\sMCS^{-}$ by one of MCS moves 1-7 or 9-14. Then, for all $a, b \in Q(\sfront)$, $\# \mathcal{M}^{\sMCS^{+}}(a; b) = \# \mathcal{M}^{\sMCS^{-}}(a; b) \mod 2$, hence $d^{+}_1 = d^{-}_1$ and $\Phi = Id$ is a chain isomorphism. In fact, for MCS moves 2-5 and 9-14 there is a bijection from $\mathcal{M}^{\sMCS^{+}}(a; b)$ to $\mathcal{M}^{\sMCS^{-}}(a; b)$.
\end{claim}

\textbf{MCS Move 1:} Note that $\mathcal{M}^{\sMCS^{+}}(a; b, i) = \emptyset$ if $i > 0$ and $\mathcal{M}^{\sMCS^{-}}(a; b, i) = \emptyset$ if $i > 1$. Thus, in order to prove Claim~\ref{claim:iso-1}, it remains to show $\# \mathcal{M}^{\sMCS^{-}}(a; b, 1) = 0$. But there is a natural pairing of elements in $\mathcal{M}^{\sMCS^{-}}(a; b, 1)$. If $\Lambda' \in \mathcal{M}^{\sMCS^{-}}(a; b, 1)$ jumps along $\alpha$ then there is a chord path in $\mathcal{M}^{\sMCS^{-}}(a; b, 1)$ that agrees with $\Lambda'$ away from $T$ and jumps along $\beta$. Hence, Claim~\ref{claim:iso-1} follows for move 1. 

\textbf{MCS Moves 2-5, 9, and 11-14:} Since $\sfront$ is nearly plat, if a chord in a chord path appears to the immediate left or right of a cusp, then the chord must have endpoints on consecutive strands of $\sfront$. A chord between consecutive strands can not jump along a handleslide. Thus, in the case of moves 9 and 11 - 14, we find $\mathcal{M}^{\sMCS^{\pm}}(a; b, i) = \emptyset$ if $i > 0$. Hence, Claim~\ref{claim:iso-1} follows for moves 9 and 11 - 14 from our previous discussion. For moves 2-5, the arrangements of handleslide marks in $\sMCS^{+}$ and $\sMCS^{-}$ ensure $\mathcal{M}^{\sMCS^{\pm}}(a; b, i) = \emptyset$ if $i > 1$. If $\Lambda \in \mathcal{M}^{\sMCS^{+}}(a; b, 1)$ jumps along $\alpha$ (resp. $\beta$), then there is a corresponding $\Lambda' \in  \mathcal{M}^{\sMCS^{-}}(a; b, 1)$ that agrees with $\Lambda$ away from $T$ and jumps along $\alpha$ (resp. $\beta$). In fact, this gives us a bijection between chord paths in $\mathcal{M}^{\sMCS^{+}}(a; b, 1)$ and $\mathcal{M}^{\sMCS^{-}}(a; b, 1)$. Hence, Claim~\ref{claim:iso-1} follows for moves 2-5. 

\textbf{MCS Move 6:} Note $\mathcal{M}^{\sMCS^{-}}(a; b, i) = \emptyset$ if $i > 2$. If $\Lambda \in \mathcal{M}^{\sMCS^{+}}(a; b)$ jumps along $\alpha$ (resp. $\beta$), then there is a corresponding $\Lambda' \in  \mathcal{M}^{\sMCS^{-}}(a; b)$ that agrees with $\Lambda$ away from $T$ and jumps along $\alpha$ (resp. $\beta$). In fact, this gives an injection from $\mathcal{M}^{\sMCS^{+}}(a; b, 1)$ into $\mathcal{M}^{\sMCS^{-}}(a; b, 1)$. A chord path $\Lambda \in \mathcal{M}^{\sMCS^{+}}(a; b, 2)$ jumping along both $\alpha$ and $\beta$ must do so by jumping upward along $\alpha$ and $\beta$. Such a chord path corresponds to a chord path in $\mathcal{M}^{\sMCS^{-}}(a; b,1)$ that jumps upward along $\gamma$. Note that a chord path in $\mathcal{M}^{\sMCS^{-}}(a; b, 2)$ can not jump upward along any of $\alpha$, $\beta$, or $\gamma$. Hence, it remains to show that the subset of $\mathcal{M}^{\sMCS^{-}}(a; b, 1) \cup \mathcal{M}^{\sMCS^{-}}(a; b, 2)$ containing chord paths that either jump downward along $\gamma$ or jump downward along both $\alpha$ and $\beta$ contains an even number of elements. As in the case of move 1, there is a natural pairing on this set. If $\Lambda' \in \mathcal{M}^{\sMCS^{-}}(a; b)$ jumps downward along $\gamma$, then there is a chord path in $\Lambda' \in \mathcal{M}^{\sMCS^{-}}(a; b)$ that agrees with $\Lambda'$ way from $T$ and jumps downward along both $\alpha$ and $\beta$. Hence, Claim~\ref{claim:iso-1} follows for move 6.

\begin{figure}[t]
\labellist
\small\hair 2pt
\pinlabel {1} [br] at 161 284
\pinlabel {2} [br] at 463 284
\pinlabel {3} [br] at 161 161
\pinlabel {4} [br] at 463 161
\pinlabel {5} [br] at 161 52
\pinlabel {6} [br] at 463 52
\endlabellist
\centering
\includegraphics[scale=.55]{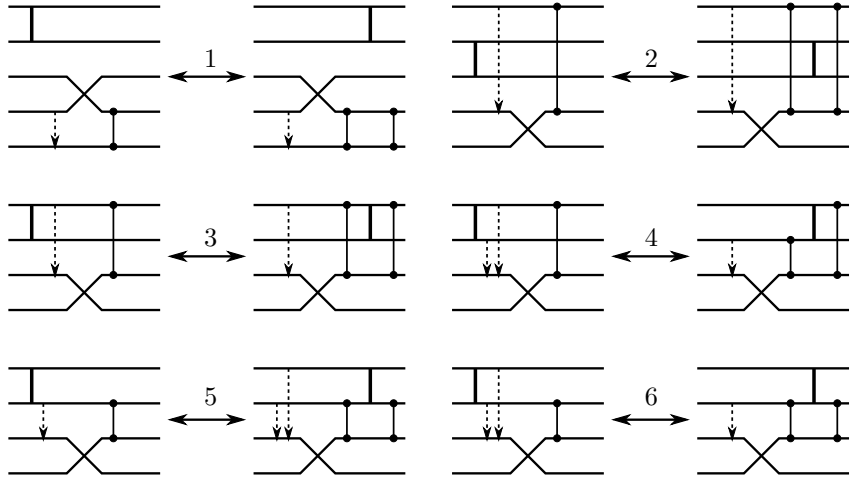}
\caption[]{The bijections relating $\mathcal{M}^{\sMCS^{+}}(a; b)$ and $ \mathcal{M}^{\sMCS^{-}}(a; b)$ in MCS move 7.}
\label{f:move-7}
\end{figure}

\begin{figure}[t]
\labellist
\small\hair 2pt
\pinlabel {7} [br] at 158 50
\endlabellist
\centering
\includegraphics[scale=.55]{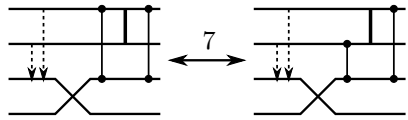}
\caption[]{The bijection between $\mmc^{\sMCS^{-}}(j)$ and $\mmc^{\sMCS^{-}}_{\alpha}(l)$ in the case $(\langle d_p e_j, e_k \rangle, \langle d_p e_l, e_k \rangle)=(1,1)$.}
\label{f:move-7-2}
\end{figure}

\textbf{MCS Move 7:} Suppose $\alpha$ has endpoints on strands $j$ and $l$, $j < l$. 

Suppose $b$ is to the left of $T$. Note that $\mathcal{M}^{\sMCS^{\pm}}(a; b, i) = \emptyset$ if $i > 1$. Following the argument above for moves 2-5, there is a bijection between the chord paths in $\mathcal{M}^{\sMCS^{+}}(a; b, 1)$ and $\mathcal{M}^{\sMCS^{-}}(a; b, 1)$. Hence, Claim~\ref{claim:iso-1} follows when $b$ is to the left of $T$. 

Suppose $b$ is the crossing in $T$. Given $m > 0$ let $\mmc^{\sMCS^{+}}(m) \subset \mmc^{\sMCS^{+}}(a;b)$ be the chord paths with a terminating chord having its upper endpoint on strand $m$ and define the subset $\mmc^{\sMCS^{-}}(m) \subset \mmc^{\sMCS^{-}}(a;b)$ similarly. Note $\mmc^{\sMCS^{+}}(a;b) = \mmc^{\sMCS^{+}}(k+1) \cup \left [\bigcup_{m < k} \mmc^{\sMCS^{+}}(m) \right ]$ and $\mmc^{\sMCS^{-}}(a;b) = \mmc^{\sMCS^{-}}(k+1) \cup \left [\bigcup_{m < k} \mmc^{\sMCS^{-}}(m) \right ]$. There is a bijection between $\mmc^{\sMCS^{+}}(m)$ and $\mmc^{\sMCS^{-}}(m)$ when $m = k+1$ or $m < k$ and $m \notin \{j, l \}$. The bijection is described in Figure~\ref{f:move-7} (1) for $m = k+1$ and in Figure~\ref{f:move-7} (2) for $m < k$ and $m \notin \{j, l \}$. Hence, we are left to show: 
\begin{equation}
\label{eq:move-7}
\# ( \mmc^{\sMCS^{+}}(j) \cup \mmc^{\sMCS^{+}}(l) ) = \# ( \mmc^{\sMCS^{-}}(j) \cup \mmc^{\sMCS^{-}}(l) ).
\end{equation}

Let $\mmc^{\sMCS^{-}}_{\alpha}(l) \subset \mmc^{\sMCS^{-}}(l)$ be the chord paths that jump along $\alpha$ in $T$ and let $(\mmc^{\sMCS^{-}}_{\alpha}(l))^{c}$ be the complement of $\mmc^{\sMCS^{-}}_{\alpha}(l)$ in $\mmc^{\sMCS^{-}}(l)$. Consider the possible values of the pair $(\langle d_p e_j, e_k \rangle, \langle d_p e_l, e_k \rangle)$ where $(C_p, d_p)$ is the chain complex of $\sMCS^{+}$ between the handleslide $\alpha$ and the crossing $b$. For each value of $(\langle d_p e_j, e_k \rangle, \langle d_p e_l, e_k \rangle)$, we list below bijections and give figure references for their descriptions. In each case, equation~(\ref{eq:move-7}) and, hence, Claim~\ref{claim:iso-1} follow immediately.
\begin{align}
&(0,0): \mmc^{\sMCS^{\pm}}(l)=\mmc^{\sMCS^{\pm}}(j)=\emptyset; \nonumber \\
&(1,0): \mmc^{\sMCS^{\pm}}(l)=\emptyset, \mmc^{\sMCS^{+}}(j) \Leftrightarrow \mmc^{\sMCS^{-}}(j) \mbox{ [\ref{f:move-7} (3)]};  \nonumber \\
&(1,1): \mmc^{\sMCS^{-}}(j)=\emptyset, \mmc^{\sMCS^{+}}(j) \Leftrightarrow \mmc^{\sMCS^{-}}_{\alpha}(l) \mbox{ [\ref{f:move-7} (4)]}, \mmc^{\sMCS^{+}}(l) \Leftrightarrow (\mmc^{\sMCS^{-}}_{\alpha}(l))^{c} \mbox{ [\ref{f:move-7} (6)]};  \nonumber \\
&(0,1): \mmc^{\sMCS^{+}}(j)=\emptyset, \mmc^{\sMCS^{+}}(l) \Leftrightarrow (\mmc^{\sMCS^{-}}_{\alpha}(l))^{c} \mbox{ [\ref{f:move-7} (5)]}, \mmc^{\sMCS^{-}}(j) \Leftrightarrow \mmc^{\sMCS^{-}}_{\alpha}(l) \mbox{ [\ref{f:move-7-2} (7)]}. \nonumber 
\end{align}

\textbf{MCS Move 10:} If $b$ is not the crossing in $T$, then Claim~\ref{claim:iso-1} follows as in the proof of move 7. Suppose $b$ is the crossing in $T$ between strands $k$ and $k+1$. A chord path in $\sMCS^{-}$ originating at $a$ cannot terminate at $b$ by jumping along $\alpha$. Hence, $\mathcal{M}^{\sMCS^{+}}(a; b, i) = \mathcal{M}^{\sMCS^{-}}(a; b, i) = \emptyset$ if $i > 0$. Suppose $(C_p, d_p)$ and $(C_{p+1}, d_{p+1})$ are the chain complexes of $\sMCS^{+}$ on either side of $\alpha$. Then, $\langle d_p e_i, e_k \rangle = \langle d_{p+1} e_i, e_k \rangle$ and $\langle d_p e_{k+1}, e_j \rangle = \langle d_{p+1} e_{k+1}, e_j \rangle$ for all $i < k$ and $k+1 < j$. Thus, for each chord path terminating at $b$ in $\sMCS^{+}$ there is a corresponding chord path of $\sMCS^{-}$ terminating at $b$; and vice versa. Therefore, $\mathcal{M}^{\sMCS^{+}}(a; b, 0)=\mathcal{M}^{\sMCS^{-}}(a; b, 0)$ and Claim 1 follows. 

In the case of the MCS move which results from reflecting MCS move 8 in Figure~\ref{f:MCS-equiv} about a vertical axis, the chain isomorphism $\Phi$ is still the identity map. In the case of MCS move 15 and the version of MCS move 8 depicted in Figure~\ref{f:MCS-equiv}, the chain isomorphism $\Phi$ may not be the identity map and, in fact, we require considerably more machinery to prove the theorem in these cases. The proofs for these cases require similar notation and machinery, so we will describe much of it in the next few paragraphs and then give the proof for each move.

\begin{figure}[t]
\labellist
\small\hair 2pt
\pinlabel {$l$} [br] at 18 193
\pinlabel {$m$} [br] at 18 171
\pinlabel {$m+1$} [br] at 18 144
\pinlabel {$\alpha$} [tl] at 40 169
\pinlabel {$\beta$} [tl] at 233 151
\pinlabel {$\alpha$} [tl] at 433 146
\pinlabel {$\alpha$} [tl] at 236 57
\pinlabel {$q_k$} [tl] at 65 149
\pinlabel {$l$} [br] at 18 80
\pinlabel {$m$} [br] at 18 57
\endlabellist
\centering
\includegraphics[scale=.55]{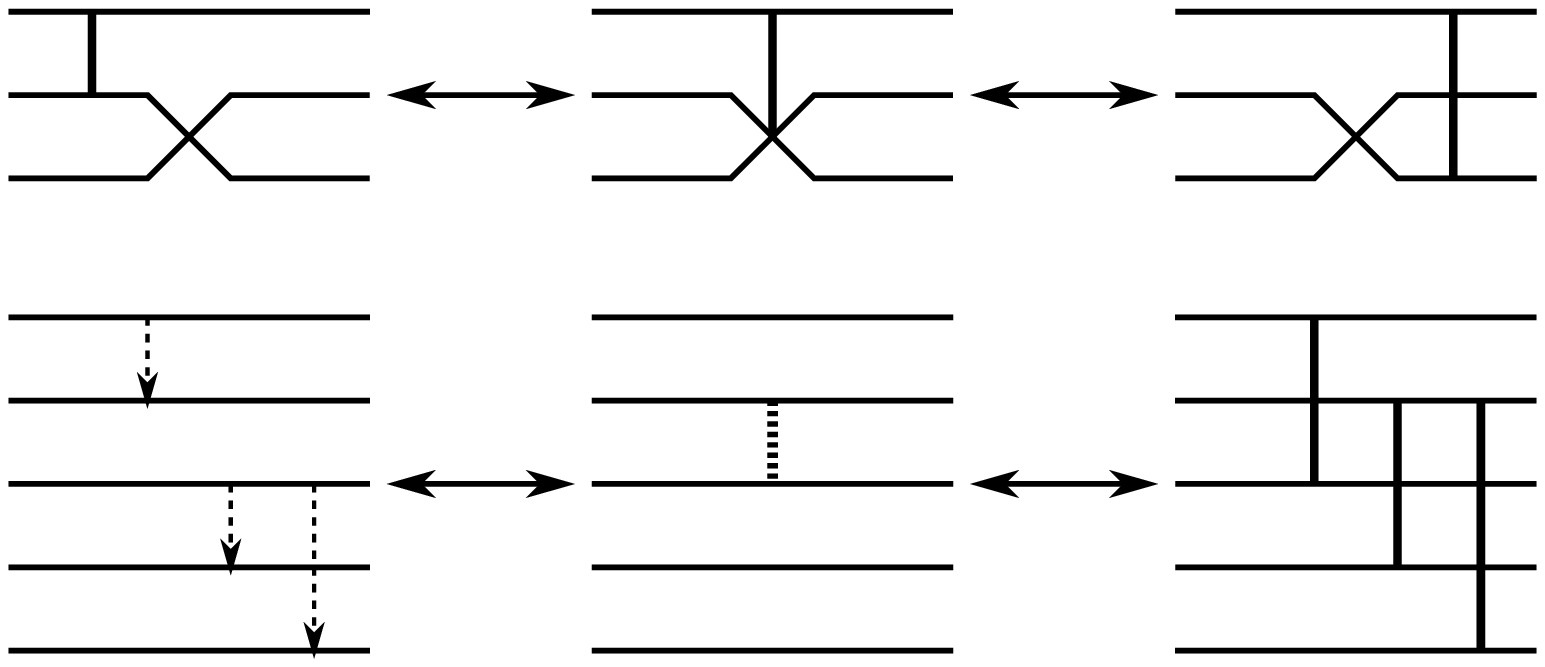}
\caption[]{$\sMCS^{+}$, $\sMCS^{0}$, and $\sMCS^{-}$ (in order) for MCS moves 8 and 15.}
\label{f:move-8-15}
\end{figure}

Label the crossings and rights cusps of $\sfront$ by $q_1, \hdots, q_n$ so that $q_1 < q_2 < \hdots < q_n$ with respect to the ordering from the $x$-axis. Suppose $\sMCS^{+},  \sMCS^{-} \in \sFMCS$ differ by either MCS move 8 or 15 so that, in either case, $\sMCS^{+}$ (resp. $\sMCS^{-}$) is to the left (resp. right) of the double arrow in Figure~\ref{f:MCS-equiv}~or~\ref{f:MCS-explosion}. Suppose $\sMCS^+ = \left( \{(C_m, d_m)\}, \{x_m\}, H \right)$. We may assume without loss of generality that the left edge of $T$ has an associated chain complex $(C_p, d_p; x_p)$ in $\sMCS^+$. In fact, the chain complexes and explicit handleslide marks of $\sMCS^+$ and $\sMCS^-$ agree outside of $T$, so $(C_r, d_r; x_r)$ is a chain complex of both $\sMCS^+$ and $\sMCS^-$ for $1 \leq r \leq p$. We will define an MCS-like object $\sMCS^{0}$ for both MCS moves 8 and 15, which we think of as an intermediary between $\sMCS^{+}$ and $\sMCS^{-}$.

In the case of move 8, suppose $q_k$ is the crossing in $T$ with strands $m$ and $m+1$ crossing at $q_k$. Suppose the handleslide mark $\alpha$ of $\sMCS^{+}$ in $T$ is between strands $l$ and $m$ with $l<m$. We also let $\alpha$ denote the corresponding handleslide in $\sMCS^{-}$; see Figure~\ref{f:move-8-15}. Let $\beta$ denote a handleslide mark in $\sfront$ between the singularity point of $q_k$ and strand $l$; see Figure~\ref{f:move-8-15}. We define $\sMCS^{0}$ to be the set of explicit handleslide marks $(H \cup \{ \beta \}) \setminus \{\alpha\}$ along with the chain complexes $\{(C_m, d_m) | m \neq p+1 \}$. 

In the case of move 15, the handleslide marks of $\sMCS^-$ in $T$ correspond to the gradient flowlines entering $e_l$ and leaving $e_m$ for two generators $e_l, e_m \in C_{p}$ satisfying  $l < m$ and $|e_l|=|e_m| - 1$. Choose $x_p < \bar{x} < x_{p+1}$ so that the tangle $\sfront \cap ([x_p, \bar{x}] \times \rr)$ does not contain handleslide marks, crossings or cusps of $\sMCS^{+}$. In fact, we may assume $x_{p+1}$ is to the right of $T$ and choose $\bar{x}$ so that $\{ \bar{x} \} \times \rr$ is the right edge of $T$. Let $\alpha$ denote a vertical mark between strands $l$ and $m$ in $\sfront \cap ([x_p, \bar{x}] \times \rr)$. We represent $\alpha$ by a dashed line; see Figure~\ref{f:move-8-15}. We define $\sMCS^{0}$ to be the set of vertical marks $H \cup \{ \alpha \}$ along with all of the chain complexes of $\sMCS^{+}$ including an additional chain complex $(C_{\bar{x}}, d_{\bar{x}})$ above $\bar{x}$ defined by $ d_{\bar{x}} = d_p$. 

\begin{figure}[t]
\centering
\includegraphics[scale=.55]{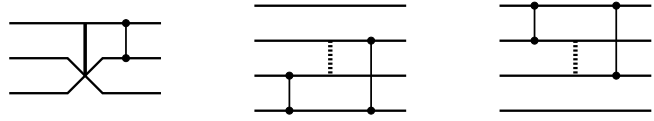}
\caption[]{Conditions $(3_8)$ and $(4_{15})$ in the definitions of $\mathcal{M}^{0}(i ; k)$.}
\label{f:move-8-15-2}
\end{figure}

For $q_i, q_j \in Q(\sfront)$, we will use the shorthand $\mathcal{M}^{+}(i ; j)=\mathcal{M}^{\sMCS^{+}}(q_i, q_j)$ and $\mathcal{M}^{-}(i ; j)=\mathcal{M}^{\sMCS^{-}}(q_i, q_j)$. Fix $q_i \in Q(\sfront)$. In the case of move 8, we define the set $\mathcal{M}^{0}(i ; k)$ by $\Lambda \in \mathcal{M}^{0}(i ; k)$ if and only if $\Lambda = (\lambda_1, \hdots, \lambda_n)$ is a finite sequence of chords originating at $q_i$, satisfying Conditions (1), (2), and (4) of Definition~\ref{def:ChordPath} for the MCS $\sMCS^{+}$, and:

\begin{condition}[$3_8$]
\label{c:3-8}
The chord $\lambda_n$ is just to the right of $q_k$ between strands $l$ and $m$; see Figure~\ref{f:move-8-15-2}.
\end{condition}

Now in the case of move 15, we define the set $\mathcal{M}^{0}(i ; k)$ by $\Lambda \in \mathcal{M}^{0}(i ; k)$ if and only if $\Lambda = (\lambda_1, \hdots, \lambda_n)$ is a finite sequence of chords originating at $q_i$, terminating at $q_j$, and satisfying Conditions (1)-(4) of Definition~\ref{def:ChordPath} for $\sMCS^{0}$ except in one place where $\Lambda$ satisfies: 
\begin{condition}[$4_{15}$]
If $1 \leq r < n$ and the chord $\lambda_r$ occurs at $\bar{x}$, then either:
\begin{enumerate}
	\item $\lambda_r = (\bar{x}, [l,j])$ with $j>m$ and $\lambda_{r+1} = (x_{p}, [m,j])$, or
	\item $\lambda_r = (\bar{x}, [j,m])$ with $j<l$ and $\lambda_{r+1} = (x_{p}, [j,l])$.
\end{enumerate}
In short, $\Lambda$ jumps along $\alpha$ in $T$; see Figure~\ref{f:move-8-15-2}.
\end{condition}

\begin{proof}[Proof of Theorem~\ref{thm:linear-iso} for MCS Move 8]

We define $\Phi : \aac_1(\sMCS^+) \to \aac_1(\sMCS^-)$ to be the linear extension of the map on $Q(\sfront)$ given by $\Phi(q_i) = q_i + \# \mathcal{M}^{0}(q_i; q_k) q_k$. The set $\mathcal{M}^{0}(i; k) \neq \emptyset$ only if $|q_i| = |q_k|$, and $\Phi^{-1} = \Phi$, hence, $\Phi$ is a graded isomorphism. It remains to show $\Phi$ is a chain map. For $i \leq k$, $\Phi(q_i) = q_i$ and $d^{+}_1 q_i = d^{-}_1 q_i$, so $\Phi$ is a chain map. Consider $i > k$. Then:

\begin{align*}
\Phi \circ d^{+}_1 (q_i) &= \Phi \left ( \sum_{j<i} \# \mathcal{M}^{+}(i; j) q_j \right ) \nonumber \\
&= \sum_{k < j < i} \# \mathcal{M}^{+}(i; j) q_j + \sum_{j < k < i} \# \mathcal{M}^{+}(i; j) q_j \nonumber \\
&+ \left [\# \mathcal{M}^{+}(i; k) + \sum_{k<j<i} \# \mathcal{M}^{+}(i; j) \# \mathcal{M}^{0}(j ; k) \right ] q_k \nonumber
\end{align*}

and

\begin{align}
d^{-}_1 \circ \Phi (q_i) &= d^{-}_1 (q_i + \mathcal{M}^{0}(i; k) q_k) \nonumber \\
&= \# \mathcal{M}^{-}(i; k) q_k + \sum_{k < j < i} \# \mathcal{M}^{-}(i; j) q_j \nonumber \\
&+ \sum_{j < k < i} \left [ \# \mathcal{M}^{-}(i; j) + \# \mathcal{M}^{-}(k; j) \# \mathcal{M}^{0}(i ; k) \right ] q_j. \nonumber 
\end{align}

Thus, $\Phi \circ d^{+}_1 (q_i) = d^{-}_1 \circ \Phi (q_i)$ holds if:

\begin{enumerate}
	\item[A.] For $k < j < i$, $\# \mathcal{M}^{+}(i; j) = \# \mathcal{M}^{-}(i; j)$;
	\item[B.] For $j < k < i$, $\# \mathcal{M}^{+}(i; j) = \# \mathcal{M}^{-}(i; j) + \# \mathcal{M}^{-}(k; j) \# \mathcal{M}^{0}(i ; k)$; and 
	\item[C.] For $j=k$, 
	
\begin{equation}
\label{eq:move-8-3}	
	\# \mathcal{M}^{-}(i; k) = \# \mathcal{M}^{+}(i; k) + \sum_{k<j<i} \# \mathcal{M}^{+}(i; j) \# \mathcal{M}^{0}(j ; k).
\end{equation}

\end{enumerate}

For $k < j < i$, $\mathcal{M}^{+}(i; j) = \mathcal{M}^{-}(i; j)$ as sets of chord paths since the handeslide marks of $\sMCS^{+}$ and $\sMCS^{-}$ are identical between $q_i$ and $q_j$. Thus, A. follows. For $j < k < i$, there is injection of $\mathcal{M}^{-}(i; j)$ into $\mathcal{M}^{+}(i; j)$ whereby the image of $\Lambda \in \mathcal{M}^{-}(i; j)$ jumps along the same handleslide marks as $\Lambda$, including possibly $\alpha$. The set $\mathcal{M}^{+}(i; j) \setminus \mathcal{M}^{-}(i; j)$ consists of chord paths that jump along $\alpha$ so that the chord after the jump is between strands $m$ and $m+1$. This set is in bijection with $\mathcal{M}^{-}(k; j) \times \mathcal{M}^{0}(i ; k)$ by the map described in Figure~\ref{f:M8-bijection}. Thus, B. follows.

\begin{figure}[t]
\labellist
\small\hair 2pt
\pinlabel {$\beta$} [br] at 193 34
\pinlabel {$\alpha$} [tl] at 354 30
\pinlabel {$q_k$} [tl] at 60 10
\pinlabel {$q_k$} [tl] at 190 10
\pinlabel {$q_k$} [tl] at 390 10
\endlabellist
\centering
\includegraphics[scale=.55]{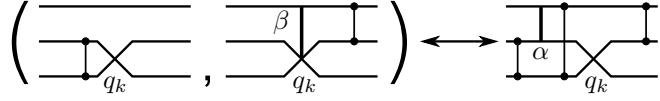}
\caption[]{The correspondence between $(\Lambda, \Omega) \in \mathcal{M}^{-}(k; j) \times \mathcal{M}^{0}(i ; k)$ and a chord path in $\mathcal{M}^{+}(i; j) \setminus \mathcal{M}^{-}(i; j)$.}
\label{f:M8-bijection}
\end{figure}

Suppose $j=k$. Most chord paths in $\mmc^+(i;k)$ correspond in a canonical way to a chord path in $\mmc^-(i;k)$, and vice versa. We define $\widetilde{\mathcal{M}}(i ; k)$ to be the subset of $ \mmc^+(i;k) \cup \mmc^{-}(i;k)$ consisting of $\Lambda = (\lambda_1, \hdots \lambda_n)$ satisfying: 

\begin{enumerate}
	\item $\Lambda \in \mmc^+(i;k)$, $\lambda_n = (x_{p+1}, [i,m]), i<l,$ and $\langle d^+_{p} e_i, e_l \rangle =1$; see Figure~\ref{f:M8-exceptions}~(a);
	\item $\Lambda \in \mmc^-(i;k)$, $\lambda_n = (x_{p+1}, [i,m]), i<l,$ and $\langle d^+_{p} e_i, e_m \rangle =0$; see Figure~\ref{f:M8-exceptions}~(b); or 
	\item $\Lambda \in \mmc^-(i;k)$, $\lambda_n = (x_{p+1}, [m+1,j])$, and $\Lambda$ jumps along $\alpha$; see Figure~\ref{f:M8-exceptions}~(c).
\end{enumerate}
The set $\widetilde{\mathcal{M}}(i ; k)$ consists of chord paths in $\mmc^+(i;k)$ and $\mmc^-(i;k)$ that do not correspond in a canonical way to another chord path after the MCS move. Hence, $\# \widetilde{\mathcal{M}}(i ; k) = \# \mmc^-(i;k) - \# \mmc^+(i;k)$ and we can rewrite equation~(\ref{eq:move-8-3}) as 

\begin{equation}
\label{eq:iso-8}
\# \widetilde{\mathcal{M}}(i ; k) = \sum_{k<j<i} \# \mathcal{M}^{+}(i; j) \# \mathcal{M}^{0}(j ; k).
\end{equation}

\begin{figure}[t]
\labellist
\small\hair 2pt
\pinlabel {(a)} [tl] at 61 16
\pinlabel {(b)} [tl] at 225 16
\pinlabel {(c)} [tl] at 400 16
\pinlabel {$i$} [br] at 11 86
\pinlabel {$l$} [br] at 11 61
\pinlabel {$m$} [br] at 11 36
\pinlabel {$m+1$} [br] at 11 11
\pinlabel {$j$} [br] at 355 11
\endlabellist
\centering
\includegraphics[scale=.55]{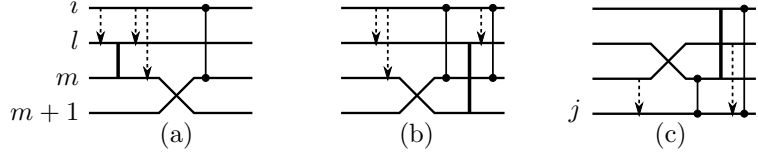}
\caption[]{The exceptional chord paths of $\mmc^+(i;k)$ and $\mmc^{-}(i;k)$ in $\widetilde{\mathcal{M}}(i ; k)$.}
\label{f:M8-exceptions}
\end{figure}

We will verify equation~(\ref{eq:iso-8}) by producing a graph $G=(V,E)$ as in the claim below. The resulting graph is a compact topological $1$-dimensional manifold with boundary points corresponding to $$\widetilde{\mathcal{M}}(i ; k) \cup \left [\bigcup_{k<j<i} \mathcal{M}^{-}(i; j) \times \mathcal{M}^{0}(j; k) \right ].$$ The boundary points are naturally paired in such a manifold, hence, the construction of $G$ verifies equation~(\ref{eq:iso-8}) and finishes the lemma. 

The build-up to the next Claim has been lengthy, so we take a moment to recall the chord paths corresponding to vertices in $G$. The sets $ \mmc^{\pm}(*;*)$ consist of chords paths in the MCSs $\sMCS^{\pm}$. The set $\widetilde{\mathcal{M}}(i ; k)$ is a certain subset of $ \mmc^+(i;k) \cup \mmc^{-}(i;k)$. The set $\mathcal{M}^{0}(*;*)$ consists of chord paths in the intermediary $\sMCS^0$ between MCSs $\sMCS^+$ and $\sMCS^-$; see Figure~\ref{f:move-8-15} and the discussion preceding Condition~$(3_8)$. The set $\widehat{\mathcal{M}}(i ; k)$ of interior vertices is defined after the claim. 

\begin{claim}
There exists a graph $G = (V, E)$ with vertex set $$V = \widehat{\mathcal{M}}(i ; k) \cup \widetilde{\mathcal{M}}(i ; k) \cup \left [ \bigcup_{k<j<i} \mathcal{M}^{-}(i; j) \times \mathcal{M}^{0}(j; k) \right ]$$ satisfying:
\begin{enumerate}
	\item[A.] Every element of $\widehat{\mathcal{M}}(i ; k)$ has valence 2; and
	\item[B.] Every element of $V \setminus \widehat{\mathcal{M}}(i ; k)$ has valence 1.
\end{enumerate}
\end{claim}

We define the set of interior vertices $\widehat{\mathcal{M}}(i ; k)$ by $\Lambda \in \widehat{\mathcal{M}}(i ; k)$ if and only if $\Lambda = (\lambda_1, \hdots, \lambda_n)$ is a finite sequence of chord paths so that:
\begin{enumerate}
	\item $\Lambda$ satisfies Definition~\ref{def:ChordPath} (1), (2), and (4) for the MCS $\sMCS^{+}$ except that for a single value of $1 \leq r < n$, $\lambda_r$ and $\lambda_{r+1}$ violate Definition~\ref{def:ChordPath} (4) and instead $\Lambda$ has an exceptional step between $\lambda_r$ and $\lambda_{r+1}$ as defined in Definition~\ref{defn:exceptional-step};
	\item $\Lambda$ originates at $q_i$; and
	\item $\Lambda$ satisfies Condition ($3_8$).
\end{enumerate}

In short, $\Lambda \in \widehat{\mathcal{M}}(i ; j)$ is a chord path in $\sMCS^{+}$ from $q_i$ that jumps along a single fiber-wise gradient flowline and satisfies Condition ($3_8$).

\begin{figure}[t]
\centering
\includegraphics[scale=.55]{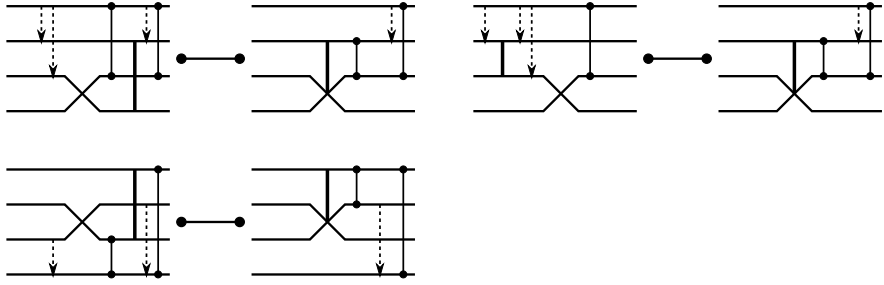}
\caption[]{The edges between vertices in $\widetilde{\mathcal{M}}(i ; k)$ and vertices in $\widehat{\mathcal{M}}(i ; k)$.}
\label{f:M8-edge-B}
\end{figure}

Edges at valence 1 vertices are defined as follows. Suppose $\Lambda = (\lambda_1, \hdots, \lambda_n) \in \widetilde{\mathcal{M}}(i ; k)$. We associate to each of the three possible chord paths in $\widetilde{\mathcal{M}}(i ; k)$ a chord path in $\widehat{\mathcal{M}}(i ; k)$ as described in Figure~\ref{f:M8-edge-B} and introduce an edge between their vertices. Note that the vertices of $\widetilde{\mathcal{M}}(i ; k)$ corresponding to (a) and (b) in Figure~\ref{f:M8-exceptions} are connected by an edge to the same chord path in $\widehat{\mathcal{M}}(i ; k)$. Hence, these two boundary vertices are paired by the graph. Given $$\Lambda \times \Omega = (\lambda_1, \ldots, \lambda_{n_1}) \times (\omega_1, \ldots, \omega_{n_2}) \in  \mathcal{M}^{-}(i; j) \times \mathcal{M}^{0}(j; k),$$ we introduce an edge whose other end corresponds to the chord path in $\widehat{\mathcal{M}}(i ; k)$ obtained from gluing $\Lambda$ and $\Omega$ as $(\lambda_1, \ldots, \lambda_{n_1}, \tau, \omega_1, \ldots, \omega_{n_2})$. The additional chord $\tau$ is obtained from pushing $\lambda_{n_1}$ through the crossing $q_j$, with an exceptional step occurring between $\tau$ and $\omega_1$; see Figure~\ref{f:M8-edge-A}. 

\begin{figure}[t]
\labellist
\small\hair 2pt
\pinlabel {$q_j$} [tl] at 60 19
\pinlabel {$q_j$} [tl] at 189 19
\pinlabel {$q_j$} [tl] at 380 19
\endlabellist
\centering
\includegraphics[scale=.55]{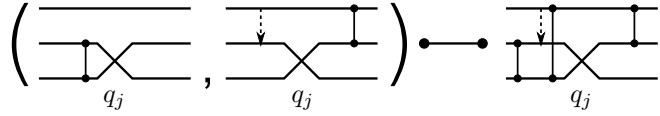}
\caption[]{The edge between a vertex in $\mathcal{M}^{-}(i; j) \times \mathcal{M}^{0}(j; k)$ and a vertex in $\widehat{\mathcal{M}}(i ; k)$.}
\label{f:M8-edge-A}
\end{figure}

The internal edges of $E$ are determined by the way an exceptional step can be moved around in $\sMCS^+$ between $q_i$ and $q_k$. For vertices in $\widehat{\mathcal{M}}(i ; k)$, the exceptional step is assigned both a ``left type'' and ``right type'' depending on the tangle appearing to the left or, in the latter case, the right of the exceptional step and the way in which the chord sequence passes through this tangle. The assignment of edges to the vertices in $\widehat{\mathcal{M}}(i ; k)$ is based on the type of the exceptional step. In fact, the construction of these edges and the proof that the vertices of $\widehat{\mathcal{M}}(i ; j)$ are valence 2 is essentially identical to the proof of Lemma~\ref{lem:MSL}. We leave the details to the reader.
\end{proof}

\begin{proof}[Proof of Theorem~\ref{thm:linear-iso} for MCS Move 15]

We define $\Phi : \aac_1(\sMCS^+) \to \aac_1(\sMCS^-)$ to be the linear extension of the map on $Q(\sfront)$ given by $\Phi(q_i) = q_i + \sum_{j<i} \# \mathcal{M}^{0}(i,j) q_j$. The set $\mathcal{M}^{0}(i; j) \neq \emptyset$ only if $|q_i| = |q_j|$, and $\Phi^{-1} = \Phi$, hence, $\Phi$ is a graded isomorphism. It remains to show $\Phi$ is a chain map. Note, 

\begin{align*}
\Phi \circ d^{+}_1 (q_i) &= \Phi \left ( \sum_{j<i} \# \mathcal{M}^{+}(i; j) q_j \right ) \nonumber \\
&= \sum_{j < i} \left [ \# \mathcal{M}^{+}(i; j) q_j + \sum_{l < j} \# \mathcal{M}^{+}(i; j) \mmc^{0}(j, l) q_l \right ] \nonumber 
\end{align*}

and

\begin{align*}
d^{-}_1 \circ \Phi (q_i) &= d^{-}_1 \left (q_i + \sum_{j<i} \mathcal{M}^{0}(i; j) q_j \right ) \nonumber \\
&= \sum_{j < i} \left [ \# \mathcal{M}^{-}(i; j) q_j + \sum_{l < j} \# \mathcal{M}^{0}(i; j) \mmc^{-}(j, l) q_l \right ] \nonumber 
\end{align*}

Thus, $\Phi \circ d^{+}_1 (q_i) = d^{-}_1 \circ \Phi (q_i)$ if:

\begin{enumerate}
	\item[A.] For $T < j < i$ or $j < i < T$, $\# \mathcal{M}^{+}(i; j) = \# \mathcal{M}^{-}(i; j)$; and
	\item[B.] For $j < T < i$, $$\# \mmc^{+}(i,j) + \sum_{j<l<T} \# \mmc^{+}(i,l) \# \mmc^{0}(l,j) = \# \mmc^{-}(i,j) + \sum_{j<T<l} \# \mmc^{0}(i,l) \# \mmc^{-}(l,j).$$
\end{enumerate}

For $T < j < i$ or $j < i < T$, $\mathcal{M}^{+}(i; j) = \mathcal{M}^{-}(i; j)$ as sets of chord paths since the handeslide marks of $\sMCS^{+}$ and $\sMCS^{-}$ are identical between $q_i$ and $q_j$, so A. follows. For $j < T < i$, the MCSs $\sMCS^{+}$ and $\sMCS^{-}$ differ only in the handleslide marks of $\sMCS^{-}$ in $T$, hence, $\mmc^{+}(i, j)$ naturally injects into $\mmc^{-}(i,j)$. Thus, $\widetilde{\mathcal{M}}(i ; j) = \mmc^{-}(i,j) \setminus \mmc^{+}(i,j)$ denotes the chord paths in $\sMCS^{-}$ that jump along handleslide marks in $T$ and $\# \widetilde{\mathcal{M}}(i ; j) = \# \mmc^{-}(i,j) - \# \mmc^{+}(i,j)$. In fact, it is easy to see that $\Lambda \in \widetilde{\mathcal{M}}(i ; j)$ jumps along exactly one handleslide in $T$. In the case of $j < T < i$, we can rewrite B. above as:

\begin{equation}
\label{eq:iso-15}
\# \widetilde{\mathcal{M}}(i ; j) = \sum_{j<l<T} \# \mmc^{+}(i,l) \# \mmc^{0}(l,j) + \sum_{j<T<l} \# \mmc^{0}(i,l) \# \mmc^{-}(l,j).
\end{equation}

We will verify equation~(\ref{eq:iso-15}) by producing a graph $G=(V,E)$ as in the claim below. The resulting graph is a compact topological $1$-dimensional manifold with boundary points corresponding to $$\widetilde{\mathcal{M}}(i ; j) \cup \left [ \bigcup_{j<T<l<i} \mathcal{M}^{+}(i; l) \times \mathcal{M}^{0}(l; j) \right ] \cup \left [ \bigcup_{j<l<T<i} \mathcal{M}^{0}(i; l) \times \mathcal{M}^{-}(l; j) \right ].$$ The boundary points are naturally paired in such a manifold, hence, the construction of $G$ verifies equation~(\ref{eq:iso-15}) and finishes the lemma. The set $\widehat{\mathcal{M}}(i ; j)$ of interior vertices is defined after the claim. 

\begin{claim}
For $T < j < i$ fixed, there exists a graph $G = (V, E)$ with vertex set 
\begin{align*}
V &= \widehat{\mathcal{M}}(i ; j) \cup \widetilde{\mathcal{M}}(i ; j) \nonumber \\
& \cup \left [ \bigcup_{j<T<l<i} \mathcal{M}^{+}(i; l) \times \mathcal{M}^{0}(l; j) \right ] \nonumber \\
& \cup \left [ \bigcup_{j<l<T<i} \mathcal{M}^{0}(i; l) \times \mathcal{M}^{-}(l; j) \right ] \nonumber
\end{align*}
satisfying:
\begin{enumerate}
	\item[A.] Every element of $\widehat{\mathcal{M}}(i ; k)$ has valence 2; and
	\item[B.] Every element of $V \setminus \widehat{\mathcal{M}}(i ; k)$ has valence 1.
\end{enumerate}
\end{claim}

For a fixed $q_i \in Q(\sfront)$, we define the set $\widehat{\mathcal{M}}(i ; j)$ by $\Lambda \in \widehat{\mathcal{M}}(i ; j)$ if and only if $\Lambda = (\lambda_1, \hdots, \lambda_n)$ is a finite sequence of chord paths so that:
\begin{enumerate}
	\item $\Lambda$ satisfies Definition~\ref{def:ChordPath} (1) - (4) for $\sMCS^{0}$ except that for a single value of $1 \leq r < n$, $\lambda_r$ and $\lambda_{r+1}$ violate Definition~\ref{def:ChordPath} (4) and instead $\Lambda$ has an exceptional step between $\lambda_r$ and $\lambda_{r+1}$ as defined in Definition~\ref{defn:exceptional-step};
	\item $\Lambda$ satisfies Condition ($4_{15}$); and
	\item $\Lambda$ originates at $q_i$ and terminates at $q_j$.
\end{enumerate}

In short, $\Lambda \in \widehat{\mathcal{M}}(i ; j)$ is a chord path in $\sMCS^{0}$ from $q_i$ and $q_j$ that jumps along $\alpha$ in $T$ and also jumps along a single fiber-wise gradient flowline, including possibly a gradient flowline in the chain complex $(C_{\bar{x}}, d_{\bar{x}})$.

\begin{figure}[t]
\centering
\includegraphics[scale=.55]{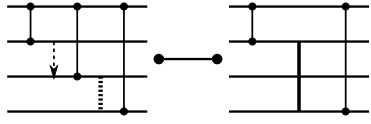}
\caption[]{The edge between a vertex in $\widehat{\mathcal{M}}(i ; j)$ and a vertex in $\widetilde{\mathcal{M}}(i ; j)$.}
\label{f:M15-edge-A}
\end{figure}

We begin by constructing the edges from valence 1 vertices. Suppose $\Lambda = (\lambda_1, \hdots, \lambda_n) \in \widetilde{\mathcal{M}}(i ; j)$. The chord path $\Lambda$ jumps along a handleslide mark $\beta$ of $\sMCS^{-}$ in the tangle $T$. Suppose $\beta$ is the handleslide between strands $a$ and $m$ with $a < l$, then $\langle d e_a, e_l \rangle = 1$ and, in $\sMCS^{0}$, $\beta$ can be uniquely broken into a gradient flowline between $a$ and $l$ and the mark $\alpha$. We introduce an edge whose other end corresponds to the chord path in $\widehat{\mathcal{M}}(i ; j)$ that jumps along the gradient flowline between $a$ and $l$ and also jumps along $\alpha$; see Figure~\ref{f:M15-edge-A}.

\begin{figure}[t]
\labellist
\small\hair 2pt
\pinlabel {$q_l$} [tl] at 60 5
\pinlabel {$q_l$} [tl] at 188 5
\pinlabel {$q_l$} [tl] at 380 5
\endlabellist
\centering
\includegraphics[scale=.55]{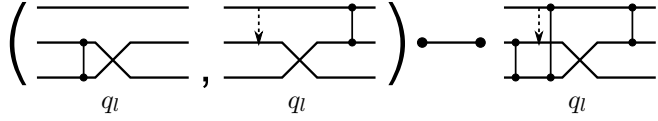}
\caption[]{This figure represents both: (a) the edge between a vertex in $\mathcal{M}^{+}(i; l) \times \mathcal{M}^{0}(l; j)$ and a vertex in $\widehat{\mathcal{M}}(i ; k)$ for $j<T<l<i$ and (b) the edge between a vertex in $\mathcal{M}^{0}(i; l) \times \mathcal{M}^{-}(l; j)$ and a vertex in $\widehat{\mathcal{M}}(i ; k)$ for $j<l<T<i$.}
\label{f:M15-edge-B}
\end{figure}

Suppose $j<T<l<i$ or $j<l<T<i$. Given $$\Lambda \times \Omega = (\lambda_1, \ldots, \lambda_{n_1}) \times (\omega_1, \ldots, \omega_{n_2}) \in  \mathcal{M}^{+}(i; l) \times \mathcal{M}^{0}(l; j),$$ we introduce an edge whose other end corresponds to the chord path in $\widehat{\mathcal{M}}(i ; j)$ obtained from gluing $\Lambda$ and $\Omega$ as $(\lambda_1, \ldots, \lambda_{n_1}, \tau, \omega_1, \ldots, \omega_{n_2})$. The additional chord $\tau$ is obtained from pushing $\lambda_{n_1}$ through the crossing $q_l$, with an exceptional step occurring between $\tau$ and $\omega_1$; see Figure~\ref{f:M15-edge-B}. 

The internal edges of $E$ between elements of $\widehat{\mathcal{M}}(i ; j)$ are determined by the way an exceptional step can be moved around in $\sMCS^{0}$. The edge assignments follow the same ``left type/right type'' argument used in Lemma~\ref{lem:MSL} and in the proof of MCS move 8. We leave the details to the reader.
\end{proof}

Having considered each MCS move depicted in Figures~\ref{f:MCS-equiv}~and~\ref{f:MCS-explosion}, the proof of Theorem~\ref{thm:linear-iso} is now complete.  
\end{proof}

\section{Comparing chord paths with difference flow trees}
\label{ch:grad-staircase}

In this section we return to the case where our MCS arises from a generating family. We give an alternate description of chord paths as certain piece-wise gradient trajectories along the fiber product of $L$ with itself. The discontinuities correspond to jumps along fiberwise gradient trajectories. We give a precise conjectural interpretation for these broken gradient trajectories as limits of the difference flow trees of Section~\ref{ch:Gen-fams} when the metric is rescaled along the fibers. In this picture a convex corner in a chord path corresponds to an internal vertex in a difference flow tree.

\subsection{Gradient staircases}

Let $H: \R\times \R^K \rightarrow \R$, $h_x = H(x, \cdot)$ denote a family of functions. Our primary interest is the case when $H$ is the difference function of a given generating family $F$. Choose a family of metrics $g_F = (g_F)_x$ on the fibers $\R^K$, and let $g_\R$ denote a metric on $\R$. The sum $g := g_\R + g_F$ provides a metric on $\R\times\R^K$ so that the components of the Cartesian product are orthogonal.

We consider the behavior of gradient trajectories between fixed critical points $p,q \in \mathit{Crit}(H)$ as the metric is rescaled along the fibers towards a degenerate limit. Specifically, put 
\begin{equation} \label{eq:Deform}
g_s := g_\R + s(g_F), \quad 1 \geq s >0,
\end{equation}
 and let $ \nabla_s H$ denote the gradient of $H$ with respect to $g_s$. If $\nabla_1H (x, e) = (\alpha, \beta) \in T_{(x,e)}(\R \times \R^K) \cong T_x\R \times T_e\R^K$ then $\displaystyle \nabla_s H (x,e) = \left ( \alpha,\frac{1}{s} \beta \right )$, so the deformation of the metric results in an exaggeration of the fiber component of $\nabla_1H$ as $s$ approaches $0$. 

The effect on gradient trajectories is illustrated by the following example. Let $H : \rr \times \rr \to \rr$ by 
\[ \displaystyle
H(x,e) = \frac{-x^3}{3} + \frac{x^2}{2} + \frac{-e^3}{3} + \frac{e^2}{2}.
\]
We take our initial metric to be Euclidean so that 
\[
-\nabla_s H = \left(x^2 -x, \frac{1}{s} (e^2 -e) \right). 
\]
 The fiber critical set $S_H$ consists of the horizontal lines $e=0$ and $e=1$. For each value of $s$ there is a unique flow line $\gamma_s : \R \rightarrow \R\times \R$ of $-\nabla_s H $ with $\gamma_s(0) = {\bf v}_0 = \left (\frac{1}{2}, \frac{1}{2} \right )$. Each $\gamma_s$ is contained within the square $[0,1]\times[0,1]$ and limits to $(1,1)$ and $(0,0)$ as $t \rightarrow \pm \infty$. As $s \rightarrow 0$ the portion of $\gamma_s$ near ${\bf v}_0$ becomes more and more vertical and travels past ${\bf v}_0$ with ever increasing speed. On intervals $(-\infty, -\epsilon],$ (resp. $[\epsilon, +\infty) \,$), $\epsilon >0$, $\gamma_s$ uniformly approaches the trajectories along the portion of $S_H$ belonging to $e=1$ (resp. $e=0$); see Figure~\ref{fig:Flow-Line}. The ``limit'' as $s\rightarrow 0$ of the $\gamma_s$ may therefore be viewed as a piecewise gradient trajectory within the fiber critical set $S_H$ with a jump occurring at $t=0$ along the fiber gradient trajectory of $H (\frac{1}{2}, \cdot ): \rr \to \rr$ from $(\frac{1}{2}, 1)$ to $(\frac{1}{2}, 0)$. 

\begin{figure}
\centerline{\includegraphics[scale=.35]{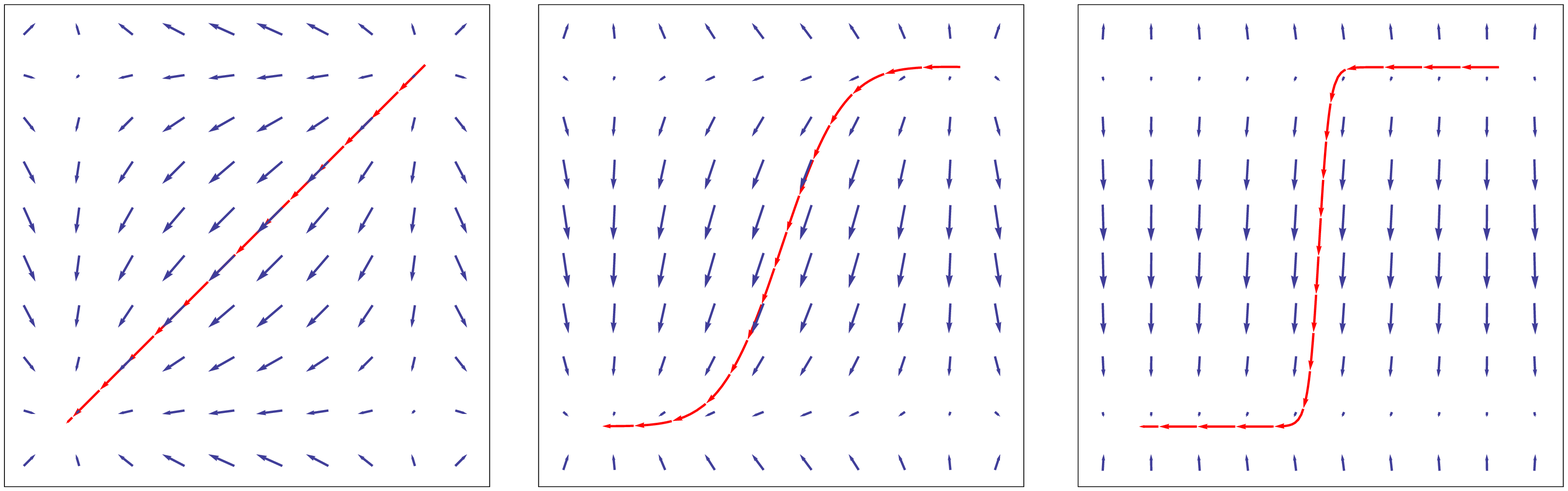}}
\caption{Behavior of the vector fields $-\nabla_s H$ and the trajectories $\gamma_s$ as $s$ approaches $0$.}
\label{fig:Flow-Line}
\end{figure} 
 
With the notion of a {\it gradient staircase} (Definition~\ref{def:GradStair} below) we aim to encode this type of limiting behavior. First, we require a
preparatory lemma regarding the gradient of the restriction $H|_{S_H}$ as $s\rightarrow 0$. Let $j: S_H \hookrightarrow \R \times \R^K$ denote the inclusion so that $j^*g_s$ denotes the restriction of $g_s$ to $S_H$.

\begin{lemma} At each point $(x, e) \in S_H$, $\lim_{s\rightarrow 0} \nabla_{j^*g_s} (H|_{S_H})$ exists.
\end{lemma}

\begin{proof} Let $(u, v)$ be a basis vector for  $T_{(x,e)}S_H \subset T_x\R \times T_e\R^K$. In general, the gradient of a function restricted to a submanifold is the orthogonal projection of the gradient. Along $S_H$, $ \nabla_1 H (x, e)$ has the form $ (\alpha, 0)$, so we can compute
\[\displaystyle
\lim_{s\rightarrow 0} \nabla_{j^*g_s} (H|_{S_H})= \lim_{s\rightarrow 0} \frac{g_s((\alpha, 0), (u,v))}{g_s((u,v),(u,v))} (u,v) =
\] 
\[\displaystyle
\lim_{s\rightarrow 0} \frac{g_\R(\alpha, u)}{g_\R(u,u) +s g_F(v,v)} (u,v).
\] 
Now if $u=0$ the limit is clearly $0$ (this happens at the degenerate critical points of the $h_x$). Otherwise, the limit is
\[
\frac{g_\R(\alpha, u)}{g_\R(u,u)} (u,v).
\]
\end{proof}

We will denote this limiting vector field defined along $S_H$ as $\nabla_0 H$. 

\begin{definition} \label{def:GradStair}
For $p, q \in \mathit{Crit}(H)$ a {\it gradient staircase} from $p$ to $q$ is a triple
\[
( \vec{\rho}, \vec{\varphi}, {\bf r}) = 
\left( (\rho_0, \rho_1, \ldots, \rho_M), (\varphi_1, \ldots, \varphi_M), (r_1, \ldots, r_M) \right)
\]
where 
\begin{itemize}
\item For each $1 \leq m \leq M$, $r_m \in \rr$ and $r_m \leq r_{m+1}$. We make the convention that $r_0= -\infty$ and $r_{M+1} = +\infty$.
\item The $\rho_m$ are continuous maps $\rho_m : [r_m, r_{m+1}] \rightarrow S_H \subset \R\times \R^K$ and satisfy $\dot{\rho}_m = -\nabla_0H\circ \rho_m$ on $(r_m, r_{m+1})$. (This interval may be empty.)
\item For $1 \leq m \leq M$ there exist $x_m\in \R$ such that $\rho_{m-1}(r_m) = (x_m, e_m)$, $\rho_{m}(r_m) = (x_m, f_m)$, and $\varphi_m \in \mathcal{M}(h_x, g_x; e_m, f_m)$.
\end{itemize}
\end{definition}

A gradient staircase may be easily visualized on the front projection for the Legendrian $L_H = i_H(S_H)$. The critical points $p$ and $q$ appear as local maximums and minimums of the strands of $L_H$. Beginning at $p$ we travel downward remaining on the front diagram except for isolated jumps between strands along the fiberwise gradient trajectories $\varphi_m$. These jumps are the ``steps'' of the gradient staircase; see Figure~\ref{fig:GradStair}.

\begin{figure}
\labellist
\small\hair 2pt
\pinlabel {$p$} [tl] at 150 156
\pinlabel {$q$} [tl] at 141 28
\endlabellist
\centerline{\includegraphics[scale=1]{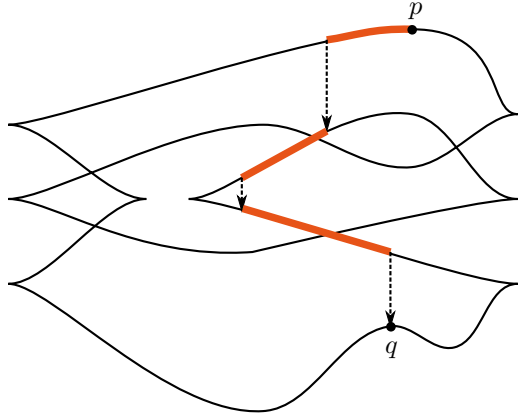}}
\caption{A gradient staircase connecting critical points $p$ and $q$ of $H$.}
\label{fig:GradStair}
\end{figure}

Notice that $\R$ acts on the set of gradient staircases from $p$ to $q$ via translating the $r_m$ and reparametrizing the $\rho_m$ in an appropriate manner. We let $\mathcal{M}(H, g_0; p, q)$ denote the orbit set of this $\R$ action.

\begin{conjecture} \label{conj:GS} For generic choices of $g_\R$ and $g_F$ and fixed $p, q \in \mathit{Crit}(H)$ with $\mathit{ind}(p) -1 = \mathit{ind}(q)$ there exists $\delta >0$ such that for $\delta \geq s >0$ there is a bijection between $\mathcal{M}(H, g_s; p, q)$ and  $\mathcal{M}(H, g_0; p, q)$.
\end{conjecture}

\subsection{Chord paths as gradient staircases}

We now specialize to the following situation. Let $L$ be a Legendrian knot whose front diagram is nearly plat and in Ng form as in Section~\ref{ch:Background}. Suppose that $F:\R\times\R^N \rightarrow \R$ is a linear at infinity generating family for $L$ and a family of metrics $g_x$ is chosen as in Proposition \ref{prop:GF2MCS} of Section 4.2, and denote the corresponding MCS as $\mathcal{C} = \mathcal{C}(F, g)$. As usual, $w: \R\times \R^N\times \R^N \rightarrow \R$ denotes the difference function of $F$, and due to the nature of the Ng form the critical points of $w$ correspond to crossings and right cusps of $L$; see Figure~\ref{f:Ng-resolution}. For considering gradient staircases of $w$ we will use the Euclidean metric on $\R$ and the metric $(g_F)_x = g_x \times g_x$ on the fibers $\R^N\times \R^N$ as in equation~(\ref{eq:Deform}).

\begin{proposition} \label{prop:GSCorr} With the above assumptions, for any generators $a$ and $b$ with $|a| = |b|+1$, there is a mod $2$ correspondence between the gradient staircases, $\mathcal{M}(w, g_0; a, b)$, and the set of chord paths $\mathcal{M}^{\mathcal{C}}(a; b)$. That is,
\[
\# \mathcal{M}(w, g_0; a, b) = \# \mathcal{M}^{\mathcal{C}}(a; b) \mod 2.
\]

\end{proposition}

\begin{proof}

The fiber critical set of the difference function is the fiber product of that of the generating family $F$ with itself,
\[
S_w = S_F * S_F = \{ (x, e_1, e_2) \,|\, (x, e_i) \in S_F, \, i = 1, 2 \} .
\]
Therefore, there is a one-to-one correspondence between points of $S_w$ satisfying $w > \delta$ and chords on the front diagram of $L_F$ with length greater than $\delta$; the value of $w$ corresponds to the length of a chord. Since $S_w$ is $1$-dimensional, the flow of $-\nabla_0w$ simply pushes chords in the direction where $w$ decreases. Due to the nature of the Ng form, this means that, during the $\rho_m$ portions of a gradient staircase, chords will proceed to the left. This is because the only chords whose lengths decrease when moved to the right are forced to enter the region where $w \leq \delta$. In fact, the only such chords connect two crossing strands in the interval to the left of the crossing between the change in slope (which corresponds to the crossing on the Lagrangian projection) and the actual location of the crossing on the front projection. All of these observations can be seen in the second row of images in Figure~\ref{f:Ng-resolution}.

A step, $\varphi_m$, in a gradient staircase corresponds to an element of a moduli space of fiber gradient trajectories, $\mathcal{M}\left(w_x, g_x\times g_x; (e_1, e_2), (f_1, f_2)\right)$, where $(x, e_1, e_2), (x, f_1, f_2) \in S_w$. Now, $w_x(e_1, e_2) = f_x(e_1) - f_x(e_2)$, so, since we use the product metric $g_x\times g_x$, such a step $\varphi_m$ amounts to a pair $(\alpha, \beta)$ with $\alpha$ (resp. $\beta$) a trajectory of $-\nabla f_x$ (resp. $+\nabla f_x$) from $e_1$ to $f_1$ (resp. from $e_2$ to $f_2$). This appears in the front diagram picture as the upper end point of a chord jumping downward and/or the lower end point of a chord jumping upward. When $|a| = |b|+1$ one of $\alpha$ or $\beta$ must be constant, and the non-constant trajectory must connect fiber critical points whose Morse indices either agree or differ by $1$. The mod $2$ count of such trajectories is specified by the presence of handleslide marks and the complexes $(C_m, d_m)$ appearing in the MCS $\mathcal{C}$ which are precisely the Morse complexes of the $f_x$.   

Let us now turn to the portion of a gradient staircase near $b$. When restricted to $S_w$ all of the critical points of $w$ with positive critical value are local maxima; this again follows from the resolution construction. Therefore, the only way for a gradient staircase to terminate at $b$ is for the final step, $\phi_M$, to end directly at $b$ after which $\rho_M$ remains there as a constant trajectory. This explains item (3) from Definition~\ref{def:ChordPath}: the terms $\langle d_p e_i, e_k \rangle$ and $\langle d_p e_{k+1}, e_j \rangle$ which appear there correspond mod $2$ to the possibilities for the final step, $\phi_M$. A picture of the end of a chord path which more accurately reflects the resolution construction is given in Figure~\ref{fig:EndOfGradStair}.

\begin{figure}
\centerline{\includegraphics[scale=.7]{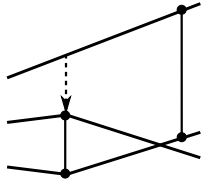}}
\caption{A gradient staircase for $w$ near the terminal critical point $b$.}
\label{fig:EndOfGradStair}
\end{figure}

Next, recall that in the case of a front projection in Ng form the degree of a crossing or cusp is the difference of the Maslov potential of the upper strand and the lower strand. Equivalently, the degree is the difference of the Morse index of the upper and lower strands. For a chord $\lambda$, let $\mu(\lambda)$ denote this difference. Notice $\mu(\lambda)$ can only change at the steps $\varphi_m$ where it cannot increase. At the final step, $\varphi_M$, $\mu(\lambda)$ decreases by $1$. As $|a| = |b| +1$, it follows that $\mu(\lambda)$ must remain unchanged during the rest of the steps. Therefore, all steps $\varphi_m$ with $m < M$ must consist of a single end of a chord jumping along a fiber gradient trajectory between two critical points of the same index, i.e. along a handleslide. 

We have seen that the steps of a gradient staircase $( \vec{\rho}, \vec{\varphi}, {\bf r}) \in \mathcal{M}(w, g_0; a, b)$ only occur at $b$ and handleslides, and that the $\rho_m$ simply push chords to the left.  An element of $\mathcal{M}(w, g_0; a, b)$ is therefore determined by recording the chords that occur at those $x$-values, $x_r$, from the MCS, $\mathcal{C}$, which lie between $a$ and $b$ and remembering which fiber gradient trajectory constitutes the final step, $\varphi_M$. Definition~\ref{def:ChordPath} is designed to include precisely the sequences of chords that can arise in this manner, and item (3) remembers the number of choices for $\varphi_M$ mod $2$. The result follows.

\end{proof}

\subsection{Convex corners and flow trees}

One can easily generalize Conjecture~\ref{conj:GS} to include the case of difference flow trees with more than $1$ output. In this case, we expect a correspondence for small enough $s$ between difference flow trees in $\mathcal{M}^{F, g_s}(a; b_1, \ldots, b_n)$ and trees with edges parameterizing portions of gradient staircases for appropriate difference functions. In this subsection, we will refer to the latter type of tree as a {\it staircase tree}.

When $|a| = |b_1\cdots b_n| +1$, Proposition~\ref{prop:GSCorr} may be extended to show that staircase trees correspond (mod $2$) to chord paths with convex corners.  Again, one should assume that a product metric $\underbrace{(g_x \times \ldots \times g_x)}_{n+1 \,\,\textit{times}}$ is used on the fibers of $P_{n+1}$.  For a generalized difference function $w^{n+1}_{i,j}$, the fiber critical set is identified with that of the original difference function $w = w^2_{1,2}$: Map $(x, e_1, e_2) \in S_w$ to $S_{w^{n+1}_{i,j}}$ by taking $e_1$ and $e_2$ as the $i$-th and $j$-th coordinates and setting the remaining $e_k=0$. Hence, one can produce a chord path from a gradient staircase for $w^{n+1}_{i,j}$ by paying attention to the variables $x$, $e_i$, and $e_j$. 

Let us examine the branch points of a staircase tree. We restrict our attention to a $3$-valent vertex, $v$, where gradient staircases for $w_{i,j}$ (resp. $w_{i,k}$ and $w_{k,j}$) are assigned to the inward oriented edge (resp. outward oriented edges) for some $ i < k < j$; see Figure~\ref{fig:GradStairBranch}. Such a branch point will map to a point $(x, \mathbf{e}) \in P_{n+1}$ with\footnote{This is a consequence of the assumption $|a| = |b_1\cdots b_n| +1$, but need not be the case in general. See, for instance, the edge connecting the cases ${\bf L}\mathrm{BH}3 \leftrightarrow {\bf L}\mathrm{BH}4$ in Proof of Theorem \ref{thm:d2is0}.} $(x, e_i), (x, e_k), (x, e_j) \in S_F$ and $F(x, e_i) > F(x, e_k) > F(x, e_j)$.  This appears on the front projection of $L_F$ as the chord corresponding to $w_{i,j}$ breaking into two chords corresponding to $w_{i,k}$ and $w_{k,j}$ respectively; see Figure~\ref{fig:GradStairBranch}. Each of the new chords then proceeds to the left satisfying the requirements of a chord path. However, in order that $|a| = |b_1\cdots b_n| +1$, it must be the case that at every such branch point one of the new chords remains fixed as a constant trajectory at one of the $b_i$. For a Ng resolution form front diagram this appears as in Figure~\ref{fig:BranchAndCorners}; the constant trajectory at a branch point is conveniently recorded as a convex corner in a chord path. 

\begin{figure}
\labellist
\small\hair 2pt
\pinlabel {$i$} [tl] at 3 57
\pinlabel {$j$} [tl] at 43 57
\pinlabel {$k$} [tl] at 21 11
\pinlabel {$v$} [tl] at 32 36
\endlabellist
\centering
\includegraphics[scale=.7]{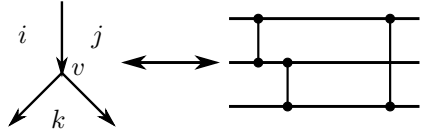}
\caption{The appearance of a branch point in a staircase tree (left) when viewed as a chord path (right).}
\label{fig:GradStairBranch}
\end{figure}

\begin{figure}
\centerline{\includegraphics[scale=.7]{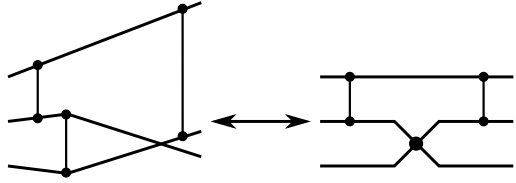}}
\caption{At a branch point, the constant edge is recorded combinatorially as a convex corner in a chord path.}
\label{fig:BranchAndCorners}
\end{figure}

We close this discussion by remarking that along each edge meeting at $v$ the parametrization of the corresponding gradient staircase is such that the vertex occurs in the middle of one of the steps, $\varphi_m$.  For instance, along the inwardly oriented edge corresponding to the difference function $w_{i,j}$, the $e_k$ coordinate remains zero since this is the case along $S_{w_{i,j}}$.  However, at the vertex a portion of a step occurs which relocates $e_k$ to the appropriate value while leaving the remaining coordinates constant.  This can always occur since $e_k$ appears as a negative definite quadratic summand in $w_{i,j}$. The outgoing edges, then begin by immediately returning $e_i$ or $e_j$ to $0$ via a partial step.  
\section{Relating the MCS-DGA to the CE-DGA}
\mylabel{ch:DGA-to-CE-DGA}

Given a Legendrian knot with front projection $\sfront$ and MCS $\sMCS$, the MCS-DGA $(\salg_{\sMCS}, d)$ and the CE-DGA $(\salg(\sfront), \df)$ are both generated by the crossings and right cusps of $\sfront$, regardless of the chosen MCS. The gradings of these generators are also identical in both DGAs. The question then arises, can we understand the differential of one of these DGAs in terms of the other? In the case of a special class of MCSs called \emph{$A$-form MCSs}, defined below and first introduced in \cite{Henry2011}, the answer is yes. 

In Section~\ref{sec:MCS-graphic} we uniquely encoded an MCS $\sMCS \in \sFMCS$ using implicit and explicit handleslide marks on $\sfront$. We say $\sMCS$ is \emph{simple} if there are no implicit handleslides at left cusps. In terms of a generating family, this corresponds to the introduction of a pair of canceling critical points of adjacent index away from the existing critical points. We let $\sSMCS \subset \sFMCS$ denote the set of simple MCSs. The implicit handleslide marks at a left cusp may be made explicit using MCS move 13, hence every MCS is equivalent to a simple MCS. 

The chain complexes of $\sMCS$ may be reconstructed inductively using the explicit and implicit handleslide marks and Conditions (1) - (5) of Definition~\ref{def:MCS}. In fact, in the case of a simple MCS, the chain complexes are determined by Conditions (1) - (5) of Definition~\ref{def:MCS} and the explicit handleslide marks. We no longer need to indicate implicit handleslide marks at the right cusps. Whereas we defined the marked front projection of $\sMCS$ to include both implicit and explicit handleslide marks, we define the \emph{simplified marked front projection} of $\sMCS \in \sSMCS$ to include only the explicit handleslide marks. 

\begin{definition}
An MCS $\sMCS \in \sSMCS$ is in \emph{$A$-form} if the simplified marked front projection of $\sMCS$ satisfies:

\begin{enumerate}
	\item Outside of a small neighborhood of the crossings of $\sfront$, $\sMCS$ has no handleslide marks; and 
	\item Within a small neighborhood of a crossings $q$, either $\sMCS$ has no handleslide marks or it has a single handleslide mark to the left of $q$ between the strands crossings at $q$. 
\end{enumerate}

If a handleslide mark appears in a small neighborhood of $q$, then we say $q$ is \emph{marked}. Finally, we define $\sAMCS \subset \sFMCS$ to be the set of all $A$-form MCSs.
\end{definition}

If we make the right-most explicit handleslide mark in Figure~\ref{f:MCS-trefoil-example} implicit using MCS Move 13, then the resulting MCS is in $A$-form. Every MCS is equivalent to an MCS in $A$-form and a bijection exists between $\sAMCS$ and $\sAugL$; see \cite{Henry2011}.

\begin{theorem}[\cite{Henry2011}]
\label{thm:aug-MCS-bijec}
Given $\sMCS \in \sAMCS$ we define an algebra map $\saug : (\salg(\sfront), \df) \to \zz_2$ by $\saug(q)=1$ for a crossing $q$ in $\sfront$ if and only if $q$ is marked in $\sMCS$. Then $\saug$ is an augmentation and this map from $\sAMCS$ to $\sAugL$ is a bijection.
\end{theorem}

Recall from Section~\ref{ch:Background} that an augmentation $\saug$ of $(\salg(\sfront), \df)$ determines a DGA $(\salg(\sfront), \df^{\saug})$ with differential $\df^{\saug} = \phi^{\saug} \circ \df \circ (\phi^{\saug})^{-1}$. The main result of this section proves the differentials of $(\salg_{\sMCS}, d)$ and $(\salg(\sfront), \df^{\saug})$ are equal.

\begin{theorem}
\label{thm:A-form-equivalence}
Given an $A$-form MCS $\sMCS$ with corresponding augmentation $\saug$, there is a grading-preserving bijection between the generators of $\salg_{\sMCS}$ and $\salg(\sfront)$ such that $d = \df^{\saug}$, and hence $(\salg_{\sMCS}, d) = (\salg(\sfront), \df^{\saug_{\sMCS}})$.
\end{theorem}

The proof of Theorem~\ref{thm:A-form-equivalence} relies on a careful comparison of the admissible disks contributing to the differential of $(\salg(\sfront), \df^{\saug})$ and the chord paths contributing to the differential of $(\salg_{\sMCS}, d)$. We begin by describing the computation of $\df^{\saug}$ in terms of admissible disk pairs.

\begin{definition}
Given $\saug \in \sAugL$ an \emph{admissible disk pair} $(D, \Psi)$ is an admissible disk $D$ and a subset $\Psi$, possibly trivial, of the augmented convex corners of $D$. The monomial $w(D, \Psi)$ is the result of removing the generators in $\Psi$ from $w(D)$. If $\Psi$ contains all of the convex corners of $D$, then we define $w(D, \Psi)=1$.

Suppose $a$ is a crossing or right cusp, $b$ is a crossing between strands $k$ and $k+1$, and $i < k$. We define $\Delta(a; b, [i, k])$ to be the set of admissible disk pairs $(D, \Psi)$ satisfying: $D$ originates at $a$; $b \notin \Psi$ and $b$ is a convex corner of $D$; the intersection of $D$ with a vertical line $\{x\} \times \rr$ where $x$ is just to the left of $b$ is a vertical line between strands $i$ and $k$; and if $c$ is a convex corner of $D$ and $c$ is to the left of $b$, then $c \in \Psi$. Given $j > k+1$, the set $\Delta(a; b, [k+1, j])$ is similarly defined.
\end{definition}

\begin{remark}
Given an admissible disk pair $(D, \Psi) \in \Delta(a; b, [i, k])$, we have assumed $b$ is a convex corner of $D$ and $b \notin \Psi$, so $b$ appears in the monomial $w(D, \Psi)$ and so $w(D, \Psi) \neq 1$.
\end{remark}

\begin{proposition}
\label{prop:d-aug}
The differential of $(\salg(\sfront), \df^{\saug})$ is computed by:

\begin{equation}
\label{eq:d-aug}
\df^{\saug} a =  \sum_{\substack{b \in Q(\sfront)\\i<k\\(D, \Psi) \in \Delta(a; b, [i, k])}} w(D, \Psi) + \sum_{\substack{b \in Q(\sfront)\\j>k+1\\(D, \Psi) \in \Delta(a; b, [k+1, j])}} w(D, \Psi) 
\end{equation}
where strands $k$ and $k+1$ cross at $b$, hence in both sums $k$ and $k+1$ depend on $b$.
\end{proposition}

\begin{proof}

Let $D \in \Delta(a,w)$ where $ w = \prod_{i=1}^l \alpha_i$, $\alpha_i \in Q(L)$. Then,
\begin{eqnarray*}
\label{eq:d-aug-disks}
\df^{\saug}a &=& \phi^{\saug} \circ \df \circ (\phi^{\saug})^{-1} (a) \\
&=& \phi^{\saug} \left( \prod_{i=1}^l \alpha_i + \hdots \right) \nonumber \\
&=& \prod_{i=1}^l (\alpha_i + \saug(\alpha_i)) + \phi^{\saug}(\hdots) \nonumber 
\end{eqnarray*}
Hence, $D$ contributes a monomial for each, possibly trivial, subset $\Psi$ of the augmented, convex corners of $D$ and this monomial is the result of removing the generators in $\Psi$ from $w(D)$. We can conclude $\df^{\saug} a = \sum w(D, \Psi) $ where the sum is over all admissible disk pairs $(D, \Psi)$ where $D$ originates at $a$. The augmentation condition $\saug \circ \df = 0$ ensures $\langle \df^{\saug}a, 1 \rangle = 0$, so we can disregard admissible disk pairs satisfying $w(D, \Psi)=1$ in the sum $ \sum w(D, \Psi) $. 

An admissible disk is embedded away from singularities since $L$ is nearly plat. Furthermore, an admissible disk pair $(D, \Psi)$ with $w(D, \Psi)\neq 1$ has a unique convex crossing $b$ so that if $c$ is a convex corner of $D$ and $c$ is to the left of $b$, then $c \in \Psi$. There exists some unique $i$ or $j$ so that $(D, \Psi) \in \Delta(a; b, [i, k])$ or $(D, \Psi) \in \Delta(a; b, [k+1, j])$. Thus, the monomials in $ \sum w(D, \Psi) $ correspond to the monomials in the right-hand side of equation~(\ref{eq:d-aug}).
\end{proof}

Given an MCS $\sMCS =\left( \{(C_m, d_m)\}, \{x_m\}, H \right)$ with MCS-DGA $(\aac_{\sMCS}, d)$, we can give a formula for $d$ that looks similar to equation~(\ref{eq:d-aug}), but uses chord paths rather than admissible disk pairs. 

\begin{definition}
Suppose $\sMCS =\left( \{(C_m, d_m)\}, \{x_m\}, H \right)$ is an MCS, $a$ is a crossing or right cusp, and $b$ is a crossing between strands $k$ and $k+1$ in the tangle between $x_p$ and $x_{p+1}$. Then for $i<k$ we let $\mathcal{M}^{\sMCS}(a; b, [i, k])$ denote the set of chord paths with convex corners with elements $\Lambda = (\lambda_1, \lambda_2, \ldots, \lambda_n)$ originating at $a$ and terminating at $b$ with $\lambda_n = (x_p,[i,k])$. Given $j > k+1$, the set $\mathcal{M}^{\sMCS}(a; b, [k+1, j])$ is similarly defined.
\end{definition}

\begin{proposition}
\label{prop:d-MCS}
The differential of the MCS-DGA $(\salg_{\sMCS}, d)$ is computed by:

\begin{equation}
\label{eqn:d-MCS}
d a =  \sum_{\substack{b \in Q(\sfront)\\i<k\\\Lambda \in \mathcal{M}^{\sMCS}(a; b, [i, k])}} w(\Lambda) + \sum_{\substack{b \in Q(\sfront)\\j>k+1\\\Lambda \in \mathcal{M}^{\sMCS}(a; b, [k+1, j])}} w(\Lambda) 
\end{equation}
where strands $k$ and $k+1$ cross $b$, hence in the sum $k$ and $k+1$ depend on $b$.
\end{proposition}

\begin{proof}
The differential $d$ of the MCS-DGA $(\salg_{\sMCS}, d)$ is defined in Section~\ref{sec:MCS-DGA-defn} to be $$da = \sum_{n=1}^{\infty} \sum \# \mathcal{M}^{\sMCS}(a; b_1, \hdots, b_n) b_1 \cdots b_n $$ where $\mathcal{M}^{\sMCS}(a; b_1, \hdots, b_n)$ is the set of chord paths with convex corners originating at $a$ with $w(\Lambda) = b_1 \cdots b_n$ and the inner sum is over all monomials $b_1 \cdots b_n$ such that $|a|=|b_1 \cdots b_n| +1$. But by Proposition~\ref{prop:CPDegree}, $\mathcal{M}^{\sMCS}(a; b_1, \hdots, b_n) \neq \emptyset$ implies $|a|=|b_1 \cdots b_n| +1$, hence, every chord path with convex corners originating at $a$ contributes to $da$. This allows us to rewrite $d$ as: 
\begin{equation}
\label{eqn:prop-d-MCS}
da =  \sum_{\Lambda \in \mathcal{M}^{\sMCS}(a)} w(\Lambda) 
\end{equation}
where $\mathcal{M}^{\sMCS}(a)$ is the set of chord paths with convex corners originating at $a$. The set $\mathcal{M}^{\sMCS}(a)$ decomposes into a disjoint union based on the terminating chord of each $\Lambda \in \mathcal{M}^{\sMCS}(a)$: 
\begin{equation}
\label{eqn:prop-d-MCS-2}
\mathcal{M}^{\sMCS}(a)= \coprod_{b \in Q(L)} \left (\coprod_{i<k} \mathcal{M}^{\sMCS}(a; b, [i, k]) \cup \coprod_{j>k+1} \mathcal{M}^{\sMCS}(a; b, [k+1, j])  \right )
\end{equation}
where strands $k$ and $k+1$ cross at $b$. Equation~(\ref{eqn:d-MCS}) follows from equations~(\ref{eqn:prop-d-MCS})~and~(\ref{eqn:prop-d-MCS-2}). 
\end{proof}

Suppose $a, b, i$ and $k$ are as in both Propositions~\ref{prop:d-aug} and \ref{prop:d-MCS}. Each admissible disk in $\Delta(a; b, [i, k])$ originates at $a$ and terminates at some left cusp to the left of $b$. On the other hand, each chord path of $\mathcal{M}^{\sMCS}(a; b, [i, k])$ originates at $a$ and terminates at $b$. In the case of an $A$-form MCS, we will find a correspondence between these sets. We do so by introducing the notion of a ``gradient path,'' which is a sequence of chords on $\sfront$ originating at a left cusp and encoding a sequence of fiber-wise gradient flowlines. Every admissible disk in $\Delta(a; b, [i, k])$ decomposes into a gradient path beginning at a left cusp and terminating at $b$ plus a chord path in $\mathcal{M}^{\sMCS}(a; b, [i, k])$. In addition, we can ``glue'' a gradient path and a chord path together to produce an admissible disk pair; see Figure~\ref{f:bijection-example}. The result is a bijection that allows us to prove Theorem~\ref{thm:A-form-equivalence} from Propositions~\ref{prop:d-aug}~and~\ref{prop:d-MCS}. This gluing argument depends on the fact that the placement of handleslide marks in an $A$-form MCS is very restrictive. In the next section, we introduce gradient paths. The proof of Theorem~\ref{thm:A-form-equivalence} follows in Section~\ref{sec:isomorphism}.

\subsection{Gradient paths in an MCS} 
\mylabel{sec:grad-paths}

We will use the notion of a chord defined in Section~\ref{sec:chord-paths} to define a sequence of chords encoding the behavior of fiberwise gradient flowlines in the chain complexes in an MCS $\sMCS =\left( \{(C_m, d_m)\}, \{x_m\}, H \right)$. The resulting ``gradient path'' is a sequence of chords originating at a left cusp and represented on the front projection as a sequence of dotted arrows. 

\begin{definition} \label{def:GradientPath}
 A {\it gradient path} from a left cusp $c$ in the MCS $\sMCS =\left( \{(C_m, d_m)\}, \{x_m\}, H \right)$ is a finite sequence of chords $\Omega = (\omega_1, \omega_2, \ldots, \omega_n)$, so that 
\begin{enumerate} 

\item There exists $k \geq 0$ so that for all $1 \leq i \leq n$ the $x$-coordinate of $\omega_i$ is $x_{i+k} \in \{x_m\}$. Therefore, for all $1 \leq r < n$, $\omega_{r+1}$ is located to the \emph{right} of $\omega_r$.  

\item Assuming $c$ occurs between $x_{l-1}$ and $x_{l}$ and involves strands $k$ and $k+1$, we require $\omega_1 = ( x_l, [k, k+1])$; see Figure~\ref{f:gradient-path-originating-terminal}. We say $\omega_1$ is the \emph{originating chord} of $\Omega$.

\item Suppose $1 \leq r < n$, $\omega_r = (x_{m}, [i,j])$, and the chord $\omega_{r+1} = (x_{m+1}, [i',j'])$.  Then, $i'$ and $j'$ are required to satisfy some restrictions depending on the type of tangle appearing in the region $x_m \leq x \leq x_{m+1}$ as follows:
\end{enumerate}

\smallskip

\noindent (a) {\bf Crossing:} Assuming the crossing involves strands $k$ and $k+1$, we require $i' = \sigma(i), j' = \sigma(j)$ where $\sigma$ denotes the transposition $(k \,\, k+1)$.  A chord of the form $\omega_r = (x_{m}, [k,k+1])$ is not allowed to occur to the left of a crossing in a gradient path.

\smallskip

\noindent (b) {\bf Left/Right cusp:}  Suppose the strands meeting at the cusp are labeled $k$ and $k+1$. In the case of a left (resp. right) cusp, we require $i' = \tau(i)$, $j' = \tau(j)$ (resp. $i = \tau(i')$, $j = \tau(j')$) were $\tau$ is defined by $\tau(l) = 
\begin{cases} 
l & \mbox{if $l < k$} \\
l+2 & \mbox{if $l \geq k$}.
\end{cases}$

\noindent (c) {\bf Handleslide:}  Suppose the handleslide occurs between strands $k$ and $l$ with $k < l$.  Here, $i' = i$ and $j' =j$ is always allowed.  In addition, if $i= l$ and $l < j$ we also allow $i'=k$ and $j' =j$.  Analogously, if $j=k$ and $i < k$ then we allow $i'=i$ and $j' = l$.  This rule can be viewed as allowing the endpoint of a gradient chord to possibly jump along the handleslide.  Such a jump will increase the length of the chord.

\end{definition}

\begin{figure}[t]
\labellist
\small\hair 2pt
\pinlabel {$k$} [tl] at 90 70
\pinlabel {$k+1$} [tl] at 90 47
\pinlabel {$c$} [tl] at 0 55
\pinlabel {$\omega_1$} [tl] at 55 78
\pinlabel {$x_l$} [tl] at 55 16
\pinlabel {$i$} [bl] at 297 74
\pinlabel {$j$} [tl] at 297 64
\pinlabel {$x_p$} [tl] at 235 8
\endlabellist
\centering
\includegraphics[scale=.7]{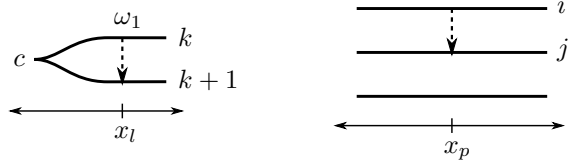}
\caption{The chord on the left is an originating chord of a gradient path. The chord on the right represents a terminating chord of a gradient path in $\mathcal{G}^{\sMCS}(x_p, [i, j])$.}
\label{f:gradient-path-originating-terminal}
\end{figure}

\begin{figure}[t]
\centering
\includegraphics[scale=.7]{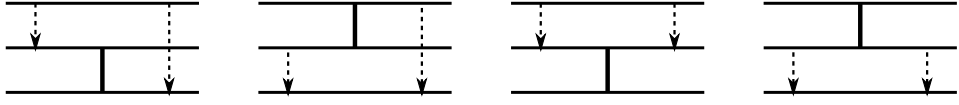}
\caption{Consecutive chords of a gradient path interacting with a handleslide.}
\label{f:gradient-path-handleslides}
\end{figure}

The left two figures in Figures~\ref{f:gradient-path-handleslides} show a jump that may occur in a gradient path as the chords move past a handleslide. However, a gradient path is not required to make such a jump, as the right two figures in Figures~\ref{f:gradient-path-handleslides} indicate. In all other cases where chords in a gradient path encounter handleslides and crossings, the consecutive chords have endpoints on the same strands. 

We let $\mathcal{G}^{\sMCS}(x_p, [i, j])$ denote the set of gradient paths $\Omega = (\omega_1, \omega_2, \ldots, \omega_n)$ with $\omega_n = (x_p,[i,j])$; see Figure~\ref{f:gradient-path-originating-terminal}. This set includes gradient paths originating at all possible left cusps.

\begin{lemma} 
\label{lemma:gradient-paths}
Suppose $\sMCS =\left( \{(C_m, d_m)\}, \{x_m\}, H \right) \in \sAMCS$ is an $A$-form MCS on a nearly plat front projection $L$ and suppose $x_p \in \{x_m\}$ is to the left of the right cusps of $L$. Then  for all $i<j$, 
\begin{equation}
\label{eq:grad-paths}
\# \mathcal{G}^{\sMCS}(x_p, [i, j])= \langle d_p e_i, e_j \rangle.
\end{equation}
\end{lemma}

\begin{proof}
We will induct on $p$, $i$ and $j$. Suppose $p=1$. The left-cusp to the left of $x_p$ involves strands $1$ and $2$. The set $\mathcal{G}^{\sMCS}(x_1, [1, 2])$ contains one element, namely the gradient path $\Omega = (\omega_1)$ with $\omega_1 = (x_1, [1,2])$. Hence, we have $\# \mathcal{G}^{\sMCS}(x_1, [1, 2]) = 1$. Definition~\ref{def:MCS}~(4) gives $\langle d_1 e_1, e_2 \rangle=1$, so we can conclude $\# \mathcal{G}^{\sMCS}(x_1, [1, 2]) = \langle d_1 e_1, e_2 \rangle$ as desired. 

Suppose $p>1$ and fix $i < j$. The proof breaks into cases depending on the tangle between $x_{p-1}$ and $x_{p}$. The hypothesis that $x_p$ occurs to the left of the right cusps of $L$ allows us to disregard the possibility of a right cusp between $x_{p-1}$ and $x_{p}$.

\textbf{Crossing}: Suppose the crossing occurs between strands $k$ and $k+1$ and let $\sigma$ denotes the transposition $(k \,\, k+1)$. Given a gradient path $\Omega = (\omega_1, \hdots, \omega_p) \in \mathcal{G}^{\sMCS}(x_p, [i, j])$, Definition~\ref{def:GradientPath}~(3a) gives $\omega_{p-1} = (x_{p-1}, [\sigma(i), \sigma(j)])$. Every gradient path in $\mathcal{G}^{\sMCS}(x_{p-1}, [\sigma(i),\sigma(j)])$ can be extended to a path in $\mathcal{G}^{\sMCS}(x_p, [i, j])$ with $\omega_p = (x_p, [i, j])$. Hence, we have $\# \mathcal{G}^{\sMCS}(x_p, [i, j]) = \# \mathcal{G}^{\sMCS}(x_{p-1}, [\sigma(i), \sigma(j)])$. Definition~\ref{def:MCS}~(5a) requires $\langle d_{p-1} e_{\sigma(i)}, e_{\sigma(j)} \rangle = \langle d_p e_{i}, e_{j} \rangle$. Equation~(\ref{eq:grad-paths}) follows from these two observations and the inductive hypothesis $\# \mathcal{G}^{\sMCS}(x_{p-1}, [\sigma(i), \sigma(j)]) = \langle d_{p-1} e_{\sigma(i)}, e_{\sigma(j)} \rangle$.  	
	
\textbf{Left cusp}: Suppose the cusp occurs between strands $k$ and $k+1$. Recall the map $\tau$ on indices defined by $\tau(l) = 
\begin{cases} 
l & \mbox{if $l < k$} \\
l+2 & \mbox{if $l \geq k$}.
\end{cases}$ The MCS $\sMCS$ is simple by assumption, hence we have:
\begin{align}
	& \langle d_p e_{k}, e_{k+1} \rangle=1; \label{eq:lemma-grad-paths-1} \\
	& \langle d_p e_{s}, e_{r} \rangle = 0 \mbox{ if } |\{r, s\} \cap \{k, k+1\}| = 1; \mbox{ and } \label{eq:lemma-grad-paths-2} \\ 
	& \langle d_p e_{r}, e_{s} \rangle=\langle d_{p-1} e_{\tau^{-1}(r)}, e_{\tau^{-1}(s)} \rangle \mbox{ if } r,s \notin \{k, k+1\}. \label{eq:lemma-grad-paths-3}
\end{align}
If $i, j \notin \{k, k+1\}$, then equation~(\ref{eq:grad-paths}) follows from equation~(\ref{eq:lemma-grad-paths-3}) using the same argument given in the case of a crossing. Definition~\ref{def:GradientPath}~(3b) ensures $\mathcal{G}^{\sMCS}(x_p, [i, j])$ is empty if $|\{r, s\} \cap \{k, k+1\}| = 1$. Equation~(\ref{eq:grad-paths}) follows from this observation and equation~(\ref{eq:lemma-grad-paths-2}). Finally, if $i=k$ and $j=k+1$, then $\mathcal{G}^{\sMCS}(x_p, [i, j])$ contains one element, similar to the base case of $p=1$. Equation~(\ref{eq:grad-paths}) then follows from this observation and equation~(\ref{eq:lemma-grad-paths-1}).

\textbf{Handleslide}: 
Suppose the handleslide occurs between strands $k$ and $l$. By Definition~\ref{def:MCS}~(5d), $\langle d_p e_i, e_j \rangle \neq \langle d_{p-1} e_i, e_j \rangle$ if and only if $j=l$ and $\langle d_{p-1} e_i, e_k \rangle=1$ or $i=k$ and $\langle d_{p-1} e_l, e_j \rangle=1$.

If $j \neq l$ or $i \neq k$, then $\langle d_p e_i, e_j \rangle = \langle d_{p-1} e_i, e_j \rangle$. There is a bijection between $\mathcal{G}^{\sMCS}(x_p, [i, j])$ and $\mathcal{G}^{\sMCS}(x_{p-1}, [i, j])$ since a path in $\mathcal{G}^{\sMCS}(x_{p-1}, [i, j])$ can be uniquely extended to a path in $\mathcal{G}^{\sMCS}(x_p, [i, j])$. The extended path can not jump along the handleslide. Hence, equation~(\ref{eq:grad-paths}) follows in this case. 

Suppose $j=l$. There is a bijection from $\mathcal{G}^{\sMCS}(x_{p-1}, [i, l]) \cup \mathcal{G}^{\sMCS}(x_{p-1}, [i, k])$ to $\mathcal{G}^{\sMCS}(x_p, [i, l])$. A path in $\mathcal{G}^{\sMCS}(x_{p-1}, [i, l])$ extends to a path in $\mathcal{G}^{\sMCS}(x_p, [i, l])$ by not jumping along the handleslide and a path in $\mathcal{G}^{\sMCS}(x_{p-1}, [i, k])$ extends to a path in $\mathcal{G}^{\sMCS}(x_p, [i, l])$ by jumping along the handleslide. Therefore, we can conclude:

\begin{align}
\# \mathcal{G}^{\sMCS}(x_p, [i, l]) &= \# \left (\mathcal{G}^{\sMCS}(x_{p-1}, [i, l]) \cup \mathcal{G}^{\sMCS}(x_{p-1}, [i, k]) \right ) \nonumber \\
&= \langle d_{p-1} e_i, e_l \rangle + \langle d_{p-1} e_i, e_k \rangle \nonumber \\
&= \langle d_{p} e_i, e_l \rangle \nonumber
\end{align}

The second equality follows from the inductive hypothesis and the final equality follows from the observation made above that $\langle d_{p} e_i, e_l \rangle \neq \langle d_{p-1} e_i, e_l \rangle$ if and only if $\langle d_{p-1} e_i, e_k \rangle = 1$. The argument for the case $i=k$ is similar to the case $j=l$. 
\end{proof}

Suppose $b$ is a crossing between strands $k$ and $k+1$ in the tangle between $x_p$ and $x_{p+1}$. For $i < k$ and $j > k+1$ we let $\mathcal{G}^{\sMCS}(b, [i, k]) = \mathcal{G}^{\sMCS}(x_p, [i, k])$ and $\mathcal{G}^{\sMCS}(b, [k+1, j])=\mathcal{G}^{\sMCS}(x_p, [k+1, j]).$ We say such gradient paths \emph{terminate at $b$}. It follows immediately from Lemma~\ref{lemma:gradient-paths} that $\# \mathcal{G}^{\sMCS}(b, [i, k])= \langle d_p e_i, e_k \rangle$ and $\# \mathcal{G}^{\sMCS}(b, [k+1, j])= \langle d_p e_{k+1}, e_j \rangle.$

\subsection{An isomorphism between $(\salg_{\sMCS}, d)$ and $(\salg(\sfront), \df^{\saug})$} 
\mylabel{sec:isomorphism}

In this section we prove Theorem~\ref{thm:A-form-equivalence} using the argument briefly sketched in the paragraph preceding Section~\ref{sec:grad-paths}. Fix an $A$-form MCS $\sMCS \in \sAMCS$ with $\sMCS =\left( \{(C_m, d_m)\}, \{x_m\}, H \right)$ and MCS-DGA $(\salg_{\sMCS}, d)$ and suppose $\saug \in \sAugL$ is the augmentation associated to $\sMCS$ by the bijection in Theorem~\ref{thm:aug-MCS-bijec}. Recall, in $\sMCS$, a handleslide appears to the left of each augmented crossings of $\saug$ between the crossing strands.

\begin{lemma}
\label{lemma:disk-bijection}
Suppose $a$ is a crossing or right cusp and $b$ is a crossing between strands $k$ and $k+1$ in the tangle between $x_p$ and $x_{p+1}$. Then for $i < k$ there is a  bijection from $\mathcal{G}^{\sMCS}(b, [i, k]) \times \mathcal{M}^{\sMCS}(a; b, [i, k])$ to $\Delta(a; b, [i, k])$ with $(\Omega, \Lambda) \mapsto (D_{\Omega, \Lambda}, \Psi_{\Omega, \Lambda})$ so that $w(D_{\Omega, \Lambda}) = w(\Lambda)$. A similar bijection exists from $\mathcal{G}^{\sMCS}(b, [k+1, j]) \times \mathcal{M}^{\sMCS}(a; b, [k+1, j])$ to $\Delta(a; b, [k+1, j])$ where $j > k+1$.
\end{lemma}

\begin{figure}[t]
\labellist
\small\hair 2pt
\pinlabel {$b$} [tl] at 25 60
\pinlabel {(a)} [tl] at 37 35
\pinlabel {(b)} [tl] at 172 35
\pinlabel {(c)} [tl] at 309 35
\pinlabel {(d)} [tl] at 445 35
\pinlabel {(e)} [tl] at 582 35
\pinlabel {(f)} [tl] at 695 35
\endlabellist
\centering
\includegraphics[scale=.45]{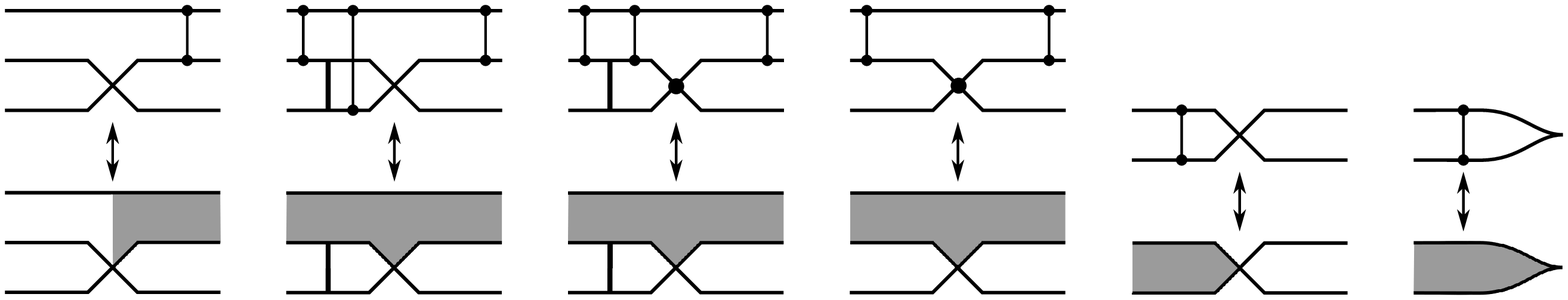}
\caption{The correspondence between local neighborhoods of an admissible disk and chords in a chord path.}
\label{f:bijection-1}
\end{figure}

\begin{figure}[t]
\labellist
\small\hair 2pt
\pinlabel {$b$} [tl] at 365 53
\pinlabel {(a)} [tl] at 34 28
\pinlabel {(b)} [tl] at 188 28
\pinlabel {(c)} [tl] at 373 28
\endlabellist
\centering
\includegraphics[scale=.5]{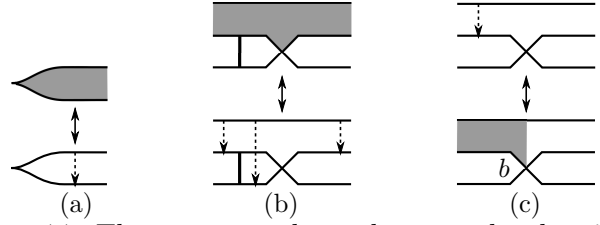}
\caption{The correspondence between local neighborhoods of an admissible disk and chords in a gradient path.}
\label{f:bijection-3}
\end{figure}

\begin{proof}
We will prove the case $i<k$. The case $j>k+1$ follows from a similar argument. 

Suppose $(\Omega, \Lambda) \in \mathcal{G}^{\sMCS}(b, [i, k]) \times \mathcal{M}^{\sMCS}(a; b, [i, k])$. Let $c$ denote the originating left cusp of $\Omega$. We will use the chords in $\Omega$ and $\Lambda$ as a recipe to build an admissible disk $D_{\Omega, \Lambda}$ originating at $c$ and terminating at $a$. An example of this process is given in Figure~\ref{f:bijection-example}. We will say $\Lambda$ (resp. $\Omega$) \emph{jumps along $e$} if $e$ is an augmented crossing and a chord in $\Lambda$ (resp. $\Omega$) jumps along the handleslide mark just to the left of $e$; see Figure~\ref{f:bijection-1}~(b) and Figure~\ref{f:bijection-3}~(b). At most one chord of $\Omega \cup \Lambda$ can jump along any given crossing. 

Since $\sfront$ is nearly plat, the intersection of an admissible disk with a vertical line in the $xz$-plane is a connected, vertical arc. Thus, we will define the disk $D_{\Omega, \Lambda}$ by describing the intersection $D_{\Omega, \Lambda} \cap (\{x\} \times \rr)$ for values of $x$ between $c$ and $a$. In a small neighborhood of $c$, $D_{\Omega, \Lambda} \cap (\{x\} \times \rr)$ is a vertical line connecting the two strands that meet at $c$; see Figure~\ref{f:bijection-3}~(a). The disk $D_{\Omega, \Lambda}$ has a convex corner at a crossing $e$ if and only if $e$ satisfies one of the following:
\begin{enumerate}
	\item $e=b$ (Figures~\ref{f:bijection-1}~(a)~and~\ref{f:bijection-3}~(c));
	\item $e$ is a convex corner of $\Lambda$ (Figure~\ref{f:bijection-1}~(c)~and~(d));
	\item $\Lambda$ jump along $e$ (Figure~\ref{f:bijection-1}~(b)); or 
	\item $\Omega$ jumps at $e$ (Figure~\ref{f:bijection-3}~(b)).
\end{enumerate}
In each case, the convex corner of $D_{\Omega, \Lambda}$ at $e$ is determined by the chords of $\Omega$ and $\Lambda$ as described in Figures~\ref{f:bijection-1}~and~\ref{f:bijection-3}. Away from the convex corners, we require the endpoints of the vertical line $D_{\Omega, \Lambda} \cap (\{x\} \times \rr)$ to vary smoothly along the strands of $\sfront$. This requirement ensures that away from the convex corners the vertical mark $D_{\Omega, \Lambda} \cap (\{x\} \times \rr)$ agrees with the chord at $x$ in $\Omega \cup \Lambda$. In particular, $D_{\Omega, \Lambda}$ looks like either Figure~\ref{f:bijection-1}~(e)~or~\ref{f:bijection-1}~(f) near $a$. Thus $(D_{\Omega, \Lambda}, \Psi_{\Omega, \Lambda})$ is an admissible disk pair in $\mathcal{M}^{\sMCS}(a; b, [i, k])$ where we define $\Psi_{\Omega, \Lambda} = \{ e \in Q(L) | $ $\Omega$ or $\Lambda$ jumps along $e \}$. Each crossing in $\Psi_{\Omega, \Lambda}$ corresponds to an augmented crossing that is convex in either $\Omega$ or $\Lambda$, but does not appear in $w(\Lambda)$. The remaining convex corners of $D_{\Omega, \Lambda}$ correspond to the convex corners of $\Lambda$ contributing to $w(\Lambda)$. Hence, we have $w(D_{\Omega, \Lambda}) = w(\Lambda)$ as desired. This map is injective since the pair $(\Omega, \Lambda) \in \mathcal{G}^{\sMCS}(b, [i, k]) \times \mathcal{M}^{\sMCS}(a; b, [i, k])$ is determined by $c$ and the behavior of $\Omega$ and $\Lambda$ near crossings, all of which is encoded by $D_{\Omega, \Lambda}$.

Suppose $(D, \Psi) \in \Delta(a; b, [i, k])$ terminates at $c$. We can ``unglue'' $(D, \Psi) \in \Delta(a; b, [i, k])$ to create a pair $(\Omega, \Lambda) \in \mathcal{G}^{\sMCS}(b, [i, k]) \times \mathcal{M}^{\sMCS}(a; b, [i, k])$ with $\Omega$ originating at $c$ so that $D=D_{\Omega, \Lambda}$ and $w(D) = w(\Lambda)$ and thus prove surjectivity. Away from convex corners of $D$, the chord of $\Omega$ (resp. $\Lambda$) at $x_p$ is the intersection of $D$ with $\{x_p\} \times \rr$. If $e \in \Psi$ (resp. $e \notin \Psi$) is a convex corner of $D$, the chords of $\Lambda$ and $\Omega$ (resp. $\Lambda$) are determined by $D$ as described in Figures~\ref{f:bijection-1}~(b) and \ref{f:bijection-3}~(b) (resp. Figure~\ref{f:bijection-1}~(c)~and~(d)). Near $b$ the terminating chords of $\Lambda$ and $\Omega$ are described in Figures~\ref{f:bijection-1}~(a) and \ref{f:bijection-3}~(c). The correspondences given in Figures~\ref{f:bijection-1} and \ref{f:bijection-3} ensure $D=D_{\Omega, \Lambda}$ and $w(D) = w(\Lambda)$ as desired.
\end{proof}

\begin{figure}[t]
\labellist
\small\hair 2pt
\pinlabel {$c$} [tl] at 20 185
\pinlabel {$b$} [tl] at 175 230
\pinlabel {$a$} [tl] at 495 185
\pinlabel {$e_1$} [tl] at 118 140
\pinlabel {$e_2$} [tl] at 285 230
\pinlabel {$e_3$} [tl] at 395 230
\pinlabel {$c$} [tl] at 20 73
\pinlabel {$b$} [tl] at 175 118 
\pinlabel {$a$} [tl] at 495 73
\pinlabel {$e_1$} [tl] at 118 30
\pinlabel {$e_2$} [tl] at 285 118
\pinlabel {$e_3$} [tl] at 395 118
\endlabellist
\centering
\includegraphics[scale=.6]{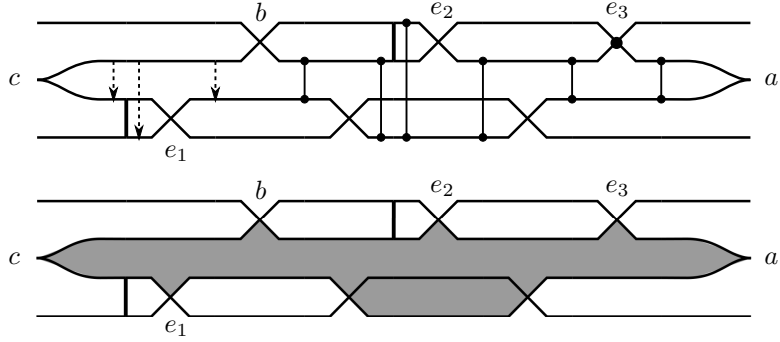}
\caption{A gradient/chord path pair $(\Omega, \Lambda)$ and its corresponding admissible disk pair $(D_{\Omega, \Lambda}, \Psi_{\Omega, \Lambda})$ with $\Psi_{\Omega, \Lambda}=\{e_1, e_2\}$ and $w(D_{\Omega, \Lambda})=e_3 b$.}
\label{f:bijection-example}
\end{figure}

Combining Propositions~\ref{prop:d-aug},~\ref{prop:d-MCS},~and~\ref{lemma:disk-bijection}, we can now prove Theorem~\ref{thm:A-form-equivalence}. 

\begin{proof}[Proof of Theorem~\ref{thm:A-form-equivalence}]
The MCS-DGA $(\salg_{\sMCS}, d)$ and CE-DGA $(\salg(\sfront), \df^{\saug})$ have the same graded generators, hence, it remains to show $d = \df^{\saug}$. Suppose $a \in Q(L)$ is a crossing or right cusp, $b \in Q(L)$ is a crossing between strands $k$ and $k+1$  in the tangle between $x_p$ and $x_{p+1}$, and $i < k$. Then we have,
\begin{align}
\sum_{(D, \Psi) \in \Delta(a; b, [i, k])} w(D, \Psi) &= \sum_{(\Omega, \Lambda) \in \mathcal{G}^{\sMCS}(b, [i, k]) \times \mathcal{M}^{\sMCS}(a; b, [i, k])} w(\Lambda) \\ \nonumber
&= \# \mathcal{G}^{\sMCS}(b, [i, k]) \sum_{\Lambda \in\mathcal{M}^{\sMCS}(a; b, [i, k])} w(\Lambda) \\ \nonumber
&= \langle d_p e_i, e_k \rangle \sum_{\Lambda \in\mathcal{M}^{\sMCS}(a; b, [i, k])} w(\Lambda) \nonumber
\end{align}
where the first equality follows from the word-preserving bijection from $\mathcal{G}^{\sMCS}(b, [i, k]) \times \mathcal{M}^{\sMCS}(a; b, [i, k])$ to $\Delta(a; b, [i, k])$ given in Lemma~\ref{lemma:disk-bijection}, the second equality is modular arithmetic, and the third equality follows from Lemma~\ref{lemma:gradient-paths}. If $\langle d_p e_i, e_k \rangle = 0$, then $\mathcal{M}^{\sMCS}(a; b, [i, k])= \emptyset$ and if $\mathcal{M}^{\sMCS}(a; b, [i, k]) \neq \emptyset$, then $\langle d_p e_i, e_k \rangle = 1$. Hence, we have
\begin{equation}
\label{eqn:thm}
\sum_{(D, \Psi) \in \Delta(a; b, [i, k])} w(D, \Psi) = \sum_{\Lambda \in\mathcal{M}^{\sMCS}(a; b, [i, k])} w(\Lambda).
\end{equation}
A similar equality holds for $j>k+1$. Finally, $da = \df^{\saug}a$ follows from equation~(\ref{eqn:thm}) and Propositions~\ref{prop:d-aug}~and~\ref{prop:d-MCS}.
\end{proof}

Theorems~\ref{thm:A-form-equivalence}~and~\ref{thm:linear-iso}, along with our previous observation that every MCS is equivalent to an $A$-form MCS, imply the following.

\begin{corollary}
\label{cor:linear-homology}
Let $H^{\sMCS}_{*}(\sfront)$ denote the homology groups of $(A_{\sMCS}, d_1)$ for an MCS $\sMCS \in \sFMCS$ and let $LCH_{*}^{\saug}(\sfront)$ denote the homology groups of $(\salg_1(\sfront), \df^{\saug}_1)$ for an augmentation $\saug \in \sAugL$. Then $$\{H^{\sMCS}_{*}(\sfront)\}_{\sMCS \in \sFMCS} = \{LCH_{*}^{\saug}(\sfront)\}_{\saug \in \sAugL}.$$
\end{corollary}

Chekanov proved in \cite{Chekanov2002} that $\{LCH_{*}^{\saug}(\sfront)\}_{\saug \in \sAugL}$ is a Legendrian isotopy invariant, hence, $\{H^{\sMCS}_{*}(\sfront)\}_{\sMCS \in \sFMCS}$ is as well.

\addcontentsline{toc}{section}{Bibliography} 
\bibliographystyle{amsplain}
\bibliography{Bibliography}

\providecommand{\bysame}{\leavevmode\hbox to3em{\hrulefill}\thinspace}
\providecommand{\MR}{\relax\ifhmode\unskip\space\fi MR }
\providecommand{\MRhref}[2]{%
  \href{http://www.ams.org/mathscinet-getitem?mr=#1}{#2}
}
\providecommand{\href}[2]{#2}
\begin{thebibliography}{10}

\bibitem{Barannikov1994}
S.~A. Barannikov, \emph{The framed {M}orse complex and its invariants}, Adv.
  Soviet Math. \textbf{21} (1994), 93--115.

\bibitem{Betz1994}
M.~Betz and R.~L. Cohen, \emph{Graph moduli spaces and cohomology operations},
  Tr. J. of Mathematics \textbf{18} (1994), 23--41.

\bibitem{Chekanov2002a}
Y.~V. Chekanov, \emph{Differential algebra of {L}egendrian links}, Invent.
  Math. (2002), no.~150, 441--483.

\bibitem{Chekanov2002}
\bysame, \emph{Invariants of {L}egendrian knots}, Proceedings of the
  International Congress of Mathematicians (Beijing), vol.~II, Higher Ed.
  Press, 2002, pp.~385--394.

\bibitem{Chekanov2005}
Y.~V. Chekanov and P.~E. Pushkar, \emph{Combinatorics of fronts of {L}egendrian
  links and {A}rnol'd's 4-conjectures}, Uspekhi Mat. Nauk \textbf{60} (2005),
  99--154.

\bibitem{Civan2011}
G.~Civan, P.~Koprowski, J.~Etnyre, J.~Sabloff, and A.~Walker, \emph{Product
  structures for legendrian contact homology}, Math. Proc. Camb Phil. Soc.
  \textbf{150} (2011), no.~2, 291--311.

\bibitem{Ekholm2007}
T.~Ekholm, \emph{Morse flow trees and {L}egendrian contact homology in 1-ject
  spaces}, Geometry \& Topology \textbf{11} (2007), 1083--1224.

\bibitem{Eliashberg2000}
Ya. Eliashberg, A.~Givental, and H.~Hofer, \emph{Introduction to {S}ymplectic
  {F}ield {T}heory}, Geom. Funct. Analysis \textbf{10} (2000), no.~3, 560--673.

\bibitem{Etnyre2002}
J.~Etnyre, L.~Ng, and J.~Sabloff, \emph{Invariants of {L}egendrian knots and
  coherent orientations}, J. Symplectic Geom. \textbf{1} (2002), no.~2,
  321--367.

\bibitem{Fuchs2003}
D.~Fuchs, \emph{{C}hekanov-{E}liashberg invariant of {L}egendrian knots:
  existence of augmentations}, J. Geom. Phys. \textbf{47} (2003), no.~1,
  43--65.

\bibitem{Fuchs2004}
D.~Fuchs and T.~Ishkhanov, \emph{Invariants of {L}egendrian knots and
  decompositions of front diagrams}, Mosc. Math. J. \textbf{4} (2004), no.~3,
  707--717.

\bibitem{Fuchs2008}
D.~Fuchs and D.~Rutherford, \emph{Generating families and {L}egendrian contact
  homology in the standard contact space}, Journal of Topology \textbf{4}
  (2011), no.~1, 190--226.

\bibitem{Fukaya1997a}
K.~Fukaya, \emph{Morse homotopy and its quantization}, Geometric topology
  (Athens, GA, 1993), AMS/IP Stud. Adv. Math., vol. 2.1, Amer. Math. Soc.,
  1997, pp.~409--440.

\bibitem{Fukaya1997}
K.~Fukaya and Y-{G}. Oh, \emph{Zero-loop open strings in the cotangent bundle
  and {M}orse homotopy}, Asian J. of Math. \textbf{1} (1997), no.~1, 96--180.

\bibitem{Hatcher1973}
A.~Hatcher and J.~Wagoner, \emph{Pseudo-isotopies of compact manifolds},
  Asterisque, no.~6, Soci\'{e}t\'{e} Math\'{e}matique de France, 1973.

\bibitem{Henry2011}
M.~B. Henry, \emph{Connections between {F}loer-type invariants and {M}orse-type
  invariants of {L}egendrian knots}, Pacific Journal of Mathematics
  \textbf{249} (2011), no.~1, 77--133.

\bibitem{Jordan2006}
J.~Jordan and L.~Traynor, \emph{Generating family invariants for {L}egendrian
  links of unknots}, Alg. Geom. Top. \textbf{6} (2006), 895--933.

\bibitem{Laudenbach1992}
F.~Laudenbach, \emph{On the {T}hom-{S}male complex}, appendix to {J}-{M}
  {B}ismut and {W}. {Z}hang, An extension of a theorem by {C}heeger and
  {M}uller, Asterisque, vol. 205, 1992.

\bibitem{Ng2003}
L.~Ng, \emph{Computable {L}egendrian invariants}, Topology \textbf{42} (2003),
  55--82.

\bibitem{Ng2006}
L.~Ng and J.~Sabloff, \emph{The correspondence between augmentations and
  rulings for {L}egendrian knots}, Pacific J. Math. \textbf{224} (2006), no.~1,
  141--150.

\bibitem{L.Traynor2004}
L.~Ng and L.~Traynor, \emph{Legendrian solid-torus links}, Journal of
  Symplectic Geom. \textbf{2} (2004), no.~3, 411--443, arXiv:math.SG/0407068.

\bibitem{Sabloff2005}
J.~Sabloff, \emph{Augmentations and rulings of {L}egendrian knots}, Int. Math.
  Res. Not. \textbf{19} (2005), 1157--1180.

\bibitem{Traynor1997a}
L.~Traynor, \emph{Legendrian circular helix links}, Math. Proc. Camb. Phil.
  Soc. \textbf{122} (1997), 301--314.

\bibitem{Traynor2001}
\bysame, \emph{Generating function polynomials for {L}egendrian links}, Geom.
  Topol. \textbf{5} (2001), 719--760.

\end{thebibliography}

\end{document}